\documentclass[11pt]{amsart}
\usepackage[latin9]{inputenc}
\usepackage{geometry}
\usepackage{enumitem}
\usepackage{amstext}
\usepackage{amsthm}
\usepackage{amssymb}
\usepackage{stmaryrd}
\usepackage[all]{xy}
\usepackage[unicode=true,pdfusetitle,
 bookmarks=true,bookmarksnumbered=false,bookmarksopen=false,
 breaklinks=false,pdfborder={0 0 1},backref=false,colorlinks=false]
 {hyperref}

\makeatletter
\numberwithin{equation}{section}
\numberwithin{figure}{section}
\newlength{\lyxlabelwidth}      
\theoremstyle{remark}
\newtheorem*{rem*}{\protect\remarkname}
\theoremstyle{plain}
\newtheorem{thm}{\protect\theoremname}[section]
\theoremstyle{definition}
\newtheorem{defn}[thm]{\protect\definitionname}
\theoremstyle{plain}
\newtheorem{lem}[thm]{\protect\lemmaname}
\theoremstyle{plain}
\newtheorem{cor}[thm]{\protect\corollaryname}
	\newenvironment{elabeling}[2][]%
	{\settowidth{\lyxlabelwidth}{#2}
		\begin{description}[font=\normalfont,style=sameline,
			leftmargin=\lyxlabelwidth,#1]}
	{\end{description}}

\usepackage{amsthm}\usepackage{bm}
\usepackage{verbatim}
\usepackage{mathrsfs}

\makeatother

\providecommand{\corollaryname}{Corollary}
\providecommand{\definitionname}{Definition}
\providecommand{\lemmaname}{Lemma}
\providecommand{\remarkname}{Remark}
\providecommand{\theoremname}{Theorem}

\begin{document}
\global\long\def\R{\mathbb{R}}%

\global\long\def\C{\mathbb{C}}%

\global\long\def\a#1{a_{T_{#1}}}%

\global\long\def\S{\mathbb{S}}%

\global\long\def\Q{\mathbb{Q}}%

\global\long\def\Z{\mathbb{Z}}%

\global\long\def\N{\mathbb{N}}%

\global\long\def\Z{\mathbb{Z}}%

\global\long\def\D{\mathcal{D}}%

\global\long\def\U{\mathcal{U}}%

\global\long\def\mU{\mathbb{U}}%

\global\long\def\G{\mathbb{G}}%

\global\long\def\Gbar{\underline{\mathbb{\G}}}%

\global\long\def\z{\mathcal{Z}}%

\global\long\def\Y{\mathcal{Y}}%

\global\long\def\X{\mathcal{X}}%

\global\long\def\W{\mathcal{W}}%

\global\long\def\mult{\Phi}%

\global\long\def\GL{\text{GL}}%

\global\long\def\SL{\text{SL}}%

\global\long\def\SO{\text{SO}}%

\global\long\def\ASL{\text{ASL}}%

\global\long\def\AGL{\text{AGL}}%

\global\long\def\primlat{\mathcal{\Z}_{\text{prim}}^{k,d}}%

\global\long\def\primlatof#1{\mathcal{\Z}_{\text{prim}}^{#1,d}}%

\global\long\def\zprim{\mathcal{\Z}_{\text{prim}}^{d}}%

\global\long\def\gr{\text{Gr}}%

\global\long\def\Hbold{\mathbf{H}}%

\global\long\def\Lbold{\mathbf{L}}%

\global\long\def\lat{\text{lat}}%

\global\long\def\df{\overset{\text{def}}{=}}%

\global\long\def\red{\vartheta}%

\global\long\def\norm#1{\left\Vert #1\right\Vert }%

\global\long\def\L{\mathcal{L}}%

\global\long\def\H{\mathcal{H}}%

\global\long\def\V{\mathcal{Z}}%

\global\long\def\M{\mathbb{M}}%

\global\long\def\Spin{\text{Spin}}%

\global\long\def\ZS{\Z\left[S^{-1}\right]}%

\global\long\def\ef{\varphi}%

\global\long\def\cov{\text{c}}%

\global\long\def\o{\theta}%

\global\long\def\ovec{\tau}%

\global\long\def\Rprim{R_{\text{prim}}^{d}}%

\global\long\def\Hprim#1{\H_{#1}^{\text{prim}}}%

\global\long\def\gr{\text{Gr}(d-1,d)}%

\global\long\def\covol{\text{covol}}%

\global\long\def\shape{\text{shape}}%

\title{Linnik's problem in fiber bundles over quadratic homogeneous varieties}
\author{Michael Bersudsky and Uri Shapira}
\begin{abstract}
We compute the statistics of $\SL_{d}(\Z)$ matrices lying on level
sets of an integral polynomial defined on $\SL_{d}(\R),$ a result
that is a variant of the well known theorem proved by Linnik about
the equidistribution of radially projected integral vectors from a
large sphere into the unit sphere.

Using the above result we generalize the work of Aka, Einsiedler and
Shapira in various directions. For example, we compute the joint distribution
of the residue classes modulo $q$ and the properly normalized orthogonal
lattices of primitive integral vectors lying on the level set $-(x_{1}^{2}+x_{2}^{2}+x_{3}^{2})+x_{4}^{2}=N$
as $N\to\infty$, where the normalized orthogonal lattices sit in
a submanifold of the moduli space of rank-$3$ discrete subgroups
of $\R^{4}$.

\end{abstract}

\maketitle

\section{Introduction}

\thispagestyle{empty}

\let\thefootnote\relax\footnotetext{This project has received funding from the European Research Council  (ERC) under the European Union's Horizon 2020 research and innovation programme (grant agreement No. 754475). The authors also acknowledge the support of ISF grant 871/17.}  

\subsection{\label{subsec:Linnik-type-problems}Linnik type problems}

To put our work in historical context, we will now recall a well known
work of Linnik and its generalizations.

Consider for an integral homogeneous polynomial $P:\R^{d}\to\R$ and
for $m\in\Z$ the level set
\[
\H_{m}(P,\R)\df P^{-1}(\{m\})=\left\{ \mathbf{v}\in\R^{d}\mid P(\mathbf{v})=m\right\} ,
\]
and let 
\[
\H_{m,\text{prim}}(P,\Z)\df\H_{m}(f,\R)\cap\zprim=\left\{ \mathbf{v}\in\zprim\mid P(\mathbf{v})=m\right\} ,
\]
 where $\Z_{\text{prim}}^{d}$ denotes the set of primitive integral
vectors in $\R^{d}$.

Assuming that the cardinalities of $\H_{m_{i},\text{prim}}(P,\Z)$
diverge to infinity along a sequence $\left\{ m_{i}\right\} _{i=1}^{\infty}\subseteq\N$,
it is natural to study the limiting statistics of $\H_{m_{i}}(P,\Z)$
when projected radially into $\H_{1}(P,\Z)$.

Linnik appears to be the first to consider the above problem in his
seminal work (see \cite{Lin68}) by computing the weak-{*} limits
of the uniform probability measures $\mu_{m}$ on the unit sphere
supported on $\frac{1}{\sqrt{m}}\H_{m,\text{prim}}(x^{2}+y^{2}+z^{2},\Z)$
as $m\to\infty$. Under suitable congruence conditions, Linnik was
able to prove that $\mu_{m}$ converges towards the natural measure
on $\S^{2}$ by developing a method known today as \emph{Linnik's
Ergodic method}, which has an arithmetic-dynamical nature.

Following Linnik's original work, the above problem was studied further
by Linnik and his collaborators, see \cite{Malyshev} for a review,
and more recently by a variety of other authors employing dynamical
or harmonic analysis tools, see for example the definitely not exhaustive
list \cite{elenberg_michel_venaktesh,Gan_Oh_Link,Michel-Venk-L.fnc_and_erg_Lin,Eskin_Oh_rep_of_intgrs,benoist-oh_effecting_equiv}.

\subsubsection{\label{subsec:Linnik-type-problem in sld}Linnik type problem in
$\protect\SL_{d}$}

The main results of our paper (see Theorems \ref{thm:main thm for Z}
and \ref{thm:main_thm_with_congruences-forZ}), concern a problem
which falls into a broader category of Linnik type problems in an
ambient manifold \emph{that is not necessarily the Euclidean space}.

More explicitly, we will replace Euclidean space with $\SL_{d}(\R)$
and primitive integral vectors with $\SL_{d}(\Z)$. We will consider
an integral polynomial $P:\SL_{d}(\R)\to\R$ such that its level sets
$\V_{T}(\R)=P^{-1}(\{T\})$ have a transitive action of a fixed group
$G\leq\SL_{d}(\R)\times\SL_{d}(\R)$ and such that there exists a
$G$-equivariant projection $\pi_{T}:\V_{T}(\R)\to\V_{T_{0}}(\R)$,
where $\V_{T_{0}}(\R)$ is a chosen reference level set. Then, similarly
to the Linnik type problems above, we will consider (properly) normalized
counting measure supported on $\V_{T_{0}}(\R)$ of the form $\mu_{T}\df\frac{1}{c(T)}\sum_{x\in\V_{T}(\Z)}\delta_{\pi_{T}(x)}$,
where $\V_{T}(\Z)\df\V_{T}(\R)\cap\SL_{d}(\Z)$, which are infinite,
locally finite, atomic measures.

Our main result will state that, under certain conditions on the range
of $T$, 
\[
\lim_{T\to\infty}\mu_{T}(f)=\mu_{\V}(f),\ \forall f\in C_{c}(\V_{T_{0}}(\R)),
\]
where $\mu_{\V}$ will be a measure on $\V_{T_{0}}(\R)$ induced by
the $G$-action.

\subsection{\label{subsec:Overview-and-motivation}On the work of Aka, Einsiedler
and Shapira}

Our original motivation for this paper comes from the work of Aka,
Einsiedler and Shapira that can be found in \cite{AES3d} and \cite{AESgrids}.
We will extend \cite{AESgrids} in various directions using the limiting
distribution of the measures $\mu_{T}$ discussed in Section \ref{subsec:Linnik-type-problem in sld}
(see Theorems \ref{thm:moduli main thm } and \ref{thm:moduli_main_thm_with_congruences}).
\begin{rem*}
This paper relies on the method of the proof of \cite{AESgrids},
and since analogue problems in dimension $d=3$ are treated by different
set of tools (see e.g. \cite{AES3d,Kuga_sato_khayutin}), the case
of dimension $d=3$ is not treated in this paper.
\end{rem*}
We will now recall the main results of \cite{AESgrids}. Fix $d\geq4$
and consider $X_{d-1}$ the space of $(d-1)$-unimodular lattices
in $\R^{d-1}$. The space of \emph{shapes} of $(d-1)$-lattices is
given by
\[
\mathcal{S}_{d-1}\df\SO_{d-1}(\R)\backslash X_{d-1}\cong\SO_{d-1}(\R)\backslash\SL_{d-1}(\R)/\SL_{d-1}(\Z)
\]
which is simply the space of full rank lattices in $\R^{d-1}$ identified
up-to a rotation.

For $\mathbf{v}\in\R^{d}$, we denote by $\mathbf{v}^{\perp}$ the
orthogonal hyperplane to $\mathbf{v}$ with respect to the usual Euclidean
inner product, and for $\mathbf{v}\in\zprim$ we define
\[
\Lambda_{\mathbf{v}}\df\mathbf{v}^{\perp}\cap\Z^{d},
\]
which is a rank $(d-1)$-discrete subgroup of $\R^{d}$.

We embed $\mathcal{S}_{d-1}$ into the space of rank $(d-1)$-discrete
subgroups of $\R^{d}$ by identifying the horizontal plane $\R^{d-1}\times\{0\}\subseteq\R^{d}$
with $\R^{d-1}$. Then, by scaling the $\Lambda_{\mathbf{v}}$'s into
unimodular lattices and by rotating them into $\R^{d-1}\times\{0\}$,
we obtain their ``shape'' in $\mathcal{S}_{d-1}$. More explicitly,
for a rank $(d-1)$-discrete subgroup $\Lambda\leq\R^{d}$, we denote
by $\covol(\Lambda)$ the volume of a fundamental domain of $\Lambda$
in the hyperplane containing $\Lambda$ with respect the volume form
obtained by the restriction of the Euclidean inner product to this
hyperplane. An elementary argument (see e.g. \cite{AESgrids}) shows
that 
\[
\covol(\Lambda_{\mathbf{v}})=\sqrt{\sum_{i=1}^{d}v_{i}^{2}}\df\norm{\mathbf{v}},\ \forall\mathbf{v}\in\zprim.
\]
By choosing $\rho_{\mathbf{v}}\in\SO_{d}(\R)$ such that $\rho_{\mathbf{v}}\mathbf{v}=\mathbf{e}_{d}$,
we get that $\rho_{\mathbf{v}}(\mathbf{\norm{\mathbf{v}}}^{-1/d-1}\Lambda_{\mathbf{v}})$
is a unimodular lattice in $\R^{d-1}\cong\R^{d-1}\times\{0\}$. We
denote by $K\cong\SO_{d-1}(\R)$ the subgroup of $\SO_{d}(\R)$ stabilizing
$\mathbf{e}_{d},$ and we define $\shape(\Lambda_{\mathbf{v}})\in\mathcal{S}_{d-1}$
by
\[
\shape(\Lambda_{\mathbf{v}})\df K\rho_{\mathbf{v}}(\mathbf{\norm{\mathbf{v}}}^{-1/d-1}\Lambda_{\mathbf{v}}),
\]
which is well defined as a function of $\mathbf{v}\in\zprim$ (see
\eqref{eq:definition of shape in X(R)} which extends the definition
of ``shape'' function to the moduli space of $(d-1)-$discrete subgroups
of $\R^{d}$).

The main result of \cite{AES3d} and \cite{AESgrids} was the joint
equidistribution of the normalized probability counting measures supported
on 
\[
\left\{ \left(\shape(\Lambda_{\mathbf{v}}),\frac{1}{\sqrt{T}}\mathbf{v}\right)\mid\mathbf{v}\in\H_{\text{prim},T}(\Z)\right\} \subseteq\mathcal{S}_{d-1}\times\S^{d-1},
\]
where $\S^{d-1}\subseteq\R^{d}$ denotes the unit sphere.

\subsubsection{Some historical context for \cite{AES3d} and \cite{AESgrids} and
subsequent works}

Statistics of shapes of subgroups of $\Z^{d}$ were studied by W.
Roelcke in \cite{Roelcke}, H. Maass in \cite{Maass:1959aa}, and
much later by W. Schmidt in \cite{Schmidt1998,Schmidt2015} who proved
more general results using elementary counting techniques. Schmidt's
theorem was given a dynamical approach in\textbf{ }\cite{Marklof:2010aa},
and T. Horesh and Y. Karasik recently in \cite{Karsik_horesh} extended
Schmidt's results to ``higher'' moduli spaces using the technique
of \cite{GorodnikNevo+2012+127+176}.

A considerably more refined problem concerning the shapes of subgroups
of $\Z^{d}$ lying in sparse subsets was first studied in \cite{einsiedler_mozes_shah_shapira_2016}
and then in \cite{AES3d} and \cite{AESgrids}. We note the recent
works \cite{Effective_aes_Ruhr,Kuga_sato_khayutin,Aka_eins_wiesser,Blomer_brumly,Aka_Musso_Wieser}
which extend and refine \cite{AES3d,AESgrids,einsiedler_mozes_shah_shapira_2016}
in a various directions.

In this paper we continue the preceding line of research and generalize
the results of \cite{AESgrids}. In a rough description, we will consider
tuples of the form $\left(\shape(\Lambda_{\mathbf{v}}),\mathbf{v},\mathbf{v}\text{ mod \ensuremath{q}}\right)$
for integral $\mathbf{v}\in\zprim\cap Q^{-1}(\{T\})$ where $Q$ is
a non-singular integral quadratic form which can be either positive
definite, or of signature $(1,d-1)$, and moreover, we will consider
``higher'' moduli spaces.

\subsubsection{AES type result in two sheeted hyperboloids}

We now give a special case of our results. We fix $d\geq4$, we let
$Q(\mathbf{x})=-(\sum_{i=1}^{d-1}x_{i}^{2})+x_{d}^{2}$ and we consider
the group $\SO_{Q}(\R)\leq\SL_{d}(\R)$ which preserves $Q$. For
$T\in\R$, we denote 
\[
\H_{T}(\R)\df\left\{ \mathbf{x}\in\R^{d}\mid Q(\mathbf{x})=T\right\} ,
\]
and we let

\[
\H_{T,\text{prim}}(\Z)\df\H_{T}(\R)\cap\zprim.
\]
In this paper we will concentrate on $T>0$ because the stabilizers
in $\SO_{Q}(\R)$ of vectors in $\H_{T}(\R)$ are compact, which is
important for the method that we use. We recall by Theorem 6.9 of
\cite{bhc} that $\H_{T,\text{prim}}(\Z)/\SO_{Q}(\Z)$ is finite,
and for $N\in\N$ we consider the following measure on $\mathcal{S}_{d-1}\times\H_{1}(\R)$
defined by 
\[
\nu_{N}\df\frac{1}{\left|\H_{N,\text{prim}}(\Z)/\SO_{Q}(\Z)\right|}\sum_{\mathbf{v}\in\H_{N,\text{prim}}(\Z)}\delta_{\left(\text{shape}(\Lambda_{\mathbf{v}}),\frac{1}{\sqrt{N}}\mathbf{v}\right)}.
\]
Note that $\mathcal{S}_{d-1}\times\H_{1}(\R)$ is a quotient of $\SL_{d-1}(\R)/\SL_{d-1}(\Z)\times\SO_{Q}(\R)$
by a compact group, and on the former space there is a choice of a
natural measure (for more details, see Section \ref{subsec:Measures-as measures on fibre bundles}),
which gives, by taking the pushforward under the natural projection,
a product measure on $\mathcal{S}_{d-1}\times\H_{1}(\R)$ which we
denote by $\mu_{\mathcal{S}_{d-1}}\otimes\mu_{\H_{1}}$.
\begin{thm}
\label{thm:main thm-particular case} For all $f\in C_{c}\left(\mathcal{S}_{d-1}\times\H_{1}(\R)\right)$
it holds that 
\[
\lim_{N\to\infty}\nu_{N}(f)=\mu_{\mathcal{S}_{d-1}}\otimes\mu_{\H_{1}}(f).
\]
\end{thm}

\textbf{}
By adding congruence assumptions on $N\in\N$, we obtain the following
joint distribution of the radial projection into $\H_{1}(\R)$, the
shapes of orthogonal lattices and the residue classes of the vectors
in $\H_{N,\text{prim}}(\Z)$ as $N\to\infty$.

We choose $q\in\N$ and we define for $a\in\Z/(q)$
\[
\H_{a}(\Z/(q))\df\left\{ \mathbf{x}\in\left(\Z/(q)\right)^{d}\mid Q(\mathbf{x})=a\right\} .
\]

For $N\in\N$ and $q\in\N$ we consider the following measures on
$\mathcal{S}_{d-1}\times\H_{1}(\R)\times\H_{N(\text{mod}q)}(\Z/(q))$
defined by 
\[
\nu_{N}^{q}\df\frac{1}{\left|\H_{N,\text{prim}}(\Z)/\SO_{Q}(\Z)\right|}\sum_{\mathbf{v}\in\H_{N,\text{prim}}(\Z)}\delta_{\left(\text{shape}(\Lambda_{\mathbf{v}}),\frac{1}{\sqrt{N}}\mathbf{v},\mathbf{v}\text{ (mod }q)\right)}
\]

\begin{thm}
\label{thm:main_thm_particular case with congruences}Let $q\in2\N+1$
and let $a\in\left(\Z/(q)\right)^{\times}$ be an invertible residue
mod $q$. Assume that $\left\{ T_{n}\right\} _{n=1}^{\infty}\subseteq\N$
satisfy
\[
T_{n}\text{ (mod }q)=a\in\left(\Z/(q)\right)^{\times},\ \forall n\in\N.
\]
Then for all $f\in C_{c}\left(\mathcal{S}_{d-1}\times\H_{1}(\R)\times\H_{a}(\Z/(q))\right)$
it holds that 
\[
\lim_{n\to\infty}\nu_{T_{n}}^{q}(f)=\mu_{\mathcal{S}_{d-1}}\otimes\mu_{\H_{1}}\otimes\mu_{\H_{a}(\Z/(q))}(f),
\]
where $\mu_{\H_{a}(\Z/(q))}$ is the uniform probability measure on
$\H_{a}(\Z/(q))$.
\end{thm}

\subsection{Structure of the paper}
\begin{itemize}
\item Section \ref{sec:Some-conventions,-standing} discusses some conventions,
standing assumptions and basic facts that will be used throughout
the paper.
\item Section \ref{sec:On-the-deity} discusses the manifolds $\V_{T}(\R)\subseteq\SL_{d}(\R)$
and presents our main ``Linnik type'' results, see Theorems \ref{thm:main thm for Z}
and \ref{thm:main_thm_with_congruences-forZ}.
\item Section \ref{sec:moduli-spaces} discusses moduli spaces of discrete
subgroups of $\Z^{d}$ and states our results concerning them refining
\cite{AESgrids}, see Theorems \ref{thm:moduli main thm } and \ref{thm:moduli_main_thm_with_congruences}.
We note that the latter results may also be interpreted conceptually
as a Linnik type result.
\item Section \ref{sec:The-results-forZ imply Y} proves Theorems \ref{thm:moduli main thm }
and \ref{thm:moduli_main_thm_with_congruences} of the moduli spaces
using Theorems \ref{thm:main thm for Z} and \ref{thm:main_thm_with_congruences-forZ}
of the $\SL_{d}(\R)$-submanifolds.
\item The rest of the paper is devoted to proving Theorems \ref{thm:main thm for Z}
and \ref{thm:main_thm_with_congruences-forZ}. The scheme is roughly
as follows:
\begin{itemize}
\item Section \ref{sec:A-revisit-to-s-arith-thm-aes} generalizes the proof
of \cite[Theorem 3.1]{AESgrids} concerning the equidistribution of
a sequence of compact orbits in an $S$-arithemetic space, which builds
on the results of \cite{GO}.
\item Sections \ref{sec:Equivalence-classes-of integral points}-\ref{sec:Proof-of-theorems}
exploit the equidistribution of orbits proved in Section \ref{sec:A-revisit-to-s-arith-thm-aes}
to prove Theorems \ref{thm:main thm for Z} and \ref{thm:main_thm_with_congruences-forZ}
by revisiting the method of \cite{AESgrids}. The preceding method
is outlined in Section \ref{sec:Equivalence-classes-of integral points}.
\end{itemize}
\end{itemize}

\subsubsection*{Acknowledgements}

We would like to thank Andreas Wieser and Yakov Karasik for helpful
discussions, and we would like to thank Daniel Goldberg for his comments
and suggestions on the manuscript.

\section{\label{sec:Some-conventions,-standing}Some conventions, standing
assumptions and basic facts}

We denote by $R$ a unital commutative ring, and for $d\in\N$ we
view $R^{d}$ as column vectors. We will denote for $1\leq i\leq d$
by $\mathbf{e}_{i}\in R^{d}$ the standard basis vectors, and for
$\mathbf{x}\in R^{d}$ we denote by $x_{1},...,x_{d}\in R$ the components
of $\mathbf{x}$, namely $\mathbf{x}=\sum_{i=1}^{d}x_{i}\mathbf{e}_{i}$,
where $x_{i}\in R$.

When $\mathcal{V}(\Z)\subseteq\text{\ensuremath{\Z}}^{d}$ is defined
by the solutions of a collection of polynomials with integer coefficients,
we denote by $\mathcal{V}(R)$ its solutions in $R^{d}$. For $q\in\N$
we denote by $\vartheta_{q}:\Z\to\Z/(q)$ the reduction map modulo
$q$, and we observe that it induces a map $\vartheta_{q}:\mathcal{V}(\Z)\to\mathcal{V}(\Z/(q))$.

Throughout the paper we will consider 
\[
\SL_{d}(R)\df\left\{ g\in M_{d}(R)\mid\det(g)=1\right\} ,
\]
\[
\ASL_{d-1}(R)\df\left\{ \left(\begin{array}{cc}
m & \mathbf{v}\\
\mathbf{0} & 1
\end{array}\right)\mid m\in\SL_{d-1}(R),\ \mathbf{v}\in R^{d-1}\right\} ,
\]
and for an integral symmetric matrix $M\in M_{d}(\Z)$ we let 
\[
\SO_{Q}(R)\df\left\{ g\in\SL_{d}(R)\mid g^{t}Mg=M\right\} ,
\]
where $Q$ is the quadratic form whose companion symmetric matrix
is $M$. We make the convention that a quadratic form $Q:\Z^{d}\to\Z$
is integral if $Q$ has an \emph{integral }companion matrix $M$,
and we say that $Q$ is non-degenerate if $\text{disc}(Q)\df\det(M)\neq0$.

We consider the right $\SO_{Q}(\R)$ linear action on $R^{d}$ given
by 
\begin{equation}
\mathbf{v}\cdot\rho\df\rho^{-1}\mathbf{v},\ \rho\in\SO_{Q}(R),\ \mathbf{v}\in R^{d},\label{eq:action of SO_Q on R^d}
\end{equation}
 and for $\mathbf{v}\in R^{d}$ we let 
\[
\mathbf{H}_{\mathbf{v}}(R)\df\left\{ g\in\SO_{Q}(R)\mid g\mathbf{v}=\mathbf{v}\right\} .
\]

\medskip{}

\subsection*{Standing Assumption\label{subsec:Standing-Assumption}}

\emph{Throughout the paper $Q$ denotes an integral, non-degenerate
quadratic form in $d\geq4$ variables such that $Q(\mathbf{e}_{d})>0$
and $\mathbf{H}_{\mathbf{e}_{d}}(\R)$ is compact.}

\medskip{}

\begin{defn}
\label{def:non-singularity modq}For $q\in\N$ we will say that $Q$
is non-singular modulo $q$ if $\text{disc}(Q)\text{(mod \ensuremath{q})}\in\left(\Z/(q)\right)^{\times}$.
\end{defn}

\subsubsection{Linear action of $\protect\SL_{d}$ by the Cartan involution}

Let $\o:\SL_{d}(R)\to\SL_{d}(R)$ be the involutive automorphism given
by 
\[
\o(g)\df\left(g^{t}\right)^{-1}.
\]
In the paper we will consider the left action of $\SL_{d}(R)$ on
$R^{d}$ given by 
\begin{equation}
g\cdot\mathbf{v}\df\theta(g)\mathbf{v},\ g\in\SL_{d}(R)\label{eq:sl_d action by involution}
\end{equation}
(where the right hand-side denotes matrix multiplication of $\mathbf{v}$
by $\theta(g)$) and we denote the translation map of $\mathbf{e}_{d}$
by
\begin{equation}
\ovec(g)\df\o(g)\mathbf{e}_{d},\ g\in\SL_{d}(R).\label{eq:def of tau}
\end{equation}
The main motivation that led us to consider the action above (and
not the usual left $\SL_{d}$ linear action) is that the vector $\ovec(g)\in R^{d}$
is orthogonal to the first $d-1$ columns of $g$ with respect to
the Euclidean inner product, as we now explain.

For $\mathbf{x},\mathbf{y}\in R^{d}$ we define the Euclidean bi-linear
form $\left\langle \mathbf{x},\mathbf{y}\right\rangle \df\sum_{i=1}^{d}x_{i}y_{i}$.
An important property of $\o$ is the invariance
\begin{equation}
\left\langle \o(g)\cdot\mathbf{x},g\cdot\mathbf{y}\right\rangle =\left\langle \mathbf{x},\mathbf{y}\right\rangle ,\ \forall g\in\SL_{d}(R),\label{eq:cartan ortho invariance}
\end{equation}
which in particular implies
\begin{equation}
\left\langle \ovec(g),g\mathbf{e}_{j}\right\rangle =\left\langle \mathbf{e}_{i},\mathbf{e}_{j}\right\rangle =\delta_{i,j}.\label{eq:orthogonality of theta}
\end{equation}

\subsubsection{Concerning the covolume and the left action of $\protect\SL_{d}(R)$
on $R^{d}$}

For a discrete subgroup $\Lambda\leq\R^{d}$ of rank $d-1$, we define
$\covol(\Lambda)$ to be the volume of a fundamental domain of $\Lambda$
in the hyperplane containing $\Lambda$, with respect to the volume
form obtained by the restriction of the Euclidean inner product to
this hyperplane.

For $g\in\SL_{d}(\R)$ we denote by $\hat{g}\in M_{d\times d-1}(\R)$
the matrix formed by the first $d-1$ columns of $g$, and we note
for $\Lambda=\hat{g}\Z^{d-1}$ (the discrete subgroup of rank $d-1$
having the columns of $\hat{g}$ as a $\Z$-basis) the formula 
\[
\covol(\Lambda)^{2}=\det\left(\hat{g}^{t}\hat{g}\right).
\]

\begin{lem}
\label{lem:the covolume of lambda_v and the map theta} We have for
\emph{$g\in\SL_{d}(\R)$} that \emph{
\[
\covol(\hat{g}\Z^{d-1})=\norm{\ovec(g)},
\]
}where $\left\Vert \cdot\right\Vert $ denotes the usual Euclidean
norm.
\end{lem}

\begin{proof}
We first note that 
\[
\left(g^{t}g\right)^{-1}=\text{adj}(g^{t}g),
\]
 where $\text{adj}\left(\cdot\right)$ denotes the matrix adjugate,
and we observe that the $d,d$ entry of the matrix $\text{adj}(g^{t}g)$
is $\det\left(\hat{g}^{t}\hat{g}\right)=\covol(\Lambda)^{2}$. In
particular, the $d,d$ entry of the matrix $\text{adj}(g^{t}g)$ can
be expressed by $\left\langle \mathbf{e}_{d},\left(g^{t}g\right)^{-1}\mathbf{e}_{d}\right\rangle $,
hence

\begin{align*}
\covol(\hat{g}\Z^{d-1})^{2}= & \left\langle \mathbf{e}_{d},\left(g^{t}g\right)^{-1}\mathbf{e}_{d}\right\rangle =\left\langle \mathbf{e}_{d},g^{-1}\o(g)\mathbf{e}_{d}\right\rangle \\
= & \left\langle \o(g)\mathbf{e}_{d},\o(g)\mathbf{e}_{d}\right\rangle =\norm{\o(g)\mathbf{e}_{d}}^{2}\\
= & \left\Vert \ovec(g)\right\Vert ^{2}.
\end{align*}
\end{proof}

\section{\label{sec:On-the-deity}Linnik type problem in $\protect\SL_{d}(\protect\R)$}

The structure of this section is as follows:
\begin{itemize}
\item Section \ref{subsec:Subvarieties-of SL_d} introduces the one-parameter
family of subvarieties $\V_{T}\subseteq\SL_{d}$ mentioned in Section
\ref{subsec:Linnik-type-problem in sld}, and discusses basic facts
concerning them.
\item Section \ref{subsec:The-equivariant-isomorphism} defines a natural
homeomorphism between a subvariety $\V_{T}(\R)$ to a reference subvariety
$\V_{Q(\mathbf{e}_{d})}(\R)$.
\item Section \ref{subsec:Statistics-of-Z} presents our main results (Theorems
\ref{thm:main thm for Z} and \ref{thm:main_thm_with_congruences-forZ}).
\end{itemize}

\subsection{\label{subsec:Subvarieties-of SL_d}Subvarieties of $\protect\SL_{d}$}

Let $Q$ be a quadratic form as in our \nameref{subsec:Standing-Assumption}
and recall that $R$ denotes a unital commutative ring. For $T\in R$,
we let
\[
\H_{T}(R)\df\{\mathbf{v}\in R^{d}\mid Q(\mathbf{v})=T\},
\]
and we consider 
\begin{equation}
\V_{T}(R)\df\tau^{-1}(\H_{T}(R)),\label{eq:def of z_T}
\end{equation}
(see \eqref{eq:def of tau} to recall $\ovec$) namely 
\[
\V_{T}(R)\df\{g\in\SL_{d}(R)\mid(Q\circ\tau)(g)=Q\left((g^{t})^{-1}\mathbf{e}_{d}\right)=T\}.
\]
Note that $Q\circ\ovec:\SL_{d}(\Z)\to\Z$ is an integral polynomial.

\subsubsection{Concerning the \emph{$(\protect\SO_{Q}\times\protect\ASL_{d-1})$}
action}

We recall the $\SL_{d}(R)$ action given in \eqref{eq:sl_d action by involution},
and we observe that the stabilizer subgroup of $\SL_{d}(R)$ stabilizing
$\mathbf{e}_{d}$ is $\ASL_{d-1}(R)$, which allows us to conclude
\begin{equation}
\SL_{d}(R)/\ASL_{d-1}(R)\cong\tau\left(\SL_{d}(R)\right).\label{eq:iden_R^d_with_SL/ASL}
\end{equation}
In light of \eqref{eq:iden_R^d_with_SL/ASL}, $\SL_{d}(R)$ can be
thought of as a union of fibers $\ovec^{-1}(\mathbf{v}),\ \mathbf{v}\in R^{d},$
where each fiber is an $\ASL_{d-1}(R)$-right coset, and according
to \eqref{eq:def of z_T}, $\V_{T}(R)$ is the union of those fibers
of vectors in $\tau\left(\SL_{d}(R)\right)\cap\H_{T}(R),$ which leads
to the identification
\begin{equation}
\V_{T}(R)/\ASL_{d-1}(R)\cong\tau\left(\SL_{d}(R)\right)\cap\H_{T}(R).\label{eq:space_of_ASL_orbits_in_V_x}
\end{equation}
We consider the following right action of $\SO_{Q}(R)$ on $\SL_{d}(R)/\ASL_{d-1}(R)$
defined by
\begin{equation}
\left(g\ASL_{d-1}(R)\right)\cdot\rho\df\o\left(\rho\right)^{-1}g\ASL_{d-1}(R),\label{eq:action of SO_Q on SL_d/ASL_d-1}
\end{equation}
and we observe that the above action is equivalent to the right $\SO_{Q}(R)$
action \eqref{eq:action of SO_Q on R^d} on the orbit $\tau\left(\SL_{d}(R)\right)\subseteq R^{d}$,
namely 
\begin{equation}
\ovec(g\ASL_{d-1}(R)\cdot\rho)=\ovec(g)\cdot\rho.\label{eq:equivariance of the so_Q action w.r.t. tau}
\end{equation}

In view of \eqref{eq:equivariance of the so_Q action w.r.t. tau},
it is natural to consider the $\left(\SO_{Q}\times\ASL_{d-1}\right)(R)$
action on $\V_{T}(R)$ from the right by
\begin{equation}
g\cdot(\rho,\eta)\df\o(\rho)^{-1}g\eta,\ g\in\V_{T}(R),\ (\eta,\rho)\in(\SO_{Q}\times\ASL_{d-1})(R),\label{eq:definion of action of =00005Cmatbb=00007BG=00007D on Z_t}
\end{equation}
and continuing with our description of $\V_{T}(R)$ as a union of
fibers of vectors in $\tau(\SL_{d}(R))\cap\H_{T}(R)$, we interpret
$g\eta$ as a ``move'' in the fiber of $\mathbf{v}\df\ovec(g)$
and by $\o(\rho)^{-1}g\eta$ as a ``transition'' of $g\eta$ into
the fiber of $\rho^{-1}\mathbf{v}$ (using \eqref{eq:equivariance of the so_Q action w.r.t. tau}),
which allows us to conclude (more formally, by \eqref{eq:space_of_ASL_orbits_in_V_x}
and \eqref{eq:equivariance of the so_Q action w.r.t. tau}) that

\begin{equation}
\V_{T}(R)/\left(\SO_{Q}\times\ASL_{d-1}\right)(R)\cong\left(\tau\left(\SL_{d}(R)\right)\cap\H_{T}(R)\right)/\SO_{Q}(R).\label{eq:space_of_ASLxSO_orbits_in_V_x}
\end{equation}

We have the following corollary from \eqref{eq:space_of_ASLxSO_orbits_in_V_x}.

\begin{cor}
\label{cor:transitivity of ASL_times_SO_d} The following hold:
\begin{enumerate}
\item \label{enu:Real_transitively}\emph{$(\SO_{Q}\times\ASL_{d-1})(\R)$
}acts transitively on $\V_{T}(\R)$ for all $T>0$.
\item \label{enu:mod_q_transitively} Let $q\in2\N+1$, and assume that
$Q$ is non-singular modulo $q$ \emph{(}Definition \ref{def:non-singularity modq}\emph{)}.
Then\emph{ $\left(\SO_{Q}\times\ASL_{d-1}\right)(\Z/(q))$ }acts transitively
on $\V_{a}(\Z/(q))$ for all $a\in\left(\Z/(q)\right)^{\times}$.
\item \label{enu:finitely many orbits in V_x(Z)} There are finitely many
\emph{$\left(\SO_{Q}\times\ASL_{d-1}\right)(\Z)$} orbits in $\V_{N}(\Z)$
for all $N\in\N$, and moreover \emph{
\[
\left|\V_{N}(\Z)/\left(\SO_{Q}\times\ASL_{d-1}\right)(\Z)\right|=\left|\H_{N,\text{prim}}(\Z)/\SO_{Q}(\Z)\right|
\]
}
\end{enumerate}
\end{cor}

\begin{proof}
To show \eqref{enu:Real_transitively} and \eqref{enu:mod_q_transitively},
we observe that using \eqref{eq:space_of_ASLxSO_orbits_in_V_x}, it
is sufficient to prove that $\SO_{Q}(R)$ acts transitively on $\H_{T}(R)$
when $R\in\{\R,\Z/(q)\}$ and $T\in R$ are as specified in \eqref{enu:Real_transitively}
and \eqref{enu:mod_q_transitively}

The claim for $R=\R$ follows from Witt's Theorem.

We now proceed to prove \eqref{enu:mod_q_transitively} by going along
the lines of the proof of \cite[Chapter 8, Lemma 3.3]{CASSELS_quad}.
Let $p$ be an odd prime and let $k\in\N$. We may consider the following
involution (a generalized reflection) 
\[
\ovec_{\mathbf{v}}:\left(\Z/(p^{k})\right)^{d}\to\left(\Z/(p^{k})\right)^{d},
\]
defined for $\mathbf{v}\in\left(\Z/(p^{k})\right)^{d}$ such that
$Q(\mathbf{v})\in\left(\Z/(p^{k})\right)^{\times}$, by
\[
\ovec_{\mathbf{v}}(\mathbf{x})\df\mathbf{x}-\frac{2Q(\mathbf{x},\mathbf{v})}{Q(\mathbf{v})}\mathbf{v},
\]
where $Q(\mathbf{x},\mathbf{y})\df\frac{1}{4}\left(Q(\mathbf{x}+\mathbf{y})-Q(\mathbf{x}-\mathbf{y})\right)$
for $\mathbf{x},\mathbf{y}\in\left(\Z/(q)\right)^{d}$ is the associated
bi-linear form of $Q$. By observing that $Q(\ovec_{\mathbf{v}}(\mathbf{x}))=Q(\mathbf{x})$
and $\det(\ovec_{\mathbf{v}})=-1$,  we deduce that $\ovec_{\mathbf{u}_{1}}\circ\ovec_{\mathbf{u}_{2}}\in\SO_{Q}(\Z/(p^{k}))$
for all $\mathbf{u}_{1},\mathbf{u}_{2}\in\left(\Z/(p^{k})\right)^{d}$
such that $Q(\mathbf{u}_{1}),Q(\mathbf{u}_{2})\in\left(\Z/(p^{k})\right)^{\times}$.

We now show that for all $\mathbf{v}_{1},\mathbf{v}_{2}\in\H_{a}(\Z/(p^{k}))$
with $a\in\left(\Z/(p^{k})\right)^{\times}$ there exist $\mathbf{u}_{1},\mathbf{u}_{2}\in\left(\Z/(p^{k})\right)^{d}$
such that $Q(\mathbf{u}_{1}),Q(\mathbf{u}_{2})\in\left(\Z/(p^{k})\right)^{\times}$
and $\ovec_{\mathbf{u}_{1}}\circ\ovec_{\mathbf{u}_{2}}(\mathbf{v}_{1})=\mathbf{v}_{2}$.
Let $\mathbf{v}_{1},\mathbf{v}_{2}\in\H_{a}(\Z/(p^{k}))$ with $a\in\left(\Z/(p^{k})\right)^{\times}$.
We observe that $Q(\mathbf{v}_{1}+\mathbf{v}_{2})+Q(\mathbf{v}_{1}-\mathbf{v}_{2})=4a\in\left(\Z/(p^{k})\right)^{\times}$,
which implies that either $Q(\mathbf{v}_{1}+\mathbf{v}_{2})\in\left(\Z/(p^{k})\right)^{\times}$
or $Q(\mathbf{v}_{1}-\mathbf{v}_{2})\in\left(\Z/(p^{k})\right)^{\times}.$
Assuming that $Q(\mathbf{v}_{1}-\mathbf{v}_{2})\in\left(\Z/(p^{k})\right)^{\times}$,
we may consider $\ovec_{\mathbf{v}_{1}-\mathbf{v}_{2}}$ and we observe
that $\ovec_{\mathbf{v}_{1}-\mathbf{v}_{2}}\mathbf{v}_{1}=\mathbf{v}_{2}$.
Assuming the existence of $\mathbf{u}\in\left(\Z/(p^{k})\right)^{d}$
such that $Q(\mathbf{u})\in\left(\Z/(p^{k})\right)^{\times}$ and
$Q(\mathbf{\mathbf{u}},\mathbf{v}_{1})=0$, we note that $\ovec_{\mathbf{\mathbf{u}}}(\mathbf{v}_{1})=\mathbf{v}_{1}$,
which implies in turn that $\tau_{\mathbf{v}_{1}-\mathbf{v}_{2}}\circ\tau_{\mathbf{\mathbf{u}}}(\mathbf{v}_{1})=\mathbf{v}_{2}$.
To prove the existence of the above $\mathbf{u}$, we note that $Q(\mathbf{v}_{1})\ \text{mod \ensuremath{p}}$
is non-zero, which implies that the restriction of the form $Q$ $(\text{mod \ensuremath{p})}$
to the vector space 
\[
V=\left\{ \mathbf{x}\in(\Z/(p))^{d}\mid Q(\mathbf{x},\mathbf{v}_{1})=0\text{ mod \ensuremath{p}}\right\} 
\]
gives a non-singular form, proving in turn that there exists $\mathbf{\tilde{u}}\in V$
such that $Q(\tilde{\mathbf{u}})$ is non-zero mod $p$. Using\textbf{
}\cite[Section 2, Theorem 1]{Serre_coure_in_arith} (Hensel's Lemma
for several variables) for the polynomial $f(\mathbf{x})=Q(\mathbf{x},\mathbf{v}_{1})$
(by lifting $\mathbf{v}_{1}$ to a $\Z_{p}^{d}$ vector) we deduce
that there exists $\mathbf{u}\in\left(\Z/(p^{k})\right)^{d}$ such
that $\mathbf{u}=\tilde{\mathbf{u}}\text{ mod }p$ and $Q(\mathbf{u},\mathbf{v}_{1})=0$,
and in particular, since $\mathbf{u}=\tilde{\mathbf{u}}\text{ mod }p$,
we get $Q(\mathbf{u})\in\left(\Z/(p^{k})\right)^{\times}$. If on
the other-hand it holds that $Q(\mathbf{v}_{1}+\mathbf{v}_{2})\in\left(\Z/(p^{k})\right)^{\times},$
then we have 
\[
\ovec_{\mathbf{v}_{2}}\circ\ovec_{\mathbf{v}_{1}+\mathbf{v}_{2}}\left(\mathbf{v}_{1}\right)=\mathbf{v}_{2}.
\]

With this we have proved \eqref{enu:mod_q_transitively} for $q$
being a power of an odd prime, and the result for a general $q\in2\N+1$
follows by the Chinese remainder theorem. 

Finally, to validate \eqref{enu:finitely many orbits in V_x(Z)},
note that for $T>0$ 
\[
\left(\tau\left(\SL_{d}(\Z)\right)\cap\H_{T}(\Z)\right)/\SO_{Q}(\Z)=\left(\zprim\cap\H_{T}(\Z)\right)/\SO_{Q}(\Z)=\H_{T,\text{prim}}(\Z)/\SO_{Q}(\Z).
\]
\end{proof}

\subsubsection{Stabilizers subgroups of \emph{$(\protect\SO_{Q}\times\protect\ASL_{d-1})(R)$}}

We now discuss some facts concerning\textbf{ }the stabilizer subgroup
of $(\SO_{Q}\times\ASL_{d-1})(R)$ stabilizing $g\in\SL_{d}(R)$ by
the right action \eqref{eq:definion of action of =00005Cmatbb=00007BG=00007D on Z_t}.
For the following recall that $\mathbf{\Hbold}_{\ovec(g)}(R)\leq\SO_{Q}(R)$
denotes the stabilizer of $\ovec(g)\in R^{d}$ by the $\SO_{Q}(\R)$
action on $R^{d}$ (to recall, see \eqref{eq:action of SO_Q on R^d}).
\begin{lem}
\label{lem:stabilizer lemma}Let \emph{$g\in\SL_{d}(R)$} and consider
the group
\begin{equation}
\Lbold_{g}(R)\overset{\text{def}}{=}\left\{ \left(w,g^{-1}\o(w)g\right)\mid w\in\mathbf{H}_{\ovec(g)}(R)\right\} .\label{eq:stab of L_g}
\end{equation}
Then \emph{$\mathbf{L}_{g}(R)\leq(\SO_{Q}\times\ASL_{d-1})(R)$} is
the stabilizer subgroup of $g$ by the action \emph{\eqref{eq:definion of action of =00005Cmatbb=00007BG=00007D on Z_t}}.
\end{lem}

\begin{proof}
To show that $\mathbf{L}_{g}(R)\leq(\SO_{Q}\times\ASL_{d-1})(R)$
we observe that for all $w\in\mathbf{\Hbold}_{\ovec(g)}(R)$ it holds
that 
\[
\ovec\left(g^{-1}\o(w)g\right)=\o(g^{-1})w\ovec(g)=\o(g^{-1})\ovec(g)=e_{d},
\]
which implies that $g^{-1}\o(w)g\in\ASL_{d-1}(R)$.

Next, as the reader should easily verify, all elements of $\mathbf{L}_{g}(R)$
stabilize $g$. For the other inclusion, let $\left(\rho,\eta\right)\in\left(\SO_{Q}\times\ASL_{d-1}\right)(R)$
be such that 
\begin{equation}
g=g\cdot\left(\rho,\eta\right)=\o(\rho^{-1})g\eta.\label{eq:stabilizing g}
\end{equation}
By rewriting \eqref{eq:stabilizing g}, we get
\begin{equation}
g^{-1}\o(\rho)g=\eta,\label{eq:eta is conjugate to theta rho}
\end{equation}
and we observe that to finish the proof, we need show that $\rho\in\mathbf{H}_{\ovec(g)}(R)$.
Indeed, we have 
\[
\rho^{-1}\ovec(g)=\ovec(\o(\rho^{-1})g)\underbrace{=}_{\text{\ensuremath{\ASL}}_{d-1}(R)\text{ invariance}}\ovec(\o(\rho^{-1})g\eta)\underbrace{=}_{\eqref{eq:stabilizing g}}\ovec(g).
\]
\end{proof}

\subsubsection{The form $Q^{*}$}

We will now go over some technical facts that we need about the groups
$\o(\SO_{Q}(\R))$ and $\o\left(\mathbf{H}_{\mathbf{v}}(\R)\right)$
for $\mathbf{v}\in R^{d}$ (which appears in the second factor of
$\mathbf{L}_{g}(\R)$). In a summary, we will show that $\o(\SO_{Q}(\R))$
is identified with $\SO_{Q^{*}}(\R)$ for a (rational) quadratic form
$Q^{*}$ defined below, and the subgroup $\o\left(\mathbf{H}_{\mathbf{v}}(\R)\right)$
is identified with the subgroup of $\SO_{Q^{*}}(\R)$ that preserves
the orthogonal hyperplane to $\mathbf{v}$ with respect to the Euclidean
inner product.

Let $M\in M_{d}(\Z)$ be the companion matrix of the form $Q$, namely
\[
Q(\mathbf{x})=\mathbf{x}^{t}M\mathbf{x}.
\]
We recall that $Q$ is a non-degenerate integral form, which implies
that $M\in\GL_{d}(\Q)$, and we define the rational form $Q^{*}$
by 
\begin{equation}
Q^{*}(\mathbf{x})\df\mathbf{x}^{t}M^{-1}\mathbf{x}.\label{eq:def of Q^*}
\end{equation}

\begin{rem*}
The form $Q^{*}$ can be defined more intrinsically as follows. Let
$Q(\cdot,\cdot)$ the bi-linear form associated to $Q$. Since $Q$
is non-degenerate, the map
\[
l^{Q}:\R^{d}\to\left(\R^{d}\right)^{*}
\]
where $\left(\R^{d}\right)^{*}$ denotes the dual space, defined by
$l^{Q}(\mathbf{x})\df Q(\cdot,\mathbf{x})$ is a linear isomorphism.
The form $Q^{*}$ can be identified as the form on $\left(\R^{d}\right)^{*}$
which is makes the map $l^{Q}$ an isometry.
\end{rem*}
\begin{lem}
\label{lem:F(SO_Q) and stabilizer of a vector}We have that \emph{$\o(\SO_{Q}(\R))=\SO_{Q^{*}}(\R).$}
Moreover, let \emph{$g\in\SL_{d}(\Z)$} such that $Q(\ovec(g))\neq0$,
then:
\begin{enumerate}
\item \label{enu:=00005Co(stab)=00003Dstab of M=00005Ctau(g)}\emph{We have
that $\o\left(\mathbf{H}_{\ovec(g)}(\R)\right)=\left\{ \rho\in\SO_{Q^{*}}(\R)\mid\rho(M\ovec(g))=M\ovec(g)\right\} $.}
\item \label{enu:orthogonal hyperplan to tau(g)}It holds that $\left(M\ovec(g)\right)^{\perp(Q^{*})}=\text{Span}_{\R}\left\{ \mathbf{g}_{1},...,\mathbf{g}_{d-1}\right\} ,$
where $\left(M\ovec(g)\right)^{\perp(Q^{*})}$ denotes the orthogonal
hyperplane to $M\ovec(g)$ with respect to $Q^{*}$ , and $\mathbf{g}_{i}$
is the $i'th$ column of $g$. Moreover 
\[
\R^{d}=\left(M\ovec(g)\right)^{\perp(Q^{*})}\oplus\text{Span}_{\R}\left\{ M\ovec(g)\right\} .
\]
\end{enumerate}
\end{lem}

\begin{proof}
To show that $\o(\SO_{Q}(\R))$ is the group preserving the form $Q^{*}$,
we observe that
\[
\rho^{t}M\rho=M\iff
\]
\[
\o\left(\rho^{t}M\rho\right)=\o(M)\iff
\]
\begin{equation}
\o(\rho)^{t}M^{-1}\o(\rho)=M^{-1}.\label{eq:F(=00005Crho) stabilizes M^-1}
\end{equation}
Next, to prove that the subgroup $\o\left(\mathbf{H}_{\ovec(g)}(\R)\right)\leq\SO_{Q^{*}}(\R)$
is the stabilizer of $M\ovec(g)$, we observe by \eqref{eq:F(=00005Crho) stabilizes M^-1}
that
\[
\o(\rho)\left(M\ovec(g)\right)=M\rho\ovec(g),
\]
and since $M$ is invertible, we deduce that 
\[
M\rho\ovec(g)=M\ovec(g)\iff\rho\ovec(g)=\ovec(g),
\]
namely $\o(\rho)$ stabilizes $M\ovec(g)$ if and only if $\rho$
stabilizes $\ovec(g)$.

Next, to show \eqref{enu:orthogonal hyperplan to tau(g)}, we note
that 
\[
Q^{*}(M\ovec(g))=Q(\ovec(g))\neq0,
\]
which by \cite[Lemma 1.3]{CASSELS_quad} shows that 
\[
\R^{d}=\left(M\ovec(g)\right)^{\perp(Q^{*})}\oplus\text{Span}_{\R}\left\{ M\ovec(g)\right\} .
\]
Note that the bi-linear form $B_{Q^{*}}$ determined by $Q^{*}$ is
given by 
\[
B_{Q^{*}}(\mathbf{u}_{1},\mathbf{u}_{2})=\left\langle \mathbf{u}_{1},M^{-1}\mathbf{u}_{2}\right\rangle .
\]
Let $\mathbf{g}_{i}$ be the $i$'th column of $g$, then 
\[
B_{Q_{*}}(\mathbf{g}_{i},M\ovec(g))=\left\langle \mathbf{g}_{i},M^{-1}M\ovec(g)\right\rangle =\left\langle g\mathbf{e}_{i},\left(g^{t}\right)^{-1}\mathbf{e}_{d}\right\rangle =\delta_{i,d},
\]
which proves that $\left(M\ovec(g)\right)^{\perp(Q^{*})}=\text{Span}_{\R}\left\{ \mathbf{g}_{1},...,\mathbf{g}_{d-1}\right\} .$
\end{proof}

\subsection{\label{subsec:The-equivariant-isomorphism}The equivariant isomorphism}

Our goal now is to describe a one-parameter group $\left\{ \a{}\right\} _{T>0}\leq\SL_{d}(\R)$
such that $\a{}\in\V_{Q(\sqrt{T}\mathbf{e}_{d})}(\R)$ for all $T>0$,
and such that the stabilizer group $\mathbf{L}_{\a{}}(\R)\leq\SO_{Q}\times\ASL_{d-1}(\R)$
of $\a{}$ is independent of $T$. This will allow us to define a
$(\SO_{Q}\times\ASL_{d-1})(\R)$ equivariant map $\V_{T_{1}}(\R)\to\V_{T_{2}}(\R)$,
for $T_{i}>0$.

We note that $Q(\ovec(I_{d}))=Q(\mathbf{e}_{d})\neq0$, and by Lemma
\ref{lem:F(SO_Q) and stabilizer of a vector},\eqref{enu:orthogonal hyperplan to tau(g)}
we obtain
\[
\R^{d}=\text{Span}_{\R}\{\mathbf{e}_{1},..,\mathbf{e}_{d-1}\}\oplus\text{Span}_{\R}\left\{ M\mathbf{e}_{d}\right\} 
\]
 where $\text{Span}_{\R}\{\mathbf{e}_{1},..,\mathbf{e}_{d-1}\}$ and
$\text{Span}_{\R}\left\{ M\mathbf{e}_{d}\right\} $ are invariant
spaces under the ordinary left $\o\left(\mathbf{H}_{\mathbf{e}_{d}}(\R)\right)$-linear
action.
\begin{defn}
\label{def:def of a_T}For $T>0$ we define $\a{}\in\SL_{d}(\R)$
to be the unique matrix which acts on $P_{0}\df\text{Span}_{\R}\{\mathbf{e}_{1},..,\mathbf{e}_{d-1}\}$
by scalar multiplication of a factor of $T^{\frac{1}{2(d-1)}}$ and
on $P_{0}^{\perp(Q^{*})}\df\text{Span}_{\R}\{M\mathbf{e}_{d}\}$ by
scalar multiplication of a factor of $T^{-1/2}$.
\end{defn}

\begin{cor}
\label{cor:Special alligned base points}It holds that $\a{}\in\V_{Q(\sqrt{T}\mathbf{e}_{d})}(\R),\ \forall T>0$,
and $\mathbf{L}_{\a{}}(\R)=\mathbf{L}_{I_{d}}(\R).$
\end{cor}

\begin{proof}
In order to validate that $\a{}\in\V_{Q(\sqrt{T}\mathbf{e}_{d})}(\R)$,
we show below that 
\begin{equation}
\tau(\a{})=\sqrt{T}\mathbf{e}_{d}.\label{eq:ovec a_T}
\end{equation}
We have
\[
\left\langle \mathbf{e}_{i},\ovec(a_{T})\right\rangle =\left\langle \mathbf{e}_{i},\left(a_{T}^{t}\right)^{-1}\mathbf{e}_{d}\right\rangle =
\]
\[
\left\langle a_{T}^{-1}\mathbf{e}_{i},\mathbf{e}_{d}\right\rangle \underbrace{=}_{\text{Definition \ref{def:def of a_T}}}T^{\frac{1}{2}}\delta_{i,d},
\]
which implies \eqref{eq:ovec a_T}. Next, since $P_{0}$ and $P_{0}^{\perp(Q^{*})}$
are invariant spaces under the left linear $\o\left(\mathbf{H}_{\mathbf{e}_{d}}(\Q)\right)$
action, and since $a_{T}$ acts by scalar multiplication on each of
these spaces, it follows that $\a{}$ is in the center of $\o\left(\mathbf{H}_{\mathbf{e}_{d}}(\R)\right)$.
Therefore
\[
\begin{aligned}\mathbf{L}_{\a{}}(\R)= & \left\{ \left(w,\left(\a{}\right)^{-1}\o(w)\a{}\right)\mid w\in\mathbf{H}_{\ovec(\a{})}(\R)\right\} \\
= & \left\{ \left(w,\o(w)\right)\mid w\in\mathbf{H}_{\tau(\a{})}(\R)\right\} .
\end{aligned}
\]

Now we have by \eqref{eq:ovec a_T} that $\mathbf{H}_{\tau(\a{})}(\R)=\mathbf{H}_{\sqrt{T}\mathbf{e}_{d}}(\R)=\mathbf{H}_{\mathbf{e}_{d}}(\R)=\mathbf{H}_{\ovec(I_{d})}(\R),$
which in turn implies that $\mathbf{L}_{\a{}}(\R)=\mathbf{L}_{I_{d}}(\R).$
\end{proof}
\emph{For the rest of the paper we denote }
\begin{equation}
H\df L_{I_{d}}(\R)=\left\{ \left(w,\o(w)\right)\mid w\in\mathbf{H}_{\mathbf{e}_{d}}(\R)\right\} .\label{eq:def_of_H}
\end{equation}
By Corollary \ref{cor:Special alligned base points} we have for all
$T>0$ the identification
\begin{equation}
\V_{Q(\sqrt{T}\mathbf{e}_{d})}(\R)\cong H\backslash(\SO_{Q}\times\ASL_{d-1})(\R)\label{eq:polar coordinates for z_T}
\end{equation}
 by the orbit map
\[
H\left(\rho,\eta\right)\mapsto\o(\rho^{-1})a_{T}\eta.
\]
We define 
\[
\pi_{\V_{Q(\sqrt{T}\mathbf{e}_{d})}}:\V_{Q(\sqrt{T}\mathbf{e}_{d})}(\R)\to\V_{Q(\mathbf{e}_{d})}(\R),
\]
by 
\begin{equation}
\pi_{\V_{Q(\sqrt{T}\mathbf{e}_{d})}}\left(\o(\rho^{-1})a_{T}\eta\right)\df\o(\rho^{-1})I_{d}\eta=\o(\rho^{-1})\eta,\label{eq:def of pi_Z}
\end{equation}
which is clearly equivariant with respect to the action of $(\SO_{Q}\times\ASL_{d-1})(\R)$
on $\V_{Q(\sqrt{T}\mathbf{e}_{d})}(\R)$ and $\V_{Q(\mathbf{e}_{d})}(\R)$
(since $a_{T}$ has the same stabilizer $\forall T>0$).

\subsubsection{The natural measure on $\protect\V_{Q(\mathbf{e}_{d})}(\protect\R)$}

We now define a $(\SO_{Q}\times\ASL_{d-1})(\R)$ invariant measure
on $\V_{Q(\mathbf{e}_{d})}(\R)$ using the identification \eqref{eq:polar coordinates for z_T}.
We choose Haar measures $m_{\SO_{Q}(\R)},\ m_{\ASL_{d-1}(\R)}$ on
$\SO_{Q}(\R)$ and $\ASL_{d-1}(\R)$ respectively with a normalization
we discuss in Section \ref{subsec:Measures-as measures on fibre bundles},
and we observe that $H$ is compact (by \eqref{eq:def_of_H}, we have
$H\cong\mathbf{H}_{\mathbf{e}_{d}}(\R)$, and recall that $\mathbf{H}_{\mathbf{e}_{d}}(\R)$
is compact under our \nameref{subsec:Standing-Assumption}). Then
on $\V_{Q(\mathbf{e}_{d})}(\R)$ we can define the following measure
\begin{equation}
\mu_{\V}\df\left(\pi_{H}\right)_{*}m_{\SO_{Q}(\R)}\otimes m_{\ASL_{d-1}(\R)},\label{eq:def of mu_Z}
\end{equation}
where $\pi_{H}:(\SO_{Q}\times\ASL_{d-1})(\R)\to H\backslash(\SO_{Q}\times\ASL_{d-1})(\R)$
is the natural quotient map.

\subsection{\label{subsec:Statistics-of-Z}Statistics of $\protect\V_{N}(\protect\Z)$
as $N\to\infty$}

We are now ready to discuss our main results. Let $N\in\N$ and consider
the following atomic measure on $\mathcal{\V}_{Q(\mathbf{e}_{d})}(\R)$
\begin{equation}
\nu_{N}^{\V}=\frac{1}{\left|\H_{N,\text{prim}}(\Z)/\SO_{Q}(\Z)\right|}\sum_{x\in\V_{N}(\Z)}\delta_{\pi_{\V_{N}}(x)}.\label{eq:counting mease on z}
\end{equation}

The following definition amounts to a congruence condition of the
range of $N\in\N$ for which we are able to obtain the asymptotics
of the measures $\nu_{N}$.
\begin{defn}
\label{def:Isotropicity definition}Given a prime $p$ and a rational
quadratic form $Q$, we say that $\mathbf{v}\in\Q^{d}$ is $\left(Q,p\right)$
co-isotropic if $\mathbf{H}_{\mathbf{v}}(\Q_{p})$ (the stabilizer
of \textbf{$\mathbf{v}$ }in the group\textbf{ }$\SO_{Q}(\Q_{p})$\textbf{)}
is non-compact. We say that $N\in\N$ has the $\left(Q,p\right)$
co-isotropic property if there exists $\mathbf{v}\in\H_{N,\text{prim}}(\Z)$
which is $\left(Q,p\right)$ co-isotropic.
\end{defn}

\begin{rem*}
For $\mathbf{v}\in\Q^{d}$ we have $\mathbf{H}_{\mathbf{v}}(\Q_{p})$
is non-compact if and only if $\exists\mathbf{u}\in\Q_{p}^{d}\otimes\mathbf{v}{}^{\perp(Q)}$
such that $Q(\mathbf{u})=0$, where $\mathbf{v}{}^{\perp(Q)}$ is
the orthogonal hyperplane \emph{with respect to $Q$. }We note that
if $Q$ is a rational quadratic form in $d\geq6$ variables, then
the form induced on $\Q_{p}^{d}\otimes\mathbf{v}{}^{\perp(Q)}$ is
in $d\geq5$ variables and by \cite{CASSELS_quad} (see \cite[Lemma 1.7]{CASSELS_quad}),
we obtain that any $\mathbf{v}\in\Q^{d}$ is $\left(Q,p\right)$ co-isotropic,
for any prime $p$.
\end{rem*}
Our main results are as follows.
\begin{thm}
\label{thm:main thm for Z} Assume that $\{T_{n}\}_{n=1}^{\infty}\subseteq\N$
is a sequence of integers satisfying the $\left(Q,p_{0}\right)$ co-isotropic
property for some fixed odd prime $p_{0}$, and $T_{n}\to\infty$.
Then for all $f\in C_{c}(\V_{Q(\mathbf{e}_{d})}(\R))$ we have that\emph{
\[
\lim_{n\to\infty}\nu_{T_{n}}^{\V}(f)=\mu_{\mathcal{\V}}(f).
\]
}
\end{thm}

Next, for $N\in\N$ and $q\in\N$ we consider the following measure
on $\mathcal{\V}_{Q(\mathbf{e}_{d})}(\R)\times\V_{\vartheta_{q}(T)}(\Z/(q))$
given by
\begin{equation}
\nu_{N}^{\V,q}=\frac{1}{\left|\H_{N,\text{prim}}(\Z)/\SO_{Q}(\Z)\right|}\sum_{x\in\V_{N}(\Z)}\delta_{(\pi_{\V_{N}}(x),\vartheta_{q}(x))}.\label{eq:congruence counting measures on nu_Z_Q(E_d))}
\end{equation}

\begin{thm}
\label{thm:main_thm_with_congruences-forZ}Let $q\in2\N+1$. In addition
to our \nameref{subsec:Standing-Assumption} on the form $Q$, assume
that $Q$ is non-singular modulo $q$ \emph{(}see Definition \ref{def:non-singularity modq}\emph{)}.
Let $\{T_{n}\}_{n=1}^{\infty}\subseteq\N$ be a sequence of integers
satisfying the $\left(Q,p_{0}\right)$ co-isotropic property for some
odd prime $p_{0}$, and assume that there is a fixed $a\in\left(\Z/(q)\right)^{\times}$
such that for all $n\in\N$ it holds $\red_{q}\left(T_{n}\right)=a$.
Then, for all $f\in C_{c}(\V_{Q(\mathbf{e}_{d})}(\R)\times\V_{a}(\Z/(q)))$
we have that\emph{
\[
\lim_{n\to\infty}\nu_{T_{n}}^{\V,q}(f)=\mu_{\mathcal{\V}}\otimes\mu_{\V_{a}(\Z/(q))}(f),
\]
}where\emph{ $\mu_{\V_{a}(\Z/(q))}$ is the uniform probability measure
on $\V_{a}(\Z/(q))$.}
\end{thm}

\section{\label{sec:moduli-spaces}moduli spaces - refinements of \cite{AESgrids}}

This section discusses our results which generalize \cite{AESgrids}.
We note that these results are also conceptually similar to the Linnik
type results that we discussed in Section \ref{subsec:Linnik-type-problems},
and are roughly described as follows. We will introduce moduli spaces
$\Y(\R)$ and $\X(\R)$ which are fiber bundles over $\R^{d}\smallsetminus\mathbf{0}$
with fibers that are isomorphic to $Y_{d-1}=\ASL_{d-1}(\R)/\ASL_{d-1}(\Z)$
and $X_{d-1}=\SL_{d-1}(\R)/\SL_{d-1}(\Z)$. Taking the preimage of
a quadratic variety $\H_{T}(\R)\subseteq\R^{d}\smallsetminus\mathbf{0}$
by the projection map to $\R^{d}\smallsetminus\mathbf{0}$, we obtain
for $\mathcal{M}\in\{\Y,\X\}$ a one parameter family of subbundles
$\mathcal{M}_{T}(\R)\subseteq\mathcal{M}(\R)$ over $\H_{T}(\R)$,
which are all isomorphic. We will define a geometrically motivated
homeomorphism $\pi_{\mathcal{M}_{T}}:\mathcal{M}_{T}(\R)\to\mathcal{M}_{Q(\mathbf{e}_{d})}(\R)$,
and our main results, Theorems \ref{thm:moduli main thm }-\ref{thm:moduli_main_thm_with_congruences},
will be about the distribution of $\pi_{\mathcal{M}_{T}}\left(\mathcal{M}_{T}(\Z)\right)$
in $\mathcal{M}_{Q(\mathbf{e}_{d})}(\R)$, where $\mathcal{M}_{T}(\Z)=\mathcal{M}(\Z)\cap\mathcal{M}_{T}(\R)$.

The structure of this section is as follows:
\begin{itemize}
\item Sections \ref{subsec:The-moduli-space X(R)}-\ref{subsec:The-space-of directed folliations}
discuss $\X(\R)$ and $\Y(\R)$.
\item Section \ref{subsec:Moduli-level-sets,} discusses the subbundles
$\mathcal{M}_{T}(\R)\subseteq\mathcal{M}(\R),\ T>0$, the homeomorphisms
$\pi_{\mathcal{M}_{T}}$, and some natural measures on these subbundles.
\item Section \ref{subsec:Statistics-in-moduli} states Theorems \ref{thm:moduli main thm }-\ref{thm:moduli_main_thm_with_congruences}.
\item Section \ref{subsec:Proof-of-example from intro} relying on Theorems
\ref{thm:moduli main thm }-\ref{thm:moduli_main_thm_with_congruences}
proves Theorems \ref{thm:main thm-particular case} and \ref{thm:main_thm_particular case with congruences}
from the introduction.
\end{itemize}

\subsection{\label{subsec:The-moduli-space X(R)}The moduli space of oriented
rank $d-1$ discrete subgroups of $\protect\R^{d}$}

Instead of considering the shapes of orthogonal lattices to integral
vectors (which we introduced in Section \ref{subsec:Overview-and-motivation}),
we may consider the orthogonal lattices ``as is'' by 
\[
\X(\Z)\df\left\{ \left(\Lambda_{\mathbf{v}},\mathbf{v}\right)\mid\mathbf{v}\in\zprim,\ \Lambda_{\mathbf{v}}=\Z^{d}\cap\mathbf{v}^{\perp}\right\} .
\]
We will now describe a homogeneous space $\X(\R)$, which can be thought
of as a natural ambient space that contains $\X(\Z)$.

We let $X_{d-1,d}$ be the space of rank $(d-1)$-discrete subgroups
of $\R^{d}$, and we define $\X(\R)\subseteq X_{d-1,d}\times\R^{d}\smallsetminus\mathbf{0}$
by
\[
\X(\R)\df\left\{ \left(\Lambda,\mathbf{v}\right)\in X_{d-1,d}\times\R^{d}\smallsetminus\mathbf{0}\mid\mathbf{v}\perp\Lambda,\ \text{covol}(\Lambda)=\norm{\mathbf{v}}\right\} ,
\]
and as we now show, $\X(\R)$ is a homogeneous space. We consider
the left action of $\SL_{d}(\R)$ on $X_{d-1,d}\times\R^{d}$ given
by
\begin{equation}
g\cdot\left(\Lambda,\mathbf{v}\right)\df\left(g\Lambda,\o(g)\mathbf{v}\right),\ g\in\SL_{d}(\R).\label{eq:action of sl_d(R) on X(R)}
\end{equation}

\begin{lem}
\label{lem:description of X(=00005CR)}It holds that \emph{
\[
\X(\R)=\SL_{d}(\R)\cdot\left(\text{Span}_{\Z}\{\mathbf{e}_{1},..,\mathbf{e}_{d-1}\},\mathbf{e}_{d}\right).
\]
}
\end{lem}

\begin{proof}
It is straightforward to verify that $\SL_{d}(\R)$ acts transitively
on $X_{d-1,d}$. The rest follows by \eqref{eq:orthogonality of theta}
and Lemma \ref{lem:the covolume of lambda_v and the map theta}.
\end{proof}
By noting that the stabilizer of $\left(\text{Span}_{\Z}\{\mathbf{e}_{1},..,\mathbf{e}_{d-1}\},\mathbf{e}_{d}\right)$
is the subgroup $\ASL_{d-1}(\Z)U\cong\SL_{d-1}(\Z)\ltimes\R^{d-1}$,
where 
\[
U=\left\{ \left(\begin{array}{cc}
I_{d-1} & \mathbf{v}\\
\mathbf{0} & 1
\end{array}\right)\mid\mathbf{v}\in\R^{d-1}\right\} ,
\]
we deduce the identification 
\[
\X(\R)=\SL_{d}(\R)/\left(\ASL_{d-1}(\Z)U\right).
\]

By restricting the above $\SL_{d}(\R)$ action on $\X(\R)$ to $\SL_{d}(\Z)$,
we obtain the following observation.
\begin{lem}
It holds that \emph{
\[
\X(\Z)=\SL_{d}(\Z)\cdot\left(\text{Span}_{\Z}\{\mathbf{e}_{1},..,\mathbf{e}_{d-1}\},\mathbf{e}_{d}\right).
\]
}
\end{lem}

\begin{proof}
Since the columns of $g\in\SL_{d}(\Z)$ form a $\Z$-basis for $\Z^{d}$,
and since $\ovec(g)$ is orthogonal to the first $d-1$ columns of
$g$ (see \eqref{eq:orthogonality of theta}), we have
\begin{align*}
\Lambda_{\ovec(g)}= & \text{Span}_{\Z}\{g\mathbf{e}_{1},..,g\mathbf{e}_{d-1}\}\\
= & g\cdot\text{Span}_{\Z}\{\mathbf{e}_{1},..,\mathbf{e}_{d-1}\}.
\end{align*}

Finally, we note that $\ovec\left(\SL_{d}(\Z)\right)=\zprim$ (to
recall $\ovec$, see \eqref{eq:def of tau})
\end{proof}
We now observe that the map $\pi_{vec}^{\X}:\X(\R)\to\R^{d}\smallsetminus\mathbf{0}$
defined by 
\begin{equation}
\pi_{vec}^{\X}\left((\Lambda,\mathbf{v})\right)\df\mathbf{v},\label{eq:def of pi^X_vec}
\end{equation}
gives $\X(\R)$ the structure of a fiber bundle with fibers isomorphic
to $X_{d-1}$. Indeed 
\[
\begin{aligned}(\pi_{vec}^{\X})^{-1}(\mathbf{v}_{0})= & \left\{ (\Lambda,\mathbf{v}_{0})\in X_{d-1,d}\times\{\mathbf{v}_{0}\}\mid\Lambda\perp\mathbf{v}_{0},\ \text{covol}(\Lambda)=\norm{\mathbf{v}_{0}}\right\} \\
\cong & \left\{ \Lambda\in X_{d-1,d}\mid\mathbf{v}_{0}\perp\Lambda,\ \text{covol}(\Lambda)=\norm{\mathbf{v}_{0}}\right\} \\
\cong & X_{d-1}.
\end{aligned}
\]

\subsubsection{\label{subsec:The-extension-of shape to X(R)}The extension of the
\textquotedblleft shape\textquotedblright{} map to $\protect\X(\protect\R)$}

We now reconsider the map $\shape:\zprim\to\mathcal{S}_{d-1}$ from
Section \ref{subsec:Overview-and-motivation} and extend it to $\X(\R)$.

We note that $\SO_{d}(\R)$ acts on $\X(\R)$ by 
\[
\rho\cdot\left(\Lambda,\mathbf{v}\right)\df\left(\rho\Lambda,\rho\mathbf{v}\right),\ \ \rho\in\SO_{d}(\R),\ (\Lambda,\mathbf{v})\in\X(\R),
\]
which is the restriction of \eqref{eq:action of sl_d(R) on X(R)}
to $\SO_{d}(\R)$, and we let $K\df\SO_{d}(\R)\cap\ASL_{d-1}(\R)$
be the stabilizer of $\mathbf{e}_{d}$ by the ordinary $\SO_{d}(\R)$
left linear action on $\R^{d}$. Since $(\pi_{vec}^{\X})^{-1}(\mathbf{e}_{d})$
is identified with the space of of full rank lattices in $\R^{d-1}$,
and since $K$ acts on $(\pi_{vec}^{\X})^{-1}(\mathbf{e}_{d})$ by
Euclidean rotations in the plane $\mathbf{e}_{d}^{\perp}$, we obtain
that $\mathcal{S}_{d-1}$ identifies naturally with the space of $K$-orbits
in $(\pi_{vec}^{\X})^{-1}(\mathbf{e}_{d})$. Since $(\pi_{vec}^{\X})^{-1}(\mathbf{e}_{d})$
is the $\ASL_{d-1}(\R)$ orbit passing through $\left(\text{Span}_{\Z}\{\mathbf{e}_{1},..,\mathbf{e}_{d-1}\},\mathbf{e}_{d}\right)$,
we get that $(\pi_{vec}^{\X})^{-1}(\mathbf{e}_{d})\cong\ASL_{d-1}(\R)/\ASL_{d-1}(\Z)U$,
and we conclude that
\begin{equation}
\mathcal{S}_{d-1}\cong K\backslash\ASL_{d-1}(\R)/\ASL_{d-1}(\Z)U.\label{eq:def of shapes in terms of ASL}
\end{equation}

Next, for $\mathbf{v}\in\R^{d}\smallsetminus\mathbf{0}$ we choose
a $\rho_{\mathbf{v}}\in\SO_{d}(\R)$ such that $\rho_{\mathbf{v}}\mathbf{v}=\norm{\mathbf{v}}\mathbf{e}_{d}$,
and for $t>0$ we define $d_{t}\in\SL_{d}(\R)$ by 
\[
d_{t}\df\left(\begin{array}{cc}
t^{-1/(d-1)}I_{d-1}\\
 & t
\end{array}\right).
\]

Then
\[
d_{\norm{\mathbf{v}}}\rho_{\mathbf{v}}\cdot(\Lambda,\mathbf{v})=((d_{\norm{\mathbf{v}}}\rho_{\mathbf{v}})\Lambda,\mathbf{e}_{d})\in(\pi_{vec}^{\X})^{-1}(\mathbf{e}_{d}),
\]
and we note that $(d_{\norm{\mathbf{v}}}\rho_{\mathbf{v}})\Lambda=\rho_{\mathbf{v}}(\norm{\mathbf{v}}{}^{-1/(d-1)}\Lambda)$.
We observe that the $K$ orbit $K(d_{\norm{\mathbf{v}}}\rho_{\mathbf{v}})\Lambda\subseteq(\pi_{vec}^{\X})^{-1}(\mathbf{e}_{d})$
is independent of the choice of $\rho_{\mathbf{v}}$, and we define
$\shape:\X(\R)\to\mathcal{S}_{d-1}$ by
\begin{equation}
\shape(\Lambda,\mathbf{v})\df K(d_{\norm{\mathbf{v}}}\rho_{\mathbf{v}})\Lambda,\ (\Lambda,\mathbf{v})\in\X(\R).\label{eq:definition of shape in X(R)}
\end{equation}

\subsection{\label{subsec:The-space-of directed folliations} The space of unimodular
lattices with a marked rational hyperplane}

As in \cite{AESgrids}, we describe an object that extracts more information
from a primitive vector $\mathbf{v}$ than we get from $\Lambda_{\mathbf{v}}$
by telling us how $\Lambda_{\mathbf{v}}$ is completed to $\Z^{d}$.
Namely, for $\mathbf{v}\in\zprim$, we let $\mathbf{w}\in\zprim$
such that 
\[
\Lambda_{\mathbf{v}}\oplus\mathbf{w}\Z=\Z^{d}.
\]
We say that $\mathbf{w}$ completes $\Lambda_{\mathbf{v}}$ in a positive
direction if $\left\langle \mathbf{v},\mathbf{w}\right\rangle >0$.
This data is concisely recorded by the triple $(\Z^{d},\mathbf{v}^{\perp},\mathbf{v})$
in a natural way, and motivates us to consider
\[
\Y(\Z)\df\left\{ (\Z^{d},\mathbf{v}^{\perp},\mathbf{v})\mid\mathbf{v}\in\zprim\right\} .
\]
As for $\X(\R)$ and $\X(\Z)$, we will now describe $\Y(\R)$ as
a homogeneous space that can be thought of as a natural ambient space
containing $\Y(\Z)$.

We let $X_{d}$ be the space of unimodular lattices in $\R^{d}$ and
we denote by $\gr$ the space of hyperplanes in $\R^{d}$. For $L\in X_{d}$
we define $\gr_{L}$ to be the space of $L$-rational hyperplanes,
namely
\[
\gr_{L}\df\left\{ P\in\text{Gr}(d-1,d)\mid P\cap L\text{ is a rank \ensuremath{(d-1)}-discrete group of \ensuremath{\R^{d}}}\right\} ,
\]
and we define 
\begin{equation}
\Y(\R)\df\left\{ \left(L,P,\mathbf{v}\right)\in X_{d}\times\gr\times\R^{d}\mid P\in\gr_{L},\ P\perp\mathbf{v},\ \norm{\mathbf{v}}=\covol(L\cap P)\right\} .\label{eq:Y(R) as space of triples}
\end{equation}
We define a left action of $\SL_{d}(\R)$ on $X_{d}\times\gr\times\R^{d}$
by
\begin{equation}
g\cdot\left(L,P,\mathbf{v}\right)\df\left(gL,gP,\o(g)\mathbf{v}\right),\ g\in\SL_{d}(\R).\label{eq:SLd(R) action on Y(R)}
\end{equation}

\begin{lem}
\label{lemtransitive action on Y(R)}It holds that \emph{$\Y(\R)=\SL_{d}(\R)\cdot\left(\Z^{d},\text{Span}_{\R}\{\mathbf{e}_{1},..,\mathbf{e}_{d-1}\},\mathbf{e}_{d}\right).$}
\end{lem}

\begin{proof}
It is well known that $\SL_{d}(\R)$ acts transitively on $X_{d}$
and that the stabilizer in $\SL_{d}(\R)$ of a lattice $L$ acts transitively
on $\gr_{L}$. The rest follows by \eqref{eq:orthogonality of theta}
and Lemma \ref{lem:the covolume of lambda_v and the map theta}.
\end{proof}
We observe that the stabilizer of $\left(\Z^{d},\text{Span}_{\R}\{\mathbf{e}_{1},..,\mathbf{e}_{d-1}\},\mathbf{e}_{d}\right)$
is $\ASL_{d-1}(\Z)$, hence 
\[
\Y(\R)=\SL_{d}(\R)/\ASL_{d-1}(\Z).
\]
By restricting the action of $\SL_{d}(\R)$ to $\SL_{d}(\Z)$, we
obtain the following observation which we leave the reader to verify.
\begin{lem}
We have \emph{$\Y(\Z)=\SL_{d}(\Z)\cdot\left(\Z^{d},\text{Span}_{\Z}\{\mathbf{e}_{1},..,\mathbf{e}_{d-1}\},\mathbf{e}_{d}\right).$}
\end{lem}

\subsubsection{The projection to $\protect\X(\protect\R)$}

A natural connection between $\Y(\R)$ and $\X(\R)$ is given by the
projection $\pi_{\cap}:\Y(\R)\to\X(\R)$ defined by 
\begin{equation}
\pi_{\cap}\left(\left(L,P,\mathbf{w}\right)\right)\df\left(L\cap P,\mathbf{w}\right).\label{eq:def of pi_cap}
\end{equation}
We observe that for $\left(\Lambda,\mathbf{v}\right)\in\X(\R)$, the
fiber $\pi_{\cap}^{-1}\left(\left(\Lambda,\mathbf{v}\right)\right)$
consists of the triples of the form
\[
\left(\Lambda+\left(\mathbf{u}+\frac{1}{\covol(\Lambda)}\mathbf{v}\right)\Z,\Lambda\otimes\R,\mathbf{v}\right),
\]
where $\mathbf{u}\in\Lambda\otimes\R$.\textbf{ }Namely, the fiber
$\pi_{\cap}^{-1}\left(\left(\Lambda,\mathbf{v}\right)\right)$ can
be identified with $\left(\Lambda\otimes\R\right)/\Lambda\cong\R^{d-1}/\Z^{d-1}$.
In terms of coset spaces, we have 
\begin{equation}
\pi_{\cap}\left(g\ASL_{d-1}(\Z)\right)=g\left(\ASL_{d-1}(\Z)U\right),\label{eq:def of pi_cap in coset space}
\end{equation}
which implies that
\[
\pi_{\cap}^{-1}\left(g(\ASL_{d-1}(\Z)U)\right)=g\ASL_{d-1}(\Z)U/\ASL_{d-1}(\Z)\cong\R^{d-1}/\Z^{d-1}.
\]
In particular, $\pi_{\cap}$ has compact fibers.
\begin{rem*}
We note that the analogue space to $\Y(\R)$ for dimensions $3\leq k<d-1$
in $d$-space was recently considered in \cite{Aka_Musso_Wieser}
which studies a problem similar to the one addressed in the current
paper.
\end{rem*}

\subsubsection{\label{subsec:-Y(R) as the space of ortiented grids}$\protect\Y(\protect\R)$
as the space of oriented $\left(d-1\right)$-grids in $\protect\R^{d}$}

We will now present another description of $\Y(\R)$ that, in our
opinion, is more geometrically transparent. This description more
clearly connects $\Y(\R)$ to the notion of grid shapes considered
in \cite{AESgrids} (which actually motivated us to consider the space
$\Y(\R)$).

We recall the space of unimodular grids in $\R^{d-1}$ (translates
of unimodular lattices) 
\[
Y_{d-1}\df\left\{ \Lambda+\mathbf{u}\mid\Lambda\in X_{d-1},\ \mathbf{u}\in\R^{d-1}\right\} .
\]
For a triple $\left(L,P,\mathbf{v}\right)\in\Y(\R)$ we let $\mathbf{w}\in L$
such that $(L\cap P)\oplus\mathbf{w}\Z=L$ and $\left\langle \mathbf{w},\mathbf{v}\right\rangle >0$.
We denote by $\pi_{P}^{\perp}:\R^{d}\to P$ the orthogonal projection.
Then 
\begin{equation}
\pi_{P}^{\perp}\left((L\cap P)+\mathbf{w}\right)=(L\cap P)+\pi_{P}^{\perp}(\mathbf{w}),\label{eq:identification of triples with grids}
\end{equation}
which can be viewed as a grid that sits in the hyperplane $P$, is
independent of the choice of $\mathbf{w}$. By defining $f(L,P,\mathbf{v})\df((L\cap P)+\pi_{P}^{\perp}(\mathbf{w}),\mathbf{v})$,
obtain an identification of $\Y(\R)$ with 
\[
\left\{ (\Lambda+\mathbf{u},\mathbf{v})\mid\Lambda\in X_{d-1,d},\ \mathbf{u}\in\Lambda\otimes\R,\ \mathbf{v\perp\text{\ensuremath{\Lambda}}},\ \norm{\mathbf{v}}=\covol(\Lambda)\right\} .
\]
Using the above description of $\Y(\R)$, we see that the projection
$\pi_{vec}^{\Y}:\Y(\R)\to\R^{d}\smallsetminus\mathbf{0}$ defined
by 
\begin{equation}
\pi_{vec}^{\Y}\left((\Lambda+\mathbf{u},\mathbf{v})\right)\df\mathbf{v},\label{eq:def of pi^y_vec}
\end{equation}
endows $\Y(\R)$ with the structure of a fiber bundle with fibers
isomorphic to $Y_{d-1}$.

\subsubsection*{A quick summary - hierarchy of moduli spaces}

We summarize the discussion concerning the moduli spaces by the following
commuting diagram 
\begin{equation}
\xymatrix{\Y(\R)\ar[r]^{\pi_{\cap}}\ar@/_{2pc}/[rr]^{\pi_{vec}^{\Y}} & \X(\R)\ar[r]^{\pi_{vec}^{\X}} & \R^{d}\smallsetminus\mathbf{0}}
\label{eq:diagram of real spaces}
\end{equation}
and we note that in terms of coset spaces, the following diagram is
equivalent to \eqref{eq:diagram of real spaces} \\
\\
\begin{equation}
\xymatrix{\SL_{d}(\R)/\ASL_{d-1}(\Z)\ar[r]\ar@/_{2pc}/[rr] & \SL_{d}(\R)/\ASL_{d-1}(\Z)U\ar[r] & \SL_{d}(\R)/\ASL_{d-1}(\R)}
\label{eq:diagram of coset spaces}
\end{equation}
\\

where all the maps are the natural projections.

\subsection{\label{subsec:Moduli-level-sets,}Moduli level sets, their measures
and their isomorphisms}

Let $Q$ be as in our \nameref{subsec:Standing-Assumption}. For $T>0$
we define 
\begin{equation}
\Y_{T}(\R)\df\left(\pi_{vec}^{\Y}\right)^{-1}\left(\H_{T}(\R)\right),\ \ \text{\ensuremath{\X}}_{T}(\R)\df\left(\pi_{vec}^{\X}\right)^{-1}\left(\H_{T}(\R)\right),\label{eq:def of moduli level sets}
\end{equation}
namely 
\[
\X_{T}(\R)\df\left\{ \left(\Lambda,\mathbf{v}\right)\in\X(\R)\mid Q(\mathbf{v})=T\right\} ,
\]
and 
\[
\Y_{T}(\R)\df\left\{ \left(L,P,\mathbf{v}\right)\in\Y(\R)\mid Q(\mathbf{v})=T\right\} .
\]

We note the following commuting diagram (which follows from \eqref{eq:diagram of real spaces})
that describes the hierarchy between the above moduli level sets
\begin{equation}
\xymatrix{\Y_{T}(\R)\ar[r]^{\pi_{\cap}}\ar@/_{2pc}/[rr]^{\pi_{vec}^{\Y}} & \X_{T}(\R)\ar[r]^{\pi_{vec}^{\X}} & \H_{T}(\R)}
\label{eq:diagram of level sets}
\end{equation}

Next, we define the integral points lying on the moduli level sets.
We consider for $N\in\N$ 
\[
\H_{N,\text{prim}}(\Z)\df\left\{ \mathbf{x}\in\zprim\mid Q(\mathbf{x})=N\right\} ,
\]
and we define 
\[
\X_{N}(\Z)\df\X(\Z)\cap\X_{N}(\R)=\left\{ (\Lambda_{\mathbf{v}},\mathbf{v})\mid\mathbf{v}\in\H_{N,\text{prim}}(\Z)\right\} ,
\]
and 
\[
\Y_{N}(\Z)\df\Y(\Z)\cap\Y_{N}(\R)=\left\{ (\Z^{d},\mathbf{v}^{\perp},\mathbf{v})\mid\mathbf{v}\in\H_{N,\text{prim}}(\Z)\right\} .
\]
We also note the following commuting diagram
\begin{equation}
\xymatrix{\Y_{N}(\Z)\ar@{<->}[r]^{\pi_{\cap}}\ar@{<->}@/_{2pc}/[rr]^{\pi_{vec}^{\Y}} & \X_{N}(\Z)\ar@{<->}[r]^{\pi_{vec}^{\X}\ \ } & \H_{N,\text{prim}}(\Z)}
\label{eq:diagram of integral level sets}
\end{equation}
where $\longleftrightarrow$ denotes bijection.

\subsubsection{\label{subsec:Maps-between-level}Maps between level sets}

We now define the homeomorphisms $\pi_{\Y_{Q(\sqrt{T}\mathbf{e}_{d})}}:\Y_{Q(\sqrt{T}\mathbf{e}_{d})}(\R)\to\Y_{Q(\mathbf{e}_{d})}(\R)$
and $\pi_{\X_{Q(\sqrt{T}\mathbf{e}_{d})}}:\X_{Q(\sqrt{T}\mathbf{e}_{d})}(\R)\to\X_{Q(\mathbf{e}_{d})}(\R)$,
by using a geometrically natural scaling transformation.

We define $\pi_{\X_{Q(\sqrt{T}\mathbf{e}_{d})}}:\X_{Q(\sqrt{T}\mathbf{e}_{d})}(\R)\to\X_{Q(\mathbf{e}_{d})}(\R)$
by 
\begin{equation}
\pi_{\X_{Q(\sqrt{T}\mathbf{e}_{d})}}(\Lambda,\mathbf{v})\df\left(\frac{1}{T^{1/2(d-1)}}\Lambda,\frac{1}{\sqrt{T}}\mathbf{v}\right),\ (\Lambda,\mathbf{v})\in\X_{Q(\sqrt{T}\mathbf{e}_{d})}(\R).\label{eq:def of pi_X(T)}
\end{equation}
We now give an alternative description of \eqref{eq:def of pi_X(T)}
using the $\SL_{d}(\R)$ action on $\X(\R)$. For $\mathbf{v}\in\H_{Q(\sqrt{T}\mathbf{e}_{d})}(\R)$
we define the unique matrix $S_{T,\mathbf{v}}\in\SL_{d}(\R)$ that
acts by scalar multiplication of a factor $T^{-\frac{1}{2(d-1)}}$
on $P=\mathbf{v}^{\perp}$ and that acts by scalar multiplication
of a factor $T^{1/2}$ on the line $\R\mathbf{v}$. Then, it follows
for $(\Lambda,\mathbf{v})\in\X_{Q(\sqrt{T}\mathbf{e}_{d})}(\R)$ that
\[
\pi_{\X_{Q(\sqrt{T}\mathbf{e}_{d})}}(\Lambda,\mathbf{v})=S_{T,\mathbf{v}}\cdot\left(\Lambda,\mathbf{v}\right)\underbrace{=}_{\text{recalling \eqref{eq:action of sl_d(R) on X(R)}}}\left(S_{T,\mathbf{v}}\Lambda,\o(S_{T,\mathbf{v}})\mathbf{v}\right).
\]

Next, using the matrices $S_{T,\mathbf{v}}$, $\mathbf{v}\in\R^{d},\ T>0$
which were defined above, we define $\pi_{\Y_{Q(\sqrt{T}\mathbf{e}_{d})}}:\Y_{Q(\sqrt{T}\mathbf{e}_{d})}(\R)\to\Y_{Q(\mathbf{e}_{d})}(\R)$
by
\begin{equation}
\pi_{\Y_{Q(\sqrt{T}\mathbf{e}_{d})}}(L,P,\mathbf{v})\df S_{T,\mathbf{v}}\cdot(L,P,\mathbf{v}),\ (L,P,\mathbf{v})\in\Y_{Q(\sqrt{T}\mathbf{e}_{d})}(\R),\label{eq:def of pi_Y_T}
\end{equation}
where $S_{T,\mathbf{v}}\cdot(L,P,\mathbf{v})\underbrace{=}_{\text{recalling \eqref{eq:SLd(R) action on Y(R)}}}(S_{T,\mathbf{v}}L,P,\o(S_{T,\mathbf{v}})\mathbf{v})=(S_{T,\mathbf{v}}L,P,\frac{1}{\sqrt{T}}\mathbf{v})$.
\begin{rem*}
By identifying $\Y(\R)$ as in Section \ref{subsec:-Y(R) as the space of ortiented grids},
we observe that $\pi_{\Y_{Q(\sqrt{T}\mathbf{e}_{d})}}:\Y_{Q(\sqrt{T}\mathbf{e}_{d})}(\R)\to\Y_{Q(\mathbf{e}_{d})}(\R)$
takes the form
\begin{equation}
\pi_{\Y_{Q(\sqrt{T}\mathbf{e}_{d})}}(\Lambda+\mathbf{u},\mathbf{v})\df\left(\frac{1}{T^{1/2(d-1)}}\left(\Lambda+\mathbf{u}\right),\frac{1}{T^{1/2}}\mathbf{v}\right).\label{eq:def of geometric pi_Y_T}
\end{equation}
\end{rem*}
It follows that $\pi_{\X_{T}}$ and $\pi_{\Y_{T}}$ are homeomorphisms
for all $T>0$, and we conclude the following commuting diagram 
\begin{equation}
\xymatrix{\Y_{T}(\R)\ar[r]^{\pi_{\cap}}\ar[d]_{\pi_{\Y_{T}}} & \X_{T}(\R)\ar[d]_{\pi_{\X_{T}}}\\
\Y_{Q(\mathbf{e}_{d})}(\R)\ar[r]^{\pi_{\cap}} & \X_{Q(\mathbf{e}_{d})}(\R)
}
\label{eq:projections diagram}
\end{equation}

\subsubsection{\label{subsec:Measures-as measures on fibre bundles} Measures on
moduli level sets}

As $\Y(\R)$ and $\X(\R)$ are fiber bundles over $\R^{d}\smallsetminus\mathbf{0},$
it follows (by \eqref{eq:def of moduli level sets}) that $\Y_{Q(\mathbf{e}_{d})}(\R)$
and $\X_{Q(\mathbf{e}_{d})}(\R)$ are fiber bundles over the base
space $\H_{Q(\mathbf{e}_{d})}(\R)$. We will now define certain measures
on $\Y_{Q(\mathbf{e}_{d})}(\R)$ and $\X_{Q(\mathbf{e}_{d})}(\R)$
by integrating the natural measures on the fibers of the maps $\left(\pi_{vec}^{\Y}\right)^{-1}(\mathbf{v})$
and $\left(\pi_{vec}^{\X}\right)^{-1}(\mathbf{v})$, with respect
to the measure on the base space $\H_{Q(\mathbf{e}_{d})}(\R)$.

For $\mathbf{v}\in\H_{Q(\mathbf{e}_{d})}(\R)$ we denote by $g_{\mathbf{v}}\in\SL_{d}(\R)$
a matrix satisfying 
\[
\ovec(g_{\mathbf{v}})\underbrace{=}_{\text{recalling \eqref{eq:def of tau}}}\o(g_{\mathbf{v}})\mathbf{e}_{d}=\mathbf{v}.
\]
Then, with the help of diagram \eqref{eq:diagram of coset spaces},
we observe that
\[
\left(\pi_{vec}^{\Y}\right)^{-1}(\mathbf{v})=g_{\mathbf{v}}\ASL_{d-1}(\R)/\ASL_{d-1}(\Z),
\]
and 
\[
\left(\pi_{vec}^{\X}\right)^{-1}(\mathbf{v})=g_{\mathbf{v}}\ASL_{d-1}(\R)/\ASL_{d-1}(\Z)U,
\]
which shows us explicitly the identification of the fibers of $\pi_{vec}^{\Y}$
with
\[
Y_{d-1}\df\ASL_{d-1}(\R)/\ASL_{d-1}(\Z)=\left(\pi_{vec}^{\X}\right)^{-1}(\mathbf{e}_{d}),
\]
and the identification of the fibers of $\pi_{vec}^{\X}$ with 
\begin{equation}
X_{d-1}\df\ASL_{d-1}(\R)/\ASL_{d-1}(\Z)U=\left(\pi_{vec}^{\X}\right)^{-1}(\mathbf{e}_{d}).\label{eq:def of X_d-1}
\end{equation}

We need the following technical definition which describes the normalization
of the Haar measures we will be using.
\begin{defn}
\label{def:weil normalization}Let $G$ be a locally compact second
countable group and let $\Gamma\leq G$ be a lattice. Let $m_{G}$
be a left Haar measure on $G$ and let $m_{G/\Gamma}$ be the unique
left $G$-invariant probability measure on $G/\Gamma$. We say that
$m_{G}$ and $m_{G/\Gamma}$ are \emph{Weil normalized }if for all
$f\in C_{c}(G)$ 
\[
\int_{G}f(x)dm_{G}(x)=\int_{G/\Gamma}\left(\sum_{\gamma\in\Gamma}f(x\gamma)\right)dm_{G/\Gamma}(x\Gamma).
\]
\end{defn}

To define a measure on $\H_{Q(\mathbf{e}_{d})}(\R)$, we recall that
$\SO_{Q}(\R)$ acts transitively on $\H_{Q(\mathbf{e}_{d})}(\R)$
(by Witt's theorem, since we assume $Q(\mathbf{e}_{d})\neq0$) via
the right action \eqref{eq:action of SO_Q on R^d}, which in turns
implies the identification
\[
\H_{Q(\mathbf{e}_{d})}(\R)\cong\mathbf{H}_{\mathbf{e}_{d}}(\R)\backslash\SO_{Q}(\R),
\]
where $\mathbf{H}_{\mathbf{e}_{d}}(\R)\leq\SO_{Q}(\R)$ denotes the
stabilizer of $\mathbf{e}_{d}$. We let $m_{\SO_{Q}(\R)}$ and $m_{\SO_{Q}(\R)/\SO_{Q}(\Z)}$
be Weil normalized, and we define the measure $\mu_{\H_{Q(\mathbf{e}_{d})}(\R)}$
on $\H_{Q(\mathbf{e}_{d})}(\R)$ by 
\[
\mu_{\H_{Q(\mathbf{e}_{d})}(\R)}\df\left(\pi_{\mathbf{H}_{e_{d}}(\R)}\right)_{*}m_{\SO_{Q}(\R)},
\]
where $\pi_{\mathbf{H}_{e_{d}}(\R)}:\SO_{Q}(\R)\to\mathbf{H}_{e_{d}}(\R)\backslash\SO_{Q}(\R)$
is the natural quotient map ($\mu_{\H_{Q(\mathbf{e}_{d})}(\R)}$ is
well defined since we assume that $\mathbf{H}_{\mathbf{e}_{d}}(\R)$
is compact).

We now proceed to define the measures on the fibers $\left(\pi_{vec}^{\Y}\right)^{-1}(\mathbf{v})$
and $\left(\pi_{vec}^{\X}\right)^{-1}(\mathbf{v})$ for $\mathbf{v}\in\H_{Q(\mathbf{e}_{d})}(\R)$.
We let $m_{\ASL_{d-1}(\R)}$ and $m_{Y_{d-1}}$ be Weil normalized,
and we let $m_{X_{d-1}}$ the unique $\ASL_{d-1}(\R)$ invariant measure
on $X_{d-1}$. We define for $\mathbf{v}\in\H_{Q(\mathbf{e}_{d})}(\R)$
the measure $\mu_{\left(\pi_{vec}^{\Y}\right)^{-1}(\mathbf{v})}$
on $\left(\pi_{vec}^{\Y}\right)^{-1}(\mathbf{v})$ by 
\[
\mu_{\left(\pi_{vec}^{\Y}\right)^{-1}(\mathbf{v})}(f)\df\int f(g_{\mathbf{v}}x)dm_{Y_{d-1}}(x),\ \forall f\in C_{c}(\left(\pi_{vec}^{\Y}\right)^{-1}(\mathbf{v}))
\]
and similarly, the measure $\mu_{\left(\pi_{vec}^{\X}\right)^{-1}(\mathbf{v})}$
on $\left(\pi_{vec}^{\X}\right)^{-1}(\mathbf{v})$ by 
\[
\mu_{\left(\pi_{vec}^{\X}\right)^{-1}(\mathbf{v})}(f)\df\int f(g_{\mathbf{v}}x)dm_{X_{d-1}}(x),\ \forall f\in C_{c}(\left(\pi_{vec}^{\X}\right)^{-1}(\mathbf{v})).
\]

We show now that $\mu_{\left(\pi_{vec}^{\Y}\right)^{-1}(\mathbf{v})}$
and $\mu_{\left(\pi_{vec}^{\X}\right)^{-1}(\mathbf{v})}$ are independent
of the choice of $g_{\mathbf{v}}.$ Indeed, if one chooses another
$\tilde{g}_{\mathbf{v}}\in\SL_{d}(\R)$ such that $\ovec(\tilde{g}_{\mathbf{v}})=\mathbf{v}$,
then $\ovec(g_{\mathbf{v}}^{-1}\tilde{g}_{\mathbf{v}})=\mathbf{e}_{d}$,
so that there exists $h\in\ASL_{d-1}(\R)$, such that $\tilde{g}_{\mathbf{v}}=g_{\mathbf{v}}h$.
Therefore we conclude for $\mathcal{M}\in\{X_{d-1},Y_{d-1}\}$ that
\[
\begin{aligned}\int f(\tilde{g}_{\mathbf{v}}x)dm_{\mathcal{M}}(x)= & \int f(g_{\mathbf{v}}hx)dm_{\mathcal{M}}(x)\\
\underbrace{=}_{\text{ \ensuremath{m_{\mathcal{M}}} is \ensuremath{\ASL_{d-1}}(\ensuremath{\R}) invariant}} & \int f(g_{\mathbf{v}}x)dm_{\mathcal{M}}(x).
\end{aligned}
\]

Finally, using the above, we define the following measures on the
spaces $\Y_{Q(\mathbf{e}_{d})}(\R)$ and $\X_{Q(\mathbf{e}_{d})}(\R)$
by
\begin{equation}
\mu_{\Y}\df\int\mu_{\left(\pi_{vec}^{\Y}\right)^{-1}(\mathbf{v})}\ d\mu_{\H_{Q(\mathbf{e}_{d})}(\R)}(\mathbf{v}),\ \text{and }\mu_{\X}\df\int\mu_{\left(\pi_{vec}^{\X}\right)^{-1}(\mathbf{v})}\ d\mu_{\H_{Q(\mathbf{e}_{d})}(\R)}(\mathbf{v}).\label{eq:definition of measure double integral}
\end{equation}

\subsubsection{\label{subsec:Pushforwards}Pushforwards}

We now turn to explain the relation between the measures $\mu_{\Y}$
and $\mu_{\X}$, as well as the connection between $\mu_{\X}$ and
the natural measure on the space of shapes.

Recall that the map $\pi_{\cap}:\Y_{Q(\mathbf{e}_{d})}(\R)\to\X_{Q(\mathbf{e}_{d})}(\R)$
defined in \eqref{eq:def of pi_cap} has compact fibers, hence $\left(\pi_{\cap}\right)_{*}\mu_{\Y}$
is a well defined measure on $\X_{Q(\mathbf{e}_{d})}(\R)$.
\begin{lem}
\label{lem:push of pi_cap}It holds that $\left(\pi_{\cap}\right)_{*}\mu_{\Y}=\mu_{\X}$.
\end{lem}

\begin{proof}
We notice that for all $\mathbf{v}\in\H_{Q(\mathbf{e}_{d})}(\R)$
it holds that $\pi_{\cap}(\left(\pi_{vec}^{\Y}\right)^{-1}(\mathbf{v}))=\left(\pi_{vec}^{\X}\right)^{-1}(\mathbf{v})$,
which shows that for all $\mathbf{v}\in\H_{Q(\mathbf{e}_{d})}(\R)$
the measure $(\pi_{\cap})_{*}\mu_{\left(\pi_{vec}^{\Y}\right)^{-1}(\mathbf{v})}$
is supported on $\left(\pi_{vec}^{\X}\right)^{-1}(\mathbf{v})$. Using
\eqref{eq:definition of measure double integral}, we conclude that
it is sufficient to show 
\begin{equation}
(\pi_{\cap})_{*}\mu_{\left(\pi_{vec}^{\Y}\right)^{-1}(\mathbf{v})}=\mu_{\left(\pi_{vec}^{\X}\right)^{-1}(\mathbf{v})},\ \forall\mathbf{v}\in\H_{Q(\mathbf{e}_{d})}(\R).\label{eq:measure along Y fibers pushed correctly}
\end{equation}
in order to prove $\left(\pi_{\cap}\right)_{*}\mu_{\Y}=\mu_{\X}$.

We let $\mathbf{v}\in\H_{Q(\mathbf{e}_{d})}(\R)$, and we observe
that in terms of cosets, the restriction of $\pi_{\cap}$ to a fiber
$\left(\pi_{vec}^{\Y}\right)^{-1}(\mathbf{v})=g_{\mathbf{v}}\ASL_{d-1}(\R)/\ASL_{d-1}(\Z)$
takes the form 
\[
\pi_{\cap}(g_{\mathbf{v}}\eta\ASL_{d-1}(\Z))=g_{\mathbf{v}}\eta\ASL_{d-1}(\Z)U,\ \ \eta\in\ASL_{d-1}(\R),
\]
(see \eqref{eq:def of pi_cap in coset space}). Since the natural
projection $\ASL_{d-1}(\R)/\ASL_{d-1}(\Z)\to\ASL_{d-1}(\R)/\ASL_{d-1}(\Z)U$
pushes $m_{Y_{d-1}}$ to $m_{X_{d-1}}$, we can deduce \eqref{eq:measure along Y fibers pushed correctly}.
\end{proof}
Next, we recall the space of shapes $\mathcal{S}_{d-1}\df K\backslash X_{d-1}$
(see Section \ref{subsec:The-extension-of shape to X(R)}), and we
consider the product space 
\[
\W\df\mathcal{S}_{d-1}\times\H_{Q(\mathbf{e}_{d})}(\R).
\]
We define the product measure $\mu_{\W}\df\mu_{\mathcal{S}_{d-1}}\otimes\mu_{\H_{Q(\mathbf{e}_{d})}(\R)},$
where $\mu_{\mathcal{S}_{d-1}}$ is the push-forward of $m_{X_{d-1}}$
by a quotient from the left by $K$.

We define the map $(\shape\times\pi_{vec}^{\X}):\X_{Q(\mathbf{e}_{d})}(\R)\to\W$
by 
\begin{equation}
\left(\shape\times\pi_{vec}^{\X}\right)\left(\Lambda,\mathbf{v}\right)\df\left(\text{shape}(\Lambda,\mathbf{v}),\mathbf{v}\right),\label{eq:def of pi^chi_vec=00005Ctimesshape}
\end{equation}
where $\text{shape}(\Lambda,\mathbf{v})$ was defined by \eqref{eq:definition of shape in X(R)}.
As above, the map $\left(\shape\times\pi_{vec}^{\X}\right)$ has compact
fibers.
\begin{lem}
\label{lem:push of shape}We have $\left(\shape\times\pi_{vec}^{\X}\right)_{*}\mu_{\X}=\mu_{\W}$.
\end{lem}

\begin{proof}
Similarly to the proof of Lemma \ref{lem:push of pi_cap}, we observe
that it suffices to show that
\begin{equation}
(\shape)_{*}\mu_{\left(\pi_{vec}^{\X}\right)^{-1}(\mathbf{v})}=\mu_{\mathcal{S}_{d-1}},\ \forall\mathbf{v}\in\H_{Q(\mathbf{e}_{d})}(\R).\label{eq:correct pushforward by shape}
\end{equation}

We now describe $\left(\pi_{vec}^{\X}\right)^{-1}(\mathbf{v})$ in
a more convenient way, which makes the description of $\shape\mid_{\left(\pi_{vec}^{\X}\right)^{-1}(\mathbf{v})}$
more transparent. Fix $\mathbf{v}\in\H_{Q(\mathbf{e}_{d})}(\R)$.
We recall the diagonal matrix (see Section \ref{subsec:The-extension-of shape to X(R)})
\[
d_{\mathbf{\norm v}}\df\left(\begin{array}{cc}
\norm{\mathbf{v}}^{-1/(d-1)}I_{d-1} & \mathbf{0}\\
\mathbf{0} & \norm{\mathbf{v}}
\end{array}\right),
\]
and we let $\rho_{\mathbf{v}}\in\SO_{d}(\R)$ such that $\rho_{\mathbf{v}}^{-1}\mathbf{e}_{d}=\frac{1}{\norm{\mathbf{v}}}\mathbf{v}$.
We denote 
\[
g_{\mathbf{v}}\df\rho_{\mathbf{v}}^{-1}d_{\mathbf{\mathbf{\norm v}}}^{-1},
\]
and we observe that 
\[
\ovec(g_{\mathbf{v}})=\rho_{\mathbf{v}}^{-1}d_{\mathbf{\mathbf{\norm v}}}\mathbf{e}_{d}=\rho_{\mathbf{v}}^{-1}\mathbf{\norm v}\mathbf{e}_{d}=\mathbf{v}.
\]
By using the action of $\SL_{d}(\R)$ on $\X(\R)$ (see \eqref{eq:action of sl_d(R) on X(R)}),
we get $\left(\pi_{vec}^{\X}\right)^{-1}(\mathbf{v})=g_{\mathbf{v}}\cdot\left(\pi_{vec}^{\X}\right)^{-1}(\mathbf{e}_{d})$,
and by recalling \eqref{eq:definition of shape in X(R)}, we see that
$\shape\mid_{\left(\pi_{vec}^{\X}\right)^{-1}(\mathbf{v})}$ takes
the form 
\begin{equation}
\shape(g_{\mathbf{v}}\Lambda,\mathbf{v})\underbrace{=}_{\eqref{eq:definition of shape in X(R)}}Kg_{\mathbf{v}}^{-1}g_{\mathbf{v}}\Lambda=K\Lambda,\ \ \text{ }\forall\Lambda\perp\mathbf{e}_{d}\text{ such that }\covol(\Lambda)=1.\label{eq:the map shape on the fiber}
\end{equation}
Finally, by using \eqref{eq:the map shape on the fiber}, we see that
the function $f_{\mathbf{v}}:X_{d-1}\to\R$, defined by 
\[
f_{\mathbf{v}}(x)\df f\circ\shape(g_{\mathbf{v}}x),\ x\in X_{d-1}=\left(\pi_{vec}^{\X}\right)^{-1}(\mathbf{e}_{d}),
\]
is right $K$ invariant, and by noting that 
\[
\begin{aligned}(\shape)_{*}\mu_{\left(\pi_{vec}^{\X}\right)^{-1}(\mathbf{v})}(f)= & \int f_{\mathbf{v}}(x)dm_{X_{d-1}}(x),\end{aligned}
\]
we obtain \eqref{eq:correct pushforward by shape}.
\end{proof}

\subsection{\label{subsec:Statistics-in-moduli}Statistics in moduli spaces}

We are now able to state our main results for the moduli spaces.

For $N\in\N$ and for $\mathcal{M}\in\left\{ \Y,\X\right\} $, we
define the following measures on $\mathcal{M}_{Q(\mathbf{e}_{d})}(\R)$
by
\[
\nu_{N}^{\mathcal{M}}\df\frac{1}{\left|\H_{N,\text{prim}}(\Z)/\SO_{Q}(\Z)\right|}\sum_{x\in\mathcal{M}_{N}(\Z)}\delta_{\pi_{\mathcal{M}_{N}}(x)},
\]
(to recall $\pi_{\mathcal{M}_{T}}$ see \eqref{eq:def of pi_X(T)}
and \eqref{eq:def of pi_Y_T}), and we define a measure on $\W$ by
\[
\nu_{N}^{\mathcal{\W}}\df\frac{1}{\left|\H_{N,\text{prim}}(\Z)/\SO_{Q}(\Z)\right|}\sum_{\mathbf{v}\in\H_{N,\text{prim}}(\Z)}\delta_{(\text{shape}(\Lambda_{\mathbf{v}},\mathbf{v}),\frac{1}{\sqrt{N}}\mathbf{v})}.
\]

Our first main theorem is as follows.
\begin{thm}
\label{thm:moduli main thm } Assume that $\{T_{n}\}_{n=1}^{\infty}\subseteq\N$
such that $T_{n}\to\infty$ and such that for some fixed odd prime
$p_{0}$, the $\left(Q,p_{0}\right)$ co-isotropic property \emph{(}to
recall see Definition \ref{def:Isotropicity definition}\emph{)} holds.
Then\emph{
\[
\lim_{n\to\infty}\nu_{T_{n}}^{\mathcal{M}}(f)=\mu_{\mathcal{M}}(f),
\]
}where $\mathcal{M}\in\left\{ \Y,\X\right\} ,$ and $f\in C_{c}(\mathcal{M}_{Q(\mathbf{e}_{d})}(\R))$,
or for $\mathcal{M}=\W$ and $f\in C_{c}(\W)$.
\end{thm}

Let $q\in\N$ and recall that $\vartheta_{q}$ denotes the natural
reduction modulo $q$. For $N\in\N$ and for $\mathcal{M}\in\left\{ \Y,\X\right\} $
we define measures on $\mathcal{M}_{Q(\mathbf{e}_{d})}(\R)\times\H_{\vartheta_{q}(T)}(\Z/(q))$
by 
\[
\nu_{N}^{\mathcal{M},q}=\frac{1}{\left|\H_{N,\text{prim}}(\Z)/\SO_{Q}(\Z)\right|}\sum_{x\in\mathcal{M}_{N}(\Z)}\delta_{(\pi_{\mathcal{M}_{N}}(x),\vartheta_{q}(\pi_{vec}^{\mathcal{M}}(x)),}
\]
and similarly a measure on $\mathcal{W}\times\H_{\vartheta_{q}(T)}(\Z/(q))$
by
\[
\nu_{N}^{\mathcal{\W},q}\df\frac{1}{\left|\H_{N,\text{prim}}(\Z)/\SO_{Q}(\Z)\right|}\sum_{\mathbf{v}\in\H_{N,\text{prim}}(\Z)}\delta_{(\text{shape}(\Lambda_{\mathbf{v}},\mathbf{v}),\frac{1}{\sqrt{N}}\mathbf{v},\vartheta_{q}(\mathbf{v}))}.
\]
 By adding some further assumptions on the sequence $\{T_{n}\}_{n=1}^{\infty}$
appearing in Theorem \ref{thm:moduli main thm }, we are able to obtain
the following.
\begin{thm}
\label{thm:moduli_main_thm_with_congruences}Let $q\in2\N+1$\emph{.}
In addition to our \nameref{subsec:Standing-Assumption} on the form
$Q$ assume that $Q$ is non-singular modulo $q$ \emph{(}see Definition
\ref{def:non-singularity modq}\emph{)}. Let $\{T_{n}\}_{n=1}^{\infty}\subseteq\N$
be a sequence of integers satisfying the $\left(Q,p_{0}\right)$ for
some odd prime $p_{0}$ and assume that there is a fixed $a\in\left(\Z/(q)\right)^{\times}$
such that for all $n\in\N$ it holds $\red_{q}\left(T_{n}\right)=a$.
Then\emph{
\[
\lim_{n\to\infty}\nu_{T_{n}}^{\mathcal{M},q}(f)=\mu_{\mathcal{M}}\otimes\mu_{\H_{a}(\Z/(q)}(f),
\]
}where $\mathcal{M}\in\left\{ \Y,\X\right\} $, and $f\in C_{c}(\mathcal{M}_{Q(\mathbf{e}_{d})}(\R)\times\H_{a}(\Z/(q)))$,
or for $\mathcal{M}=\W$ and $f\in C_{c}(\W\times\H_{a}(\Z/(q)))$.
\end{thm}

\subsection{\label{subsec:Proof-of-example from intro}Proof of Theorems \ref{thm:main thm-particular case}
and \ref{thm:main_thm_particular case with congruences}}

We now prove Theorems \ref{thm:main thm-particular case} and \ref{thm:main_thm_particular case with congruences}
by validating the assumptions of Theorems \ref{thm:moduli main thm }
and \ref{thm:moduli_main_thm_with_congruences} for the form $Q_{d}(\mathbf{x})\df x_{d}^{2}-\sum_{i=1}^{d-1}x_{i}^{2}$,
for $d\geq4$. 

Fix $d\geq4$. We observe that the form $Q_{d}$ satisfies our \nameref{subsec:Standing-Assumption},
since $Q_{d}$ is clearly non-degenerate, since $Q_{d}(\mathbf{e}_{d})=1\neq0$
and since $\mathbf{H}_{\mathbf{e}_{d}}(\R)\cong\SO_{d-1}(\R)$ which
is compact.

Since the determinant of $Q_{d}$'s companion matrix is $\pm1$, the
form $Q_{d}$ is non-singular modulo $p$ for any prime $p$.

We now claim that the sequence $\N$ has the $(Q_{d},5)$ co-isotropic
property. Let $\Q_{5}$ be the field of 5-adic numbers. We note that
$\sqrt{-1}\in\Q_{5}$ (by Hensel's lemma, since $2^{2}=-1$ mod $5$)
and we observe that the plane
\[
V\df\text{Span}_{\Q_{5}}\left\{ \sqrt{-1}\mathbf{e}_{2}+\mathbf{e}_{3},\mathbf{e}_{1}+\mathbf{e}_{d}\right\} \subseteq(\Q_{5})^{d},
\]
consists of $Q_{d}$-isotropic vectors. For $N\in\N$ and for $\mathbf{v}\in\H_{N}(\Q)$,
we let $\mathbf{v}^{\perp(Q_{d})}$ be the orthogonal space to $\mathbf{v}$
with respect to $Q_{d}$. Since $\mathbf{v}^{\perp(Q_{d})}\otimes\Q_{5}$
is a $(d-1)$-dimensional subspace of $(\Q_{5})^{d}$, we deduce that
$V\cap\left(\mathbf{v}^{\perp(Q_{d})}\otimes\Q_{5}\right)\neq\{\mathbf{0}\}$.
By the remark below Definition \ref{def:Isotropicity definition}
we deduce that the sequence $\N$ has the $(Q_{d},5)$ co-isotropic
property.

We now verify that $\H_{N,\text{prim}}(\Z)\neq\emptyset$ for all
$N\in\N$. We recall that there exists $\mathbf{u}\in\Z_{\text{prim}}^{3}$
such that 
\begin{equation}
u_{1}^{2}+u_{2}^{2}+u_{3}^{2}=m,\label{eq:sum of threee squares}
\end{equation}
for all positive integers $m\neq0,4,7$ modulo $8$ (see e.g. \cite{Grosswald_sum_of_three_squres}).
Since a square modulo $8$ attains the residues $0,1,4$, for we deduce
that all $N\in\N$ there exists $x_{4}\in\Z$ such that $x_{4}^{2}-N>0$
and such that $x_{4}^{2}-N\neq0,4,7$, which implies by \eqref{eq:sum of threee squares}
that there exists $\mathbf{x}\in\Z_{\text{prim}}^{4}\subseteq\zprim$
such that 
\[
x_{4}^{2}-x_{1}^{2}-x_{2}^{2}-x_{3}^{2}=N.
\]

\section{\label{sec:The-results-forZ imply Y}The results for $\protect\V$
imply the results for $\protect\Y$}

Our goal in this section is to use Theorems \ref{thm:main thm for Z}
- \ref{thm:main_thm_with_congruences-forZ} to deduce Theorems \ref{thm:moduli main thm }
- \ref{thm:moduli_main_thm_with_congruences}. We divide this section
into two parts as follows.
\begin{itemize}
\item Section \ref{subsec:Proof-of-Theorems for Y} proves Theorems \ref{thm:moduli main thm }
- \ref{thm:moduli_main_thm_with_congruences} for $\Y$. This is the
main difficulty in proving Theorems \ref{thm:moduli main thm } -
\ref{thm:moduli_main_thm_with_congruences}.
\item Section \ref{subsec:The-results-for Y imply X,W} gives the proof
for Theorems \ref{thm:moduli main thm } - \ref{thm:moduli_main_thm_with_congruences}
for $\X$ and $\W$, which relies on Section \ref{subsec:Pushforwards}
and Theorems \ref{thm:moduli main thm } - \ref{thm:moduli_main_thm_with_congruences}
for $\mathcal{M}=\Y$.
\end{itemize}

\subsection{\label{subsec:Proof-of-Theorems for Y}Proof of Theorems \ref{thm:moduli main thm }
- \ref{thm:moduli_main_thm_with_congruences} for $\protect\Y$}

We now outline our method for proving Theorems \ref{thm:moduli main thm }
- \ref{thm:moduli_main_thm_with_congruences} for $\mathcal{M}=\Y$
which is based on the result of Theorems \ref{thm:main thm for Z}
- \ref{thm:main_thm_with_congruences-forZ}.

We claim that for all $T>0$ it holds that
\begin{equation}
\Y_{T}(\R)\cong\V_{T}(\R)/\ASL_{d-1}(\Z),\label{eq:Y_T(R) is G orbits in Z_T(R)}
\end{equation}
Indeed, we recall that $\SL_{d}(\R)/\ASL_{d-1}(\Z)$ identifies with
$\Y(\R)$ by the orbit map 
\begin{equation}
\ovec_{\Y}(g\ASL_{d-1}(\Z))\df(g\Z^{d},\text{Span}_{\R}\{g\mathbf{e}_{1},..,g\mathbf{e}_{d-1}\},\ovec(g)),\ g\in\SL_{d}(\R),\label{eq:orbit map tau_Y}
\end{equation}
(see Section \ref{subsec:The-space-of directed folliations}), and
we observe that
\begin{equation}
\begin{aligned}\ovec_{\Y}^{-1}(\Y_{T}(\R))\underbrace{=}_{\text{recalling \eqref{eq:def of moduli level sets}}} & \left\{ g\ASL_{d-1}(\Z)\in\SL_{d}(\R)/\ASL_{d-1}(\Z)\mid\ovec(g)\in\H_{T}(\R)\right\} \\
\underbrace{=}_{\text{recalling \eqref{eq:def of moduli level sets}}} & \V_{T}(\R)/\ASL_{d-1}(\Z).
\end{aligned}
\label{eq:Z_T(R) is  tau_Y^-1(Y_T(R)}
\end{equation}
Similarly, we obtain for all $N\in\N$ that
\begin{equation}
\Y_{N}(\Z)\cong\V_{N}(\Z)/\ASL_{d-1}(\Z).\label{eq:Y_T(Z) is G(Z) orbits in Z_T(Z)}
\end{equation}

Using \eqref{eq:Y_T(R) is G orbits in Z_T(R)}, we can relate the
measure $\mu_{\Y}$ on $\Y_{Q(\mathbf{e}_{d})}(\R)$ to the measure
$\mu_{\V}$ on $\V_{Q(\mathbf{e}_{d})}(\R)$ by using ``unfolding'',
as we will now explain. For $f\in C_{c}(\V_{Q(\mathbf{e}_{d})}(\R))$
we obtain $\bar{f}\in C_{c}(\Y_{Q(\mathbf{e}_{d})}(\R))$ by defining
\begin{equation}
\bar{f}(g\ASL_{d-1}(\Z))\df\sum_{\gamma\in\ASL_{d-1}(\Z)}f(g\gamma),\label{eq:f bar as unfolding of f in Z}
\end{equation}
and, as we show in Section \ref{subsec:Unwinding}, it holds that
the map $f\mapsto\bar{f}$ is onto $C_{c}(\Y_{Q(\mathbf{e}_{d})}(\R))$
and that $\mu_{\V}(f)=\mu_{\Y}(\bar{f})$ for all $f\in C_{c}(\V_{Q(\mathbf{e}_{d})}(\R))$.

Next, recall that $\pi_{\V_{T}}:\V_{T}(\R)\to\V_{Q(\mathbf{e}_{d})}(\R)$
(defined in \eqref{eq:def of pi_Z}), is right $\ASL_{d-1}(\R)$ equivariant,
namely 
\begin{equation}
\pi_{\V_{T}}(g\eta)=\pi_{\V_{T}}(g)\eta,\ \forall g\in\V_{T}(\R),\ \eta\in\ASL_{d-1}(\R).\label{eq:asl_d-1 equivariant of pi_Z_T}
\end{equation}
Using the equivariance of $\pi_{\V_{T}}$ and using \eqref{eq:Y_T(R) is G orbits in Z_T(R)}
we define $\pi_{\Y_{T}}^{Q}:\Y_{T}(\R)\to\Y_{Q(\mathbf{e}_{d})}(\R)$
by 
\begin{equation}
\pi_{\Y_{T}}^{Q}(z\ASL_{d-1}(\Z))\df\pi_{\V_{T}}(z)\ASL_{d-1}(\Z).\label{eq:def of pi_T^Q}
\end{equation}

The main reason for introducing $\pi_{\Y_{T}}^{Q}$ is that by assuming
the asymptotics of the form 
\[
\sum_{g\in\V_{N}(\Z)}f(\pi_{\V_{N}}(g))\sim c(T)\mu_{\V}(f),\ \text{as \ensuremath{N\to\infty},}
\]
we are able to obtain the asymptotics 
\[
\sum_{y\in\Y_{N}(\Z)}\bar{f}(\pi_{\Y_{N}}^{Q}(y))\sim c(T)\mu_{\Y}(\bar{f}),\ \text{as \ensuremath{N\to\infty},}
\]
by observing that 
\[
\begin{aligned}\sum_{g\in\V_{N}(\Z)}f(\pi_{\V_{N}}(g))= & \sum_{g\ASL_{d-1}(\Z)\in\V_{N}(\Z)/\ASL_{d-1}(\Z)}\ \sum_{\gamma\in\ASL_{d-1}(\Z)}f(\pi_{\V_{N}}(g\gamma))\\
\underbrace{=}_{\eqref{eq:asl_d-1 equivariant of pi_Z_T}} & \sum_{g\ASL_{d-1}(\Z)\in\V_{N}(\Z)/\ASL_{d-1}(\Z)}\ \sum_{\gamma\in\ASL_{d-1}(\Z)}f(\pi_{\V_{N}}(g)\gamma)\\
\underbrace{=}_{\eqref{eq:Y_T(Z) is G(Z) orbits in Z_T(Z)}} & \sum_{y\in\Y_{N}(\Z)}\bar{f}(\pi_{\Y_{N}}^{Q}(y)),
\end{aligned}
\]
and by using that $\mu_{\V}(f)=\mu_{\Y}(\bar{f})$.

However, we are interested in proving Theorems \ref{thm:moduli main thm }
- \ref{thm:moduli_main_thm_with_congruences} for $\mathcal{M}=\Y$
which concern the asymptotics of averages of the form
\[
\sum_{y\in\Y_{N}(\Z)}\bar{f}(\pi_{\Y_{N}}(y)),\ \text{as \ensuremath{N\to\infty},}
\]
where $\pi_{\Y_{T}}:\Y_{T}(\R)\to\Y_{Q(\mathbf{e}_{d})}(\R)$ was
defined in \eqref{eq:def of geometric pi_Y_T}. Fortunately, it turns
out that $\pi_{\Y_{T}}^{Q}$ and $\pi_{\Y_{T}}$ differ asymptotically
uniformly by a fixed map that preserves the measure $\mu_{\Y}$, allowing
us to prove Theorems \ref{thm:moduli main thm } - \ref{thm:moduli_main_thm_with_congruences}.
\begin{rem*}
Observe that the right $\SO_{Q}(\R)$-actions on $\Y_{Q(\mathbf{e}_{d})}(\R)$
and on $\Y_{T}(\R)$ given by
\[
(L,P,\mathbf{v})\cdot\rho\df\left(\o(\rho^{-1})L,\o(\rho^{-1})P,\rho^{-1}\mathbf{v}\right),\ (L,P,\mathbf{v})\in\Y_{s}(\R),\ \mathbf{\rho\in\SO}_{Q}(\R),
\]
are equivariant with respect to the map $\pi_{T}^{Q}$. Yet, as we
will see\textbf{ }in Section \ref{subsec:The-difference-of geom proj on Y and equiv proj on Y},
this statement is wrong in general for $\pi_{\Y_{T}}.$
\end{rem*}
The structure of the rest of the section is as follows:
\begin{itemize}
\item Section \ref{subsec:Unwinding} relates the measure $\mu_{\Y}$ and
$\mu_{\V}$ by ``unfolding''.
\item Section \ref{subsec:The-difference-of geom proj on Y and equiv proj on Y}
compares $\pi_{\Y_{T}}^{Q}$ and $\pi_{\Y_{T}}$.
\item Section \ref{subsec:Concluding-the-proof that Y follows from Z} proves
Theorems \ref{thm:moduli main thm } - \ref{thm:moduli_main_thm_with_congruences}
for $\mathcal{M}=\Y$.
\end{itemize}

\subsubsection{\label{subsec:Unwinding}Unfolding the measure on $\text{\ensuremath{\protect\Y}}_{Q(\mathbf{e}_{d})}(\protect\R)$}

To relate the measure $\mu_{\V}$ on $\V_{Q(\mathbf{e}_{d})}(\R)$
(defined in Section \ref{subsec:Statistics-of-Z}) with $\mu_{\Y}$
(defined in Section \ref{subsec:Measures-as measures on fibre bundles}),
we now give $\mu_{\V}$ a different description, which is conceptually
similar to the definition of $\mu_{\Y}$. We observe that $\ovec:\V_{Q(\mathbf{e}_{d})}(\R)\to\H_{Q(\mathbf{e}_{d})}(\R)$
endows $\V_{Q(\mathbf{e}_{d})}(\R)$ with a fiber bundle structure
over $\H_{Q(\mathbf{e}_{d})}(\R)$ with fibers being right $\ASL_{d-1}(\R)$
cosets (to recall $\ovec$, see \eqref{eq:def of tau}). As for $\mu_{\Y},$
we define for each $\mathbf{v}\in\H_{Q(\mathbf{e}_{d})}(\R)$ a measure
on the fiber $\ovec^{-1}(\mathbf{v})$ by
\[
\mu_{\ovec^{-1}(\mathbf{v})}(f)\df\int f(g_{\mathbf{v}}x)dm_{\ASL_{d-1}(\R)}(x),\ \mathbf{v}\in\H_{Q(\mathbf{e}_{d})}(\R),\ f\in C_{c}(\ovec^{-1}(\mathbf{v})),
\]
where $g_{\mathbf{v}}\in\SL_{d}(\R)$ is chosen such that $\ovec(g_{\mathbf{v}})=\mathbf{v}$.
By integrating the measures on the fibers we define the measure $\nu_{\V}$
on $\V_{Q(\mathbf{e}_{d})}(\R)$ by
\begin{equation}
\nu_{\V}\df\int\mu_{\ovec^{-1}(\mathbf{v})}\ d\mu_{\H_{Q(\mathbf{e}_{d})}(\R)}(\mathbf{v}).\label{eq:definiton of measure on Z as integral on fibers}
\end{equation}

We obtain the lemma below which we leave the reader to verify.
\begin{lem}
\label{lem:pushforward of polar meas}It holds that $\nu_{\mathcal{\V}}=\mu_{\V}$,
where $\mu_{\V}$ was defined in \eqref{eq:def of mu_Z}.
\end{lem}

The unfolding relation between $\mu_{\Y}$ and $\mu_{\V}$ is given
by the following lemma.
\begin{lem}
\label{lem:integral unwinding of z w.r.t y}For all $f\in C_{c}(\V_{Q(\mathbf{e}_{d})}(\R))$
it holds that $\mu_{\V}(f)=\mu_{\Y}(\bar{f}),$ where $\bar{f}$ is
given by \eqref{eq:f bar as unfolding of f in Z}.
\end{lem}

\begin{proof}
Using Lemma \ref{lem:pushforward of polar meas} and by recalling
the definition of $\mu_{\Y}$ in \eqref{eq:definition of measure double integral},
we see that it is sufficient to prove that $\mu_{\ovec^{-1}(\mathbf{v})}(f)=\mu_{\left(\pi_{vec}^{\Y}\right)^{-1}(\mathbf{v})}(\bar{f})$
for all $\mathbf{v}\in\H_{Q(\mathbf{e}_{d})}(\R)$. Let $g_{\mathbf{v}}\in\SL_{d}(\R)$
such that $\ovec(g_{\mathbf{v}})=\mathbf{v}$, and recall that $m_{\mathcal{\ASL}_{d-1}(\R)}$
and $m_{Y_{d-1}}$ are Weil normalized (see Definition \ref{def:weil normalization}).
Then,

\begin{align*}
\mu_{\left(\pi_{vec}^{\Y}\right)^{-1}(\mathbf{v})}(\bar{f})= & \int\left(\sum_{\gamma\in\ASL_{d-1}(\Z)}f(g_{\mathbf{v}}x\gamma)\right)dm_{Y_{d-1}}(x\ASL_{d-1}(\Z))\\
= & \int f(g_{\mathbf{v}}x)dm_{\ASL_{d-1}(\R)}(x)\\
= & \mu_{\ovec^{-1}(\mathbf{v})}(f).
\end{align*}
\end{proof}
We now turn to show that for all $T>0$ the map $\bar{*}:C_{c}(\V_{T}(\R))\to C_{c}(\Y_{T}(\R))$
defined by $f\mapsto\bar{f}$ is onto (to recall $\bar{f}$ see \eqref{eq:f bar as unfolding of f in Z}).
To prove the latter, we note the following general lemma.
\begin{lem}
\label{lem:star bar map is onto}Let $G$ be a locally compact, second
countable group, $K\leq G$ be compact, and $\Gamma\leq G$ be discrete.
Then the map 
\[
\bar{*}:C_{c}(K\backslash G)\to C_{c}(K\backslash G/\Gamma)
\]
defined by $\bar{f}(Kg\Gamma)\df\sum_{\gamma\in\Gamma}f(Kg\gamma)$
is onto.
\end{lem}

\begin{proof}
We let $\pi_{K}:G/\Gamma\to K\backslash G/\Gamma$ be the natural
map. Since $K$ is compact, for $\ef\in C_{c}(K\backslash G/\Gamma)$
it holds that $\ef\circ\pi_{K}\in C_{c}(G/\Gamma)$. We recall that
\cite[Proposition 2.50]{Folland_harmonic} tells us there exists $\tilde{f}\in C_{c}(G)$
such that
\[
\ef\circ\pi_{K}(g\Gamma)=\sum_{\gamma\in\Gamma}\tilde{f}(g\gamma).
\]
We let $m_{K}$ be the Haar probability measure on $K$ and we observe
that 
\[
\begin{aligned}\ef\circ\pi_{K}(g\Gamma)= & \int\ef\circ\pi_{K}(kg\Gamma)dm_{K}(k)\\
= & \int\left(\sum_{\gamma\in\Gamma}\tilde{f}(kg\gamma)\right)dm_{K}(k)\\
= & \sum_{\gamma\in\Gamma}\int\tilde{f}(kg\gamma)dm_{K}(k).
\end{aligned}
\]
where in the last line we used that for all $g\in G$, the sum $\sum_{\gamma\in\Gamma}\tilde{f}(kg\gamma)$
is a finite sum, where the number of summands is bounded uniformly
in $k\in K$ (this follows by Lemma \ref{lem:uniform bound intersection with cpct set}).
The proof is complete by denoting $f(Kg)\df\int\tilde{f}(kg)dm_{K}(k)$
and by observing that $f\in C_{c}(K\backslash G)$.
\end{proof}
Let $G\df(\SO_{Q}\times\ASL_{d-1})(\R)$, $K\df H$ which was defined
in \eqref{eq:def_of_H}, and $\Gamma\df\{\mathbf{e}\}\times\ASL_{d-1}(\Z)\leq(\SO_{Q}\times\ASL_{d-1})(\R)$.
Lemma \ref{lem:double coset structure for Y} below shows that $\Y_{T}(\R)\cong K\backslash G/\Gamma$.
Since $\V_{T}(\R)\cong K\backslash G$, the proof that $\bar{*}:C_{c}(\V_{T}(\R))\to C_{c}(\Y_{T}(\R))$
is onto will be done by Lemma \ref{lem:star bar map is onto} and
Lemma \ref{lem:double coset structure for Y}.
\begin{lem}
\label{lem:double coset structure for Y}For all $T>0$, \emph{$H\backslash(\SO_{Q}\times\ASL_{d-1})(\R)/\ASL_{d-1}(\Z)$}
is homeomorphic to $\Y_{Q(\sqrt{T}\mathbf{e}_{d})}(\R)$, by the map
\emph{
\[
\tilde{\Phi}(H(\rho,\eta)\ASL_{d-1}(\Z))=\ovec_{\Y}\left(\o(\rho^{-1})a_{T}\eta\ASL_{d-1}(\Z)\right),
\]
}where $\ovec_{\Y}$ is given by \eqref{eq:orbit map tau_Y} and \emph{$a_{T}\in\SL_{d}(\R)$}
is given by Definition \ref{def:def of a_T}.
\end{lem}

\begin{proof}
We recall that $\V_{Q(\sqrt{T}\mathbf{e}_{d})}(\R)$ is identified
with $H\backslash(\SO_{Q}\times\ASL_{d-1})(\R)$ by the map 
\[
\Phi(H(\rho,\eta))=\o(\rho^{-1})a_{T}\eta,
\]
(to recall, see \eqref{eq:def_of_H} defining $H$, and see below
\eqref{eq:polar coordinates for z_T}) which shows that 
\[
\V_{Q(\sqrt{T}\mathbf{e}_{d})}(\R)/\ASL_{d-1}(\Z)\cong H\backslash(\SO_{Q}\times\ASL_{d-1})(\R)/\ASL_{d-1}(\Z)
\]
by the map 
\[
\tilde{\Phi}(H(\rho,\eta)\ASL_{d-1}(\Z))=\o(\rho^{-1})a_{T}\eta\ASL_{d-1}(\Z).
\]
Because $\Y_{T}(\R)$ is identified with $\V_{T}(\R)/\ASL_{d-1}(\Z)$
for all $T>0$ via $\ovec_{\Y}$ (see below \eqref{eq:Y_T(R) is G orbits in Z_T(R)}),
the proof is complete.
\end{proof}

\subsubsection{\label{subsec:The-difference-of geom proj on Y and equiv proj on Y}Comparing
of $\pi_{\protect\Y_{T}}$ and $\pi_{\protect\Y_{T}}^{Q}$}

We will now discuss the difference between $\pi_{\Y_{T}}$ and $\pi_{\Y_{T}}^{Q}$
with the goal of showing that it converges as $T\to\infty$ in a certain
uniform way to a fixed map that preserves the measure on $\Y_{Q(\mathbf{e}_{d})}(\R)$.

We recall that for $T>0$ and $(L,P,\mathbf{v})\in\Y_{Q(\sqrt{T}\mathbf{e}_{d})}(\R)$,
\begin{equation}
\pi_{\Y_{Q(\sqrt{T}\mathbf{e}_{d})}}(L,P,\mathbf{v})=(S_{T,\mathbf{v}}L,P,\frac{1}{\sqrt{T}}\mathbf{v}),\ \text{}\label{eq:geometric projection in terms of matricre}
\end{equation}
where $S_{T,\mathbf{v}}\in\SL_{d}(\R)$ acts by scalar multiplication
of a factor $T^{-\frac{1}{2(d-1)}}$ on $P=\mathbf{v}^{\perp}$ and
acts on the line $\R\mathbf{v}$ by scalar multiplication by a factor
$T^{1/2}$ (see Section \ref{subsec:Maps-between-level}).

Next, we describe $\pi_{\Y_{T}}^{Q}$ in a manner similar to \eqref{eq:geometric projection in terms of matricre}.
\begin{defn}
\label{def:of S_T,v^Q}Recall the form $Q^{*}$ defined in \eqref{eq:def of Q^*}.
For $\mathbf{v}\in\R^{d}\smallsetminus\mathbf{0}$ such that $Q(\mathbf{v})>0,$
we denote by $\mathbf{v}_{Q}\in\R^{d}\smallsetminus\mathbf{0}$ the
unique vector orthogonal with respect to the form $Q^{*}$ to the
hyperplane $\mathbf{v}^{\perp}$, having the normalization
\end{defn}

\[
\mathbf{v}_{Q}=\mathbf{v}+\hat{\mathbf{v}}_{Q},
\]
where $\hat{\mathbf{v}}_{Q}\in\mathbf{v}^{\perp}$. We define $S_{T,\mathbf{v}}^{Q}\in\SL_{d}(\R)$
which acts by scalar multiplication of a factor $T^{-\frac{1}{2(d-1)}}$
on the hyperplane $\mathbf{v}^{\perp}$ and which acts on $\R\mathbf{v}_{Q}$
(the orthogonal line to the hyperplane $\mathbf{v}^{\perp}$ with
respect to the form $Q^{*}$) by scalar multiplication of a factor
$T^{1/2}$.
\begin{rem*}
We observe that $\mathbf{v}_{Q}=\frac{1}{Q(\mathbf{v})}M\mathbf{v},$
where $M$ is the companion matrix of the form $Q$. This implies
that the map \textbf{$\mathbf{v}\mapsto\mathbf{v}_{Q}$ }is continuous.
\end{rem*}
\begin{lem}
For all $T>0$ it holds that 
\begin{equation}
\pi_{\Y_{Q(\sqrt{T}\mathbf{e}_{d})}}^{Q}(L,P,\mathbf{v})\df(S_{T,\mathbf{v}}^{Q}L,P,\frac{1}{\sqrt{T}}\mathbf{v}),\ \forall(L,P,\mathbf{v})\in\Y_{Q(\sqrt{T}\mathbf{e}_{d})}(\R).\label{eq:equiv projection in terms of matrices}
\end{equation}
\end{lem}

\begin{proof}
Let $(L,\mathbf{v}^{\perp},\mathbf{v})\in\Y_{Q(\sqrt{T}\mathbf{e}_{d})}(\R)$.
By using the identification \eqref{eq:Y_T(R) is G orbits in Z_T(R)},
we take $g\in\V_{T}(\R)$ such that $(L,\mathbf{v}^{\perp},\mathbf{v})=\ovec_{\Y}(g\ASL_{d-1}(\Z))$.

Using \eqref{eq:polar coordinates for z_T}, we take $(\rho,\eta)\in(\SO_{Q}\times\ASL_{d-1})(\R)$
such that 
\[
g=\o(\rho^{-1})a_{T}\eta,
\]
and we observe that 
\begin{equation}
\begin{aligned}\pi_{\Y_{Q(\sqrt{T}\mathbf{e}_{d})}}^{Q}(L,\mathbf{v}^{\perp},\mathbf{v})\underbrace{=}_{\text{recalling \eqref{eq:def of pi_T^Q}}} & \ovec_{\Y}(\pi_{\V_{Q(\sqrt{T}\mathbf{e}_{d})}}(g)\ASL_{d-1}(\Z))\\
\underbrace{=}_{\text{recalling \eqref{eq:def of pi_Z}}} & \ovec_{\Y}(\o(\rho^{-1})a_{T}^{-1}\o(\rho)g\ASL_{d-1}(\Z)).
\end{aligned}
\label{eq:rewriting pi_Y_^Q}
\end{equation}
By recalling that $\o(\rho^{-1})\in\SO_{Q^{*}}(\R)$ (see Lemma \ref{lem:F(SO_Q) and stabilizer of a vector})
and by recalling the definition of $a_{T}$ (see Definition \ref{def:def of a_T}),
we deduce that 
\[
\o(\rho^{-1})a_{T}^{-1}\o(\rho)=S_{T,\mathbf{v}}^{Q},
\]
where $S_{T,\mathbf{v}}^{Q}$ was given in Definition \ref{def:of S_T,v^Q}.
Then by \eqref{eq:rewriting pi_Y_^Q}, 
\[
\begin{aligned}\pi_{\Y_{Q(\sqrt{T}\mathbf{e}_{d})}}^{Q}(L,\mathbf{v}^{\perp},\mathbf{v})\underbrace{=}_{\text{recalling \eqref{eq:orbit map tau_Y}}} & (S_{T,\mathbf{v}}^{Q}L,S_{T,\mathbf{v}}^{Q}\mathbf{v}^{\perp},\o(S_{T,\mathbf{v}}^{Q})\mathbf{v})\\
= & (S_{T,\mathbf{v}}^{Q}L,\mathbf{v}^{\perp},\frac{1}{\sqrt{T}}\mathbf{v}).
\end{aligned}
\]
\end{proof}
\begin{lem}
\label{lem:uniform convergence to a unipotent}Let $(L,P,\mathbf{v})\in\Y_{Q(\mathbf{e}_{d})}(\R)$,
and consider the unipotent matrix $u_{\mathbf{v}}^{Q}$ which satisfies
that $u_{\mathbf{v}}^{Q}\mathbf{v}=\mathbf{v}_{Q}$ and acts as identity
on $\mathbf{v}^{\perp}$. Let 
\begin{equation}
u_{T,\mathbf{v}}^{Q}\df\left(S_{T,\mathbf{v}}^{Q}\right)S_{T,\mathbf{v}}^{-1},\label{eq:u_T,Q,v}
\end{equation}
then
\[
\lim_{T\to\infty}\left(u_{\mathbf{v}}^{Q}\right)^{-1}u_{T,\mathbf{v}}^{Q}=I_{d},
\]
and the convergence is uniform when $\mathbf{v}$ is restricted to
a compact subset of $\R^{d}\smallsetminus\mathbf{0}.$
\end{lem}

\begin{proof}
It is easy to verify that $\left(S_{T,\mathbf{v}}^{Q}\right)S_{T,\mathbf{v}}^{-1}$
acts as identity on $\mathbf{v}^{\perp}$, namely $\left(S_{T,\mathbf{v}}^{Q}\right)S_{T,\mathbf{v}}^{-1}$
and $u_{\mathbf{v}}^{Q}$ agree on $\mathbf{v}^{\perp}.$ Next,

\begin{align*}
\left(S_{T,\mathbf{v}}^{Q}\right)S_{T,\mathbf{v}}^{-1}\mathbf{v}= & S_{T,\mathbf{v}}^{Q}\left(\frac{1}{\sqrt{T}}\mathbf{v}\right)\\
= & \frac{1}{\sqrt{T}}S_{T,\mathbf{v}}^{Q}\left(\mathbf{v}_{Q}-\hat{\mathbf{v}}_{Q}\right)\\
= & \mathbf{v}_{Q}-T^{-\frac{d}{2(d-1)}}\hat{\mathbf{v}}_{Q},
\end{align*}
namely 
\[
\left(u_{\mathbf{v}}^{Q}-\left(S_{T,\mathbf{v}}^{Q}\right)S_{T,\mathbf{v}}^{-1}\right)\mathbf{v}=-T^{-\frac{d}{2(d-1)}}\hat{\mathbf{v}}_{Q}.
\]
By multiplying both sides of the preceding equality by $\left(u_{\mathbf{v}}^{Q}\right)^{-1}$,
and by recalling \eqref{def:def of a_T}, we get
\[
\left(I_{d}-\left(u_{\mathbf{v}}^{Q}\right)^{-1}u_{T,\mathbf{v}}^{Q}\right)\mathbf{v}=-T^{-\frac{d}{2(d-1)}}\hat{\mathbf{v}}_{Q}.
\]
Since the map $\mathbf{v}\mapsto\mathbf{v}_{Q}$ is continuous (see
remark below Definition \ref{def:of S_T,v^Q}), we deduce that $\lim_{T\to\infty}\left(u_{\mathbf{v}}^{Q}\right)^{-1}u_{T,\mathbf{v}}^{Q}=I_{d}$
converges uniformly when $\mathbf{v}$ varies in a compact set of
$\R^{d}\smallsetminus\mathbf{0}$.
\end{proof}
Now let $u^{Q}:\Y_{Q(\mathbf{e}_{d})}(\R)\to\Y_{Q(\mathbf{e}_{d})}(\R)$
be defined by $u^{Q}(L,P,\mathbf{v})\df(u_{\mathbf{v}}^{Q}L,P,\mathbf{v}).$
\begin{lem}
\label{lem:U^Q preserves mu_Y}The map $u^{Q}$ preserves the measure
$\mu_{\Y}$ on $\Y_{Q(\mathbf{e}_{d})}(\R)$.
\end{lem}

\begin{proof}
Let $f\in C_{c}(\Y_{Q(\mathbf{e}_{d})}(\R))$. By recalling the definition
of $\mu_{\Y}$ in \eqref{eq:definition of measure double integral},
it is sufficient to prove that $\mu_{\left(\pi_{vec}^{\Y}\right)^{-1}(\mathbf{v})}(f\circ u^{Q})=\mu_{\left(\pi_{vec}^{\Y}\right)^{-1}(\mathbf{v})}(f)$
for all $\mathbf{v}\in\H_{Q(\mathbf{e}_{d})}(\R)$. Let $\mathbf{v}\in\H_{Q(\mathbf{e}_{d})}(\R)$
and let $g_{\mathbf{v}}\in\SL_{d}(\R)$ such that $\ovec(g_{\mathbf{v}})=\mathbf{v}$.
Then 
\[
\begin{aligned}\mu_{\left(\pi_{vec}^{\Y}\right)^{-1}(\mathbf{v})}(f\circ u^{Q})= & \int f(u_{\mathbf{v}}^{Q}g_{\mathbf{v}}x)dm_{Y_{d-1}}(x)\\
= & \int f(g_{\mathbf{v}}(g_{\mathbf{v}}^{-1}u_{\mathbf{v}}^{Q}g_{\mathbf{v}})x)dm_{Y_{d-1}}(x).
\end{aligned}
\]
As the reader may verify, it follows that $g_{\mathbf{v}}^{-1}u_{\mathbf{v}}^{Q}g_{\mathbf{v}}\in\ASL_{d-1}(\R)$,
and by recalling that $m_{Y_{d-1}}$ is left $\ASL_{d-1}(\R)$ invariant,
the proof is done.
\end{proof}
Consider $\delta_{T}^{Q}:\Y_{Q(\mathbf{e}_{d})}(\R)\to\Y_{Q(\mathbf{e}_{d})}(\R)$
defined by $\delta_{T}^{Q}\df\left(u^{Q}\right)^{-1}\circ\pi_{\Y_{T}}^{Q}\circ\pi_{\Y_{T}}^{-1}$.
Using Lemma \ref{lem:uniform convergence to a unipotent}, we obtain
that $\delta_{T}^{Q}$ converges to the identity transformation on
$\Y_{Q(\mathbf{e}_{d})}(\R)$ as $T\to\infty$ in the following uniform
manner.
\begin{cor}
\label{cor:uniform conv of delta^Q_T_n}Assume that $y_{n}\to y_{0}$
in $\Y_{Q(\mathbf{e}_{d})}(\R)$ and let $\{T_{n}\}_{n=1}^{\infty}\subseteq\R_{>0}$
such that $T_{n}\to\infty$. Then $\delta_{T_{n}}^{Q}(y_{n})\to y_{0}$
and $\left(\delta_{T_{n}}^{Q}\right)^{-1}(y_{n})\to y_{0}$.
\end{cor}

\begin{proof}
We write $y_{n}=\left(L_{n},P_{n},\mathbf{v}_{n}\right)$ and $y_{0}=(L_{0},P_{0},\mathbf{v}_{0})$,
and we observe that $y_{n}\to y_{0}$ implies that $L_{n}\to L_{0}$
and $\mathbf{v}_{n}\to\mathbf{v}_{0}$ in the usual topology of $X_{d}$
and $\R^{d}$ correspondingly.

Let for $T_{n}>0$ and $\mathbf{v}_{n}\in\R^{d}\smallsetminus\mathbf{0},$
let $I_{T_{n},\mathbf{v}_{n}}\in\SL_{d}(\R)$ be defined by
\[
I_{T_{n},\mathbf{v}_{n}}\df\left(u_{\mathbf{v}_{n}}^{Q}\right)^{-1}\left(S_{T,\mathbf{v}_{n}}^{Q}\right)S_{T,\mathbf{v}_{n}}^{-1},
\]
and observe that 
\[
\delta_{T_{n}}^{Q}(L_{n},P_{n},\mathbf{v}_{n})=(I_{T_{n},\mathbf{v}_{n}}L_{n},P_{n},\mathbf{v}_{n}).
\]
Since $L_{n}\to L_{0}$, since $\mathbf{v}_{n}\to\mathbf{v}_{0}$
and since $I_{T_{n},\mathbf{v}_{n}}\to I_{d}$ uniformly when $\mathbf{v}$
is restricted to a compact subset of $\R^{d}\smallsetminus\mathbf{0}$
(by Lemma \ref{lem:uniform convergence to a unipotent}), we conclude
that $I_{T_{n},\mathbf{v}_{n}}L_{n}\to L_{0}$, which shows $\delta_{T_{n}}^{Q}(y_{n})\to y_{0}$.

Similarly, we have that 
\[
\left(\delta_{T_{n}}^{Q}\right)^{-1}(L_{n},P_{n},\mathbf{v}_{n})=(I_{T_{n},\mathbf{v}_{n}}^{-1}L_{n},P_{n},\mathbf{v}_{n}),
\]
and since $I_{T_{n},\mathbf{v}_{n}}^{-1}\to I_{d}$ converges uniformly
when $\mathbf{v}$ is restricted to a compact subset of $\R^{d}\smallsetminus\mathbf{0}$
(which follows by Lemma \ref{lem:uniform convergence to a unipotent}),
we also obtain that $\left(\delta_{T_{n}}^{Q}\right)^{-1}(y_{n})\to y_{0}.$
\end{proof}
\begin{lem}
\label{lem:uniform convergence in manifold}Let $X$ be a manifold
and assume that $\left\{ \ef_{T}\right\} _{T\in\R_{>0}}$ is a family
of bijections $\ef_{T}:X\to X$ such that for any sequence $\left\{ x_{n}\right\} \subseteq X$
with $\lim_{n\to\infty}x_{n}=x_{0}$ and any $\{T_{n}\}_{n=1}^{\infty}\subseteq\R_{>0}$
such that $T_{n}\to\infty$ it holds that $\lim_{n\to\infty}\ef_{T_{n}}(x_{n})=x_{0}$
and $\lim_{n\to\infty}\ef_{T_{n}}^{-1}(x_{n})=x_{0}$ . Then for all
$f\in C_{c}(X)$, $f\circ\ef_{T}$ converges to $f$ uniformly. Namely,
for all $f\in C_{c}(X)$ and all $\epsilon>0$ there is $T_{0}>0$
such that 
\[
\left|f\circ\ef_{T}(x)-f(x)\right|<\epsilon,\ \forall T>T_{0},\ \forall x\in X.
\]
\end{lem}

\begin{proof}
Let $f\in C_{c}(X)$ and assume for contradiction that $f\circ\ef_{T}$
doesn't converge uniformly to $f$. Then there exists a $\delta>0$,
a sequence $T_{n}\to\infty$ and a sequence $\left\{ x_{n}\right\} \subseteq X$
such that $\left|f\circ\ef_{T_{n}}(x_{n})-f(x_{n})\right|>\delta$
for all $n\in\N.$ Let $K\df\text{supp}(f)$ and observe by the preceding
inequality that either $\ef_{T_{n}}(x_{n})\in K$ infinitely often
or $x_{n}\in K$ infinitely often. Assume that $\ef_{T_{n}}(x_{n})\in K$
infinitely often. By sequential compactness we may assume that $\ef_{T_{n}}(x_{n})\to x_{0}$
which implies by assumption on $\ef_{T}^{-1}$ that $x_{n}=\ef_{T_{n}}^{-1}(\ef_{T_{n}}(x_{n}))\to x_{0}$.
We reach a contradiction since 
\[
\left|f\circ\ef_{T_{n}}(x_{n})-f(x_{n})\right|\leq\left|f\circ\ef_{T_{n}}(x_{n})-f(x_{0})\right|+\left|f(x_{0})-f(x_{n})\right|,
\]
and since the continuity of $f$ implies $\left|f\circ\ef_{T_{n}}(x_{n})-f(x_{0})\right|\to0$
and $\left|f(x_{0})-f(x_{n})\right|\to0$.

In a manner similar to the preceding, we obtain a contradiction when
assuming that $x_{n}\in K$ infinitely often.
\end{proof}
\begin{cor}
\label{cor:uniform left u^Q continuity for cont functions}Let $f\in C_{c}(\mathcal{\Y}_{Q(\mathbf{e}_{d})}(\R)\times\H_{a}(\Z/(q)))$,
let $\epsilon>0$ and let $\mathcal{K}\supsetneq\text{Supp}(f)$ be
an open precompact set. Then, there exists $T_{0}>0$ such that for
all $T>T_{0}$ the following hold
\begin{enumerate}
\item \label{enu:uniformity of difference of  values of function composed with projections}$\left|f((u^{Q})^{-1}\circ\pi_{\Y_{T}}^{Q}(y),\mathbf{v})-f(\pi_{\Y_{T}}(y),\mathbf{v})\right|<\epsilon,\ \forall\left(y,\mathbf{v}\right)\in\Y_{T}(\R)\times\H_{a}(\Z/(q)).$
\item \label{enu:out K implies out supp}if $((u^{Q})^{-1}\circ\pi_{\Y_{T}}^{Q}(y),\mathbf{v})\notin\mathcal{K}$,
then $(\pi_{\Y_{T}}(y),\mathbf{v})\notin\text{Supp}(f)$.
\end{enumerate}
\end{cor}

\begin{proof}
Let $f\in C_{c}(\mathcal{\Y}_{Q(\mathbf{e}_{d})}(\R)\times\H_{a}(\Z/(q)))$
and let $\epsilon\in(0,1)$. Using Corollary \ref{cor:uniform conv of delta^Q_T_n}
and Lemma \ref{lem:uniform convergence in manifold} with the fact
that $\H_{a}(\Z/(q))$ is a finite set, we obtain $T_{1}>0$ such
that for all $T>T_{1}$ it holds 
\[
\left|f\left(\delta_{T}^{Q}(y'),\mathbf{v})\right)-f\left(y',\mathbf{v})\right)\right|<\epsilon,\ \forall\left(y',\mathbf{v}\right)\in\Y_{Q(\mathbf{e}_{d})}(\R)\times\H_{a}(\Z/(q)).
\]
Then, by substituting $y'=\pi_{\Y_{T}}(y)$, we obtain for all $T>T_{1}$
that 
\[
\left|f\left(((u^{Q})^{-1}\circ\pi_{\Y_{T}}^{Q}(y),\mathbf{v})\right)-f\left((\pi_{\Y_{T}}(y),\mathbf{v})\right)\right|<\epsilon,\ \forall\left(y,\mathbf{v}\right)\in\Y_{T}(\R)\times\H_{a}(\Z/(q)).
\]

Let $\mathcal{K}\supsetneq\text{Supp}(f)$ be an open precompact set.
By Urysohn's lemma there exists $\ef:C_{c}(\mathcal{\Y}_{Q(\mathbf{e}_{d})}(\R)\times\H_{a}(\Z/(q)))\to[0,1]$
such that 
\[
\ef(y,\mathbf{v})\df\begin{cases}
0 & (y,\mathbf{v})\notin\mathcal{K}\\
1 & (y,\mathbf{v})\in\text{Supp}(f).
\end{cases}
\]
As above, there exists $T_{2}>0$ such that for all $T>T_{2}$
\begin{equation}
\left|\ef((u^{Q})^{-1}\circ\pi_{\Y_{T}}^{Q}(y),\mathbf{v})-\ef(\pi_{\Y_{T}}(y),\mathbf{v})\right|<\epsilon,\ \forall(y,\mathbf{v})\in\Y_{T}(\R)\times\H_{a}(\Z/(q)).\label{eq:difference of urysohn function projections}
\end{equation}
Assuming $((u^{Q})^{-1}\circ\pi_{\Y_{T}}^{Q}(y),\mathbf{v})\notin\mathcal{K}$,
we see by \eqref{eq:difference of urysohn function projections} and
by the definition of $\ef$ that $\ef(\pi_{\Y_{T}}(y),\mathbf{v})=0,$
which implies that $(\pi_{\Y_{T}}(y),\mathbf{v})\notin\text{Supp}(f)$.

By defining $T_{0}\df\max\{T_{1},T_{2}\}$ the proof of the statements
of Corollary \ref{cor:uniform left u^Q continuity for cont functions}
is done.
\end{proof}
Fix $q\in\N$ and let $\left\{ T_{n}\right\} _{n=1}^{\infty}\subseteq\N$
be an unbounded sequence such that $\vartheta_{q}(T_{n})=a$, where
$a\in\Z/(q)$ is fixed. We consider the following measure on $\mathcal{\Y}_{Q(\mathbf{e}_{d})}(\R)\times\H_{a}(\Z/(q))$
defined by 
\begin{equation}
\nu_{T_{n}}^{\mathcal{\Y},Q,q}\df\frac{1}{\left|\H_{T_{n},\text{prim}}(\Z)/\SO_{Q}(\Z)\right|}\sum_{y\in\Y_{T_{n}}(\Z)}\delta{}_{(\pi_{\Y_{T_{n}}}^{Q}(y),\vartheta_{q}(\pi_{vec}^{\Y}(y)))}\label{eq:def of nu^Y,Q,q}
\end{equation}

\begin{cor}
\label{cor:measure on equiv proj and geom proj are same..}For all
$f\in C_{c}(\mathcal{\Y}_{Q(\mathbf{e}_{d})}(\R)\times\H_{a}(\Z/(q)))$
it holds that 
\begin{equation}
\lim_{n\to\infty}\nu_{T_{n}}^{\mathcal{\Y},Q,q}(f\circ\left(u^{Q}\right)^{-1})-\nu_{T_{n}}^{\mathcal{\Y},q}(f)=0,\label{eq:measure on equiv proj and geom proj are same..}
\end{equation}
where we recall that\emph{
\[
\nu_{T}^{\mathcal{Y},q}=\frac{1}{\left|\H_{T,\text{prim}}(\Z)/\SO_{Q}(\Z)\right|}\sum_{y\in\Y_{T}(\Z)}\delta_{(\pi_{\mathcal{Y}_{T}}(y),\vartheta_{q}(\pi_{vec}^{\mathcal{Y}}(y))}.
\]
}
\end{cor}

\begin{proof}
We let $f\in C_{c}(\mathcal{\Y}_{Q(\mathbf{e}_{d})}(\R)\times\H_{a}(\Z/(q)))$
and we denote 
\[
\phi_{T}(y)\df f((u^{Q})^{-1}\circ\pi_{\Y_{T}}^{Q}(y),\vartheta_{q}(\pi_{vec}^{\mathcal{Y}}(y))-f(\pi_{\Y_{T}}(y),\vartheta_{q}(\pi_{vec}^{\mathcal{Y}}(y)),\ y\in\Y_{T}(\Z).
\]
Then

\begin{align}
\nu_{T_{n}}^{\mathcal{\Y},Q,q}\left(f\circ\left(u^{Q}\right)^{-1}\right)-\nu_{T_{n}}^{\mathcal{\Y},q}(f) & =\frac{1}{\left|\H_{T,\text{prim}}(\Z)/\SO_{Q}(\Z)\right|}\sum_{y\in\Y_{T}(\Z)}\phi_{T}(y)\nonumber \\
= & \frac{1}{\left|\H_{T,\text{prim}}(\Z)/\SO_{Q}(\Z)\right|}\sum_{y\SO_{Q}(\Z)\in\Y_{T}(\Z)/\SO_{Q}(\Z)}\ \sum_{\gamma\in\SO_{Q}(\Z)}\phi_{T}(y\cdot\gamma)\label{eq:unwinding phi_T}
\end{align}
Let $\epsilon>0$ and let $\mathcal{K}\supsetneq\text{Supp}(f)$ be
an open precompact set. We fix $T_{0}>0$ such that Corollary \ref{cor:uniform left u^Q continuity for cont functions}
holds.\textbf{ }By Corollary \ref{cor:uniform left u^Q continuity for cont functions},\eqref{enu:uniformity of difference of  values of function composed with projections}
it holds for all $T>T_{0}$ 
\begin{equation}
\left|\phi_{T}(y)\right|\leq\epsilon,\ \forall y\in\Y_{T}(\Z).\label{eq:phi_T boundd by epsilon}
\end{equation}
We now claim that there exists a constant $c=c(f)>0$ such that for
all $y\SO_{Q}(\Z)\in\Y_{T}(\Z)/\SO_{Q}(\Z)$ it holds that 
\begin{equation}
\left|\left\{ \gamma\in\SO_{Q}(\Z)\mid y\cdot\gamma\in\text{Supp}(\phi_{T})\right\} \right|\leq c.\label{eq:number of gammas in support of phi_T}
\end{equation}
By Lemma \ref{lem:uniform bound intersection with cpct set} (for
$G=\left(\SO_{Q}\times\ASL_{d-1}\right)(\R)$, $K=H$, $\Gamma=(\SO_{Q}\times\ASL_{d-1})(\Z)$
and $\tilde{\Gamma}=\{e\}\times\ASL_{d-1}(\Z)$), we obtain that for
any precompact set $\mathcal{C}\subseteq\Y_{Q(\mathbf{e}_{d})}(\R)$
there exists a uniform constant $c>0$ such that for all $y_{0}\in\Y_{Q(\mathbf{e}_{d})}(\R)$
\begin{equation}
\left|\left\{ \gamma\in\SO_{Q}(\Z)\mid y_{0}\cdot\gamma\in\mathcal{C}\right\} \right|\leq c.\label{eq:bound on so_Q(Z)  orbit in Y_T}
\end{equation}
We recall that $\pi_{\Y_{T}}^{Q}$ is $\SO_{Q}(\Z)$ equivariant,
so that 
\[
(u^{Q})^{-1}\circ\pi_{\Y_{T}}^{Q}(y\cdot\gamma)=(u^{Q})^{-1}(\pi_{\Y_{T}}^{Q}(y)\cdot\gamma).
\]
By Corollary \ref{cor:uniform left u^Q continuity for cont functions},\eqref{enu:out K implies out supp},
for all $T>T_{0}$ and for $\gamma\in\SO_{Q}(\Z)$ such that 
\[
(\pi_{\Y_{T}}^{Q}(y)\cdot\gamma,\vartheta_{q}(\pi_{vec}^{\mathcal{Y}}(y\cdot\gamma)))\notin u^{Q}(\mathcal{K}),
\]
we have $\left|\phi_{T}(y\cdot\gamma)\right|=0$, namely $\text{Supp}(\phi_{T})\subseteq u^{Q}(\mathcal{K})$,
which shows
\[
\left|\left\{ \gamma\in\SO_{Q}(\Z)\mid y\cdot\gamma\in\text{Supp}(\phi_{T})\right\} \right|\leq\left|\left\{ \gamma\in\SO_{Q}(\Z)\mid(\pi_{\Y_{T}}^{Q}(y)\cdot\gamma,\vartheta_{q}(\pi_{vec}^{\mathcal{Y}}(y\cdot\gamma)))\in u^{Q}(\mathcal{K})\right\} \right|
\]
 Consider the natural map $\pi_{\infty}:\mathcal{\Y}_{Q(\mathbf{e}_{d})}(\R)\times\H_{a}(\Z/(q))\to\mathcal{\Y}_{Q(\mathbf{e}_{d})}(\R)$.
Since $u^{Q}$ is a homeomorphism, and as $\mathcal{K}$ is precompact,
by \eqref{eq:bound on so_Q(Z)  orbit in Y_T} there is a constant
$c>0$ such that for all $y\in\Y_{T}(\Z)$ 
\[
\left|\left\{ \gamma\in\SO_{Q}(\Z)\mid\pi_{\Y_{T}}^{Q}(y)\cdot\gamma\in\pi_{\infty}(u^{Q}(\mathcal{K}))\right\} \right|\leq c,
\]
which shows \eqref{eq:number of gammas in support of phi_T}. Finally,
by \eqref{eq:unwinding phi_T}, \eqref{eq:phi_T boundd by epsilon}
and \eqref{eq:number of gammas in support of phi_T} we obtain for
all $T>T_{0}$
\[
\left|\nu_{T_{n}}^{\mathcal{\Y},Q,q}\left(f\circ\left(u^{Q}\right)^{-1}\right)-\nu_{T_{n}}^{\mathcal{\Y},q}(f)\right|\leq\frac{\left|\Y_{T}(\Z)/\SO_{Q}(\Z)\right|}{\left|\H_{T,\text{prim}}(\Z)/\SO_{Q}(\Z)\right|}\epsilon c.
\]
Now the map $\pi_{vec}^{\Y}:\Y_{T}(\Z)\to\H_{T,\text{prim}}(\Z)$
is a bijection which is equivariant with respect to the right $\SO_{Q}(\Z)$
action, which shows that 
\[
\frac{\left|\Y_{T}(\Z)/\SO_{Q}(\Z)\right|}{\left|\H_{T,\text{prim}}(\Z)/\SO_{Q}(\Z)\right|}=1,
\]
and completes our proof.
\end{proof}

\subsubsection{\label{subsec:Concluding-the-proof that Y follows from Z}Concluding
the proof that the results for $\text{\ensuremath{\protect\V}}$ imply
the results for $\protect\Y$}

We now give a detailed proof that Theorem \ref{thm:main_thm_with_congruences-forZ}
implies Theorem \ref{thm:moduli_main_thm_with_congruences} for $\mathcal{M}=\Y$.
The proof that Theorem \ref{thm:main thm for Z} implies Theorem \ref{thm:moduli main thm }
follows along the same lines, and is left for the reader.

In the following we fix $q\in2\N+1$ and we let $a\in\left(\Z/(q)\right)^{\times}$.

Let $f\in C_{c}(\mathcal{\Y}_{Q(\mathbf{e}_{d})}(\R)\times\H_{a}(\Z/(q)))$,
and consider $\bar{\ef}_{f}\in C_{c}(\mathcal{\Y}_{Q(\mathbf{e}_{d})}(\R)\times\H_{a}(\Z/(q)))$
given by 
\[
\bar{\ef}_{f}\df f\circ(u^{Q})^{-1},
\]
where we abuse notations with $f\circ(u^{Q})^{-1}(y,\mathbf{v})=f\left((u^{Q})^{-1}(y),\mathbf{v}\right)$.

By Lemma \ref{lem:star bar map is onto}, there exists $\ef_{f}\in C_{c}(\V_{Q(\mathbf{e}_{d})}(\R)\times\H_{a}(\Z/(q))$
such that 
\[
\bar{\ef}_{f}(z\ASL_{d-1}(\Z),\mathbf{v})=\sum_{\gamma\in\ASL_{d-1}(\Z)}\ef_{f}(z\gamma,\mathbf{v}).
\]
Let $\ef_{f}^{\tau}\in C_{c}(\V_{Q(\mathbf{e}_{d})}(\R)\times\V_{a}(\Z/(q))$
be defined by $\ef_{f}^{\tau}(z,g)\df\ef_{f}(z,\ovec(g))$ (where
$\ovec$ defined in \eqref{eq:def of tau}) We claim that 
\[
\nu_{T}^{\mathcal{\V},q}(\ef_{f}^{\tau})=\nu_{T}^{\mathcal{\Y},Q,q}(\bar{\ef}_{f}),
\]
where $\nu_{T}^{\mathcal{\V},q}$ defined in \eqref{eq:congruence counting measures on nu_Z_Q(E_d))}
and $\nu_{T}^{\mathcal{\Y},Q,q}$ defined in \eqref{eq:def of nu^Y,Q,q}.
We have

\begin{align*}
\nu_{T}^{\mathcal{\V},q}(\ef_{f}^{\tau})= & \frac{1}{\left|\H_{T,\text{prim}}(\Z)/\SO_{Q}(\Z)\right|}\sum_{z\in\V_{T}(\Z)}\ef_{f}^{\tau}(\pi_{\V_{T}}(z),\vartheta_{q}(z))\\
= & \frac{1}{\left|\H_{T,\text{prim}}(\Z)/\SO_{Q}(\Z)\right|}\sum_{z\in\V_{T}(\Z)}\ef_{f}(\pi_{\V_{T}}(z),\vartheta_{q}(\ovec(z)))\\
= & \frac{1}{\left|\H_{T,\text{prim}}(\Z)/\SO_{Q}(\Z)\right|}\sum_{z\ASL_{d-1}(\Z)\in\V_{T}(\Z)/\ASL_{d-1}(\Z)}\ \sum_{\gamma\in\ASL_{d-1}(\Z)}\ef_{f}(\pi_{\V_{T}}(z\gamma),\vartheta_{q}(\ovec(z\gamma)))\\
= & \frac{1}{\left|\H_{T,\text{prim}}(\Z)/\SO_{Q}(\Z)\right|}\sum_{z\ASL_{d-1}(\Z)\in\V_{T}(\Z)/\ASL_{d-1}(\Z)}\ \sum_{\gamma\in\ASL_{d-1}(\Z)}\ef_{f}(\pi_{\V_{T}}(z)\gamma,\vartheta_{q}(\ovec(z)))\\
= & \frac{1}{\left|\H_{T,\text{prim}}(\Z)/\SO_{Q}(\Z)\right|}\sum_{y\in\Y_{T}(\Z)}\bar{\ef}_{f}(\pi_{\Y_{T}}^{Q}(y),\vartheta_{q}(\pi_{vec}^{\Y}(y))\\
= & \nu_{T}^{\mathcal{\Y},Q,q}(\bar{\ef}_{f}).
\end{align*}
Assume that $Q$ is non-singular modulo $q\in2\N+1$. Let $\{T_{n}\}_{n=1}^{\infty}\subseteq\N$
be an unbounded sequence of integers satisfying the $\left(Q,p_{0}\right)$
co-isotropic property for some $p_{0}$ and assume that $\red_{q}\left(T_{n}\right)=a,\ \forall n\in\N$.
Then by assuming Theorem \ref{thm:main_thm_with_congruences-forZ},
we get
\[
\lim_{n\to\infty}\nu_{T_{n}}^{\mathcal{\Y},Q,q}(\bar{\ef}_{f})=\lim_{n\to\infty}\nu_{T_{n}}^{\mathcal{\V},q}(\ef_{f}^{\ovec})=\mu_{\mathcal{\V}}\otimes\mu_{\V_{a}(\Z/(q)}(\ef_{f}^{\ovec}).
\]
We recall by the proof of Corollary \ref{cor:transitivity of ASL_times_SO_d}
that $\tau(\V_{a}(\Z/(q))=\H_{a}(\Z/(q))$ and we observe that
\[
\mu_{\mathcal{\V}}\otimes\mu_{\V_{a}(\Z/(q))}(\ef_{f}^{\tau})=\mu_{\V}\otimes\tau_{*}\mu_{\V_{a}(\Z/(q))}(\ef_{f})=\mu_{\V}\otimes\mu_{\H_{a}(\Z/(q))}(\ef_{f}).
\]
By Lemma \ref{lem:integral unwinding of z w.r.t y} 
\[
\mu_{\V}\otimes\mu_{\H_{a}(\Z/(q)}(\ef_{f})=\mu_{\Y}\otimes\mu_{\H_{a}(\Z/(q)}(\bar{\ef}_{f}),
\]
which implies in turn that \emph{
\begin{equation}
\lim_{n\to\infty}\nu_{T_{n}}^{\mathcal{\Y},Q,q}(\bar{\ef}_{f})=\mu_{\Y}\otimes\mu_{\H_{a}(\Z/(q)}(\bar{\ef}_{f}).\label{eq:assumption of thm with cong}
\end{equation}
}Our goal now is to show that \eqref{eq:assumption of thm with cong}
implies 
\begin{equation}
\lim_{n\to\infty}\nu_{T_{n}}^{\mathcal{\Y},q}(f)=\mu_{\Y}\otimes\mu_{\H_{a}(\Z/(q)}(f),\label{eq:statement of thm with cong}
\end{equation}
which is the statement of Theorem \ref{thm:moduli_main_thm_with_congruences}.

We have by definition of $\bar{\ef}_{f}$
\begin{equation}
\begin{aligned}\lim_{n\to\infty}\nu_{T_{n}}^{\mathcal{\Y},Q,q}(f\circ(u^{Q})^{-1})= & \lim_{n\to\infty}\nu_{T_{n}}^{\mathcal{\Y},Q,q}(\bar{\ef}_{f})\\
= & \lim_{n\to\infty}\mu_{\Y}\otimes\mu_{\H_{a}(\Z/(q)}(\bar{\ef}_{f})\\
= & \lim_{n\to\infty}\mu_{\Y}\otimes\mu_{\H_{a}(\Z/(q)}(f\circ(u^{Q})^{-1}).
\end{aligned}
\label{eq:asuming limit on z yields limit on equiv proj}
\end{equation}
By Corollary \ref{cor:measure on equiv proj and geom proj are same..}\textbf{
}and by \eqref{eq:asuming limit on z yields limit on equiv proj}
we obtain that 
\[
\lim_{n\to\infty}\nu_{T_{n}}^{\mathcal{\Y},q}(f)=\mu_{\Y}\otimes\mu_{\H_{a}(\Z/(q)}(f\circ(u^{Q})^{-1}),
\]
and finally, since $u^{Q}$ preserves $\mu_{\Y}$ (see Lemma \ref{lem:U^Q preserves mu_Y})
we obtain \eqref{eq:statement of thm with cong}.

\subsection{\label{subsec:The-results-for Y imply X,W}The results for $\protect\Y$
imply the results for $\protect\X$ and $\protect\W$}

In the following we show that Theorems \ref{thm:moduli main thm }
- \ref{thm:moduli_main_thm_with_congruences} for $\mathcal{M}=\Y$
imply Theorems \ref{thm:moduli main thm } - \ref{thm:moduli_main_thm_with_congruences}
for $\mathcal{M}\in\left\{ \X,\W\right\} .$ It may be helpful for
the reader to recall Section \ref{subsec:Pushforwards}.

We fix $T\in\N$ and we note the following commuting diagram (which
follows from \eqref{eq:projections diagram}),
\[
\xymatrix{\Y_{T}(\Z)\ar@{^{(}->}[d]_{\pi_{\Y_{T}}}\ar@{<->}[r]^{\pi_{\cap}}\ar@{<->}@/^{2pc}/[rr]^{\pi_{vec}^{\Y}} & \X_{T}(\Z)\ar@{^{(}->}[d]^{\pi_{\X_{T}}}\ar@{<->}[r]^{\pi_{vec}^{\X}} & \H_{T,\text{prim}}(\Z)\\
\Y_{Q(\mathbf{e}_{d})}(\R)\ar[r]^{\pi_{\cap}} & \X_{Q(\mathbf{e}_{d})}(\R)
}
\]
which shows that

\begin{align*}
\nu_{T}^{\mathcal{\X},q}= & \frac{1}{\left|\H_{T,\text{prim}}(\Z)/\SO_{Q}(\Z)\right|}\sum_{x\in\mathcal{\X}_{T}(\Z)}\delta_{(\pi_{\X_{T}}(x),\vartheta_{q}(\pi_{vec}^{\mathcal{\X}}(x))}\\
= & \frac{1}{\left|\H_{T,\text{prim}}(\Z)/\SO_{Q}(\Z)\right|}\sum_{y\in\mathcal{\Y}_{T}(\Z)}\delta_{(\pi_{\cap}\circ\pi_{\Y_{T}}(y),\vartheta_{q}(\pi_{vec}^{\Y}(y))}\\
= & \left(\pi_{\cap}\times id\right)_{*}\nu_{T}^{\Y,q}.
\end{align*}

By Lemma \ref{lem:push of pi_cap} we have $\left(\pi_{\cap}\right)_{*}\mu_{\Y}=\mu_{\X}$,
hence we obtain the limits for $\X$ from the limits of $\Y$.

Next, we observe that

\begin{align*}
\nu_{T}^{\W,q}= & \frac{1}{\left|\H_{T,\text{prim}}(\Z)/\SO_{Q}(\Z)\right|}\sum_{\mathbf{v}\in\H_{T,\text{prim}}(\Z)}\delta_{\left(\text{shape}(\Lambda_{\mathbf{v}}),\frac{1}{\sqrt{T}}\mathbf{v},\vartheta_{q}(\mathbf{v})\right)}\\
= & \frac{1}{\left|\H_{T,\text{prim}}(\Z)/\SO_{Q}(\Z)\right|}\sum_{x\in\mathcal{\X}_{T}(\Z)}\delta_{\left((\shape\times\pi_{vec}^{\X})(\pi_{\X_{T}}(x)),\vartheta_{q}(\pi_{vec}^{\mathcal{\X}}(x))\right)}\\
= & \left((\shape\times\pi_{vec}^{\X})\times id\right)_{*}\nu_{T}^{\X,q},
\end{align*}
and by Lemma \ref{lem:push of shape}, we have $\left(\pi_{vec}^{\X}\times\shape\right)_{*}\mu_{\X}=\mu_{\W}$,
which shows that the limits for $\W$ follow from the limits of $\X$.

\section{\label{sec:Some-facts-on}Some technicalities}

This section discusses several technical facts about quadratic forms
that will be used in the rest of the paper (mainly in Section \ref{sec:A-revisit-to-s-arith-thm-aes}).

For a prime $p$ we denote by $\Z_{p}$ the ring of $p$-adic integers
and by $\Q_{p}$ the field of $p$-adic numbers.
\begin{lem}
\label{lem:properties of a form non-singular modulo q}Let $Q$ be
an integral form which is non-singular modulo $q$ \emph{(}see Definition
\ref{def:non-singularity modq}\emph{)} for $q\in2\N+1$ and let $S_{q}$
be the set of primes appearing in the prime decomposition of $q$.
Then the following hold:
\begin{enumerate}
\item \label{enu:The-reduction-map in onto modulo q}The reduction map \emph{$\red_{p^{k}}:\SO_{Q}(\Z_{p})\to\SO_{Q}(\Z/(p^{k}))$}
is onto for all $p\in S_{q}$ and $k\geq1.$
\item \label{enu:-is-isotropic for all p=00005Cin S_q}$Q$ is isotropic
over $\Q_{p}$ for all $p\in S_{q}$.
\end{enumerate}
\end{lem}

\begin{proof}
\begin{enumerate}
\item Fix $p\in S_{q}$. To prove that $\vartheta_{p^{k}}:\SO_{Q}(\Z_{p})\to\SO_{Q}(\Z/(p^{k}))$
is onto, we will prove that the natural projection
\[
\pi_{k}:\SO_{Q}(\Z/(p^{k+1}))\to\SO_{Q}(\Z/(p^{k}))
\]
is onto for all $k\geq1$. We let $\bar{g}\in\SO_{Q}(\Z/(p^{k}))$
and we take $F\in M_{d}(\Z_{p})$ such that $\vartheta_{p^{k}}(F)=\bar{g}$.
Since $\det(\bar{g})=1$, it follows that $\det\left(F\right)\in\Z_{p}^{\times}$,
which implies that $F\in\GL_{d}(\Z_{p})$. Fix a symmetric matrix
$M\in M_{d}(\Z)$ such that
\[
Q(\mathbf{x})=\mathbf{x}^{t}M\mathbf{x}
\]
Since $Q$ is non-singular modulo $q$ it follows that $\det(M)\in\Z_{p}^{\times}$
for all $p\in S_{q}$, namely $M\in\GL_{d}(\Z_{p})$ for all $p\in S_{q}$.
We may now define $S\in M_{d}(\Z_{p})$ by 
\begin{equation}
S\df\frac{1}{2}\left(M^{-1}\left(F^{t}\right)^{-1}M-F\right).\label{eq:S mod p^k}
\end{equation}
By noting that $\vartheta_{p^{k}}(F^{t}MF)=\vartheta_{p^{k}}(M)$
we obtain that $\vartheta_{p^{k}}(S)=0$, so that in particular $\vartheta_{p^{k}}(F+S)=\bar{g}$.
To finish the proof, it is sufficient to show that $\vartheta_{p^{k+1}}(F+S)\in\SO_{Q}(\Z/(p^{k+1})).$
We observe that 
\begin{equation}
\left(F+S\right)^{t}M(F+S)=F^{t}MF+F^{t}MS+S^{t}MF+S^{t}MS.\label{eq:conjugation of M}
\end{equation}
We treat each of the terms appearing in \eqref{eq:conjugation of M}
separately.
\begin{enumerate}
\item The term $F^{t}MS$. By substituting \eqref{eq:S mod p^k} in $S$,
we obtain that 
\[
F^{t}MS=\frac{1}{2}F^{t}M\left(M^{-1}\left(F^{-1}\right)^{t}M-F\right)=\frac{1}{2}M-\frac{1}{2}F^{t}MF.
\]
\item The term $S^{t}MF$. By substituting \eqref{eq:S mod p^k} in $S^{t}$,
we obtain that
\[
S^{t}MF=\frac{1}{2}\left(M^{t}F^{-1}\left(M^{t}\right)^{-1}-F^{t}\right)MF\underbrace{=}_{M^{t}=M}\frac{1}{2}M-\frac{1}{2}F^{t}MF.
\]
Hence we deduce by the above that
\[
\left(F+S\right)^{t}M(F+S)=M+S^{t}MS,
\]
Since $\vartheta_{p^{k}}(S)=0,$ we obtain that $\vartheta_{p^{2k}}(S^{t}MS)=0$.
Namely $\vartheta_{p^{k+1}}\left(\left(F+S\right)^{t}M(F+S)\right)=\vartheta_{p^{k+1}}(M)$,
which completes the proof.
\end{enumerate}
\item Let $M$ be the companion matrix of $Q$. By definition of non-singularity
modulo $q$ (see Definition \ref{def:non-singularity modq}) we have
that $\left|\det(M)\right|_{p}=1$ for all $p\in S_{q}$, where $\left|\cdot\right|_{p}$
denotes the p-adic valuation. Fix $p\in S_{q}$. By \cite[Chapter 8, Theorem 3.1]{CASSELS_quad}
there exists $g\in\GL_{d}(\Z_{p})$ such that 
\[
g^{t}Mg=\left(\begin{array}{ccc}
a_{1}\\
 & \ddots\\
 &  & a_{d}
\end{array}\right),
\]
where $a_{1},...,a_{d}\in\Z_{p}$. Now 
\[
|a_{1}|_{p}\cdot...\cdot|a_{d}|_{p}=\left|\det(g^{t}Mg)\right|_{p}=\left|\det(M)\right|_{p}=1.
\]
Hence $|a_{1}|_{p}=...=|a_{d}|_{p}=1$, and by \cite[Chapter 3, Lemma 1.7]{CASSELS_quad}
we get that $Q(g\mathbf{x})=a_{1}x_{1}^{2}+...+a_{d}x_{d}^{2}$ has
an isotropic vector over $\Q_{p}$.
\end{enumerate}
\end{proof}
For $g\in\SL_{d}(\Z)$ and $\gamma\in\SL_{d-1}(\Q)$ we define a quadratic
form $\ef_{g}^{\gamma}:\Q^{d-1}\to\Q,$ by 
\begin{equation}
\ef_{g}^{\gamma}(u)\df Q^{*}\circ g\circ\gamma(u)\label{eq:quadratic form phi_g^=00005Cgamma}
\end{equation}
 (see definition of $Q^{*}$ in \eqref{eq:def of Q^*}), where we
identify $\Q^{d-1}$ with $\Q^{d-1}\times\left\{ 0\right\} $. We
will denote $\ef_{g}\df\ef_{g}^{I_{d-\text{1}}}.$

Let $\hat{g}\in M_{d\times d-1}(\R)$ be the matrix formed by the
first $d-1$ columns of $g$. Then the matrix
\begin{equation}
M_{\ef_{g}^{\gamma}}\df\gamma^{t}\hat{g}^{t}M^{-1}\hat{g}\gamma\label{eq:def of companion matrix =00005Cef_g^=00005Cgamma}
\end{equation}
 is a companion matrix for the form $\ef_{g}^{\gamma}$.
\begin{lem}
\label{lem:discreminant of M_ef_g^=00005Cgamma}It holds that $\det\left(M_{\ef_{g}^{\gamma}}\right)=\frac{1}{\det(M)}Q(\ovec(g)).$
\end{lem}

\begin{proof}
First, by the multiplicativity of the determinant, we get that $\det\left(M_{\ef_{g}^{\gamma}}\right)=\det\left(M_{\ef_{g}}\right).$
Next, we observe that

\begin{align*}
Q(\ovec(g))= & \left\langle \left(g^{t}\right)^{-1}\mathbf{e}_{d},M\left(g^{t}\right)^{-1}\mathbf{e}_{d}\right\rangle \\
= & \left\langle \mathbf{e}_{d},\left(g^{t}M^{-1}g\right)^{-1}\mathbf{e}_{d}\right\rangle ,
\end{align*}
which is the $d,d$ entry of the matrix $\left(g^{t}M^{-1}g\right)^{-1}$.
Now

\begin{align*}
\left(g^{t}M^{-1}g\right)^{-1}= & \frac{1}{\det(g^{t}M^{-1}g)}\text{adj}\left(g^{t}M^{-1}g\right)\\
= & \det(M)\cdot\text{adj}\left(g^{t}M^{-1}g\right).
\end{align*}
We note that the $d,d$ entry of $\text{adj}\left(g^{t}M^{-1}g\right)$
is given by the minor $\det\left(\hat{g}^{t}M^{-1}\hat{g}\right)=\det\left(M_{\ef_{g}}\right)$,
which proves our claim.
\end{proof}
Consider the natural map
\[
\pi_{\SL_{d-1}}:\ASL_{d-1}\to\SL_{d-1},
\]
given by
\[
\left(\begin{array}{cc}
m & *\\
 & 1
\end{array}\right)\mapsto m.
\]

\begin{lem}
\label{lem:the second factor of L_g is orthogonal group of a form in d-1 var}We
have \emph{
\[
\gamma^{-1}\pi_{\SL_{d-1}}\left(g^{-1}\o\left(\mathbf{H}_{\ovec(g)}(\R)\right)g\right)\gamma=\SO_{\ef_{g}^{\gamma}}(\R).
\]
}
\end{lem}

\begin{proof}
We recall by Lemma \ref{lem:F(SO_Q) and stabilizer of a vector} that
$\o\left(\mathbf{H}_{\ovec(g)}(\R)\right)$ is the subgroup of $\SO_{Q^{*}}(\R)$
that preserves the hyperplane $\text{Span}_{\R}\{\mathbf{g}_{1},...,\mathbf{g}_{d-1}\}$,
where $\mathbf{g}_{i}$ denotes the $i$'th column of $g.$ Therefore
group $g^{-1}\o\left(\mathbf{H}_{\ovec(g)}(\R)\right)g$ is the subgroup
of $\SO_{Q^{*}\circ g}(\R)$ which preserves the hyperplane $\R^{d-1}\times\left\{ 0\right\} $
by the left linear action. Hence $\pi_{\SL_{d-1}}\left(g^{-1}\o\left(\mathbf{H}_{\ovec(g)}(\R)\right)g\right)$
is the restriction of $\SO_{Q^{*}\circ g}(\R)$ to the hyperplane
$\R^{d-1}\times\left\{ 0\right\} $, which shows $\pi_{\SL_{d-1}}\left(g^{-1}\o\left(\mathbf{H}_{\ovec(g)}(\R)\right)g\right)=\SO_{\ef_{g}}(\R)$.
Finally we note that
\[
\gamma^{-1}\SO_{\ef_{g}}(\R)\gamma=\SO_{\ef_{g}^{\gamma}}(\R).
\]
\end{proof}
\begin{lem}
\label{lem:gcd of companion matrix}Let \emph{$A\in M_{d}(\Z)\cap\GL_{d}(\Q)$}.
Then for any \emph{$g\in\SL_{d}(\Z)$}, the g.c.d of the entries of
the integral matrix \emph{
\[
A_{g}\df\hat{g}^{t}A\hat{g}
\]
}where $\hat{g}$ is the matrix formed by the first $d-1$ columns
of $g$, is at-most $\det A$.
\end{lem}

\begin{proof}
To prove our claim it is sufficient to show that there exist two integral
vectors $\mathbf{b},\mathbf{a}\in\Z^{d-1}$ such that 
\[
\mathbf{b}^{t}A_{g}\mathbf{a}=\alpha,
\]
for $\alpha\in\Z$ satisfying that $\alpha\mid\det(A)$. This will
be done by a variation on the geometric argument given in the proof
of \cite[Lemma 3.2]{AESgrids}. Let $\mathbf{u}_{1}\in\Z_{prim}^{d}\bigcap\left(A\mathbf{g}_{d}\right)^{\perp}\bigcap\hat{g}\Q^{d-1}$
where $\mathbf{g}_{d}\df g\mathbf{e}_{d}$ (such a vector exists since
$\left(A\mathbf{g}_{d}\right)^{\perp}\bigcap\hat{g}\Q^{d-1}$ is the
intersection of two rational hyperplanes). Namely, we choose $\mathbf{u}_{1}\in\Z_{prim}^{d}$
such that 
\[
\mathbf{u}_{1}=\hat{g}\mathbf{a},\ \mathbf{a}\in\Z^{d-1},
\]
(the entries of $\mathbf{a}$ are integral since the columns of $g$
form a $\Z$ basis for $\Z^{d}$) and 
\begin{equation}
0=\left(A\mathbf{g}_{d}\right)^{t}\mathbf{u}_{1}=\mathbf{g}_{d}^{t}A\mathbf{u}_{1}.\label{eq:orthogonality to Ag_d}
\end{equation}
Let $\alpha\in\N$ be the g.c.d. of the entries of $A\mathbf{u}_{1}$.
Since $\mathbf{u}_{1}\in\Z_{\text{prim}}^{d}$, we may use \cite[Chapter 1, Theorem 1.B]{Cas_geo_num}
to deduce that $\alpha\mid\det(A)$. Let $\tilde{\mathbf{u}}_{2}\in\Z^{d}$
such that 
\begin{equation}
\tilde{\mathbf{u}}_{2}^{t}\left(A\mathbf{u}_{1}\right)=\alpha.\label{eq:inner product is g.c.d}
\end{equation}
 Since $\mathbf{g}_{1},...,\mathbf{g}_{d}$ form a $\Z$-basis for
$\Z^{d}$, we may write 
\[
\tilde{\mathbf{u}}_{2}=\hat{g}\mathbf{b}+b_{d}\mathbf{g}_{d},\ \mathbf{b}\in\Z^{d-1},\ b_{d}\in\Z,
\]
then by \eqref{eq:orthogonality to Ag_d} and \eqref{eq:inner product is g.c.d}
we obtain
\[
\alpha=\left(\hat{g}\mathbf{b}\right)^{t}\left(A\hat{g}\mathbf{a}\right)=\mathbf{b}^{t}A_{g}\mathbf{a},
\]
which completes the proof.
\end{proof}

\section{\label{sec:A-revisit-to-s-arith-thm-aes}A revisit to the $S$-arithemetic
theorem of \cite{AESgrids}}

The purpose of this section is to prove Theorem \ref{thm:AESgrids thm}
below, which concerns the equidistribution of a sequence of compact
orbits in an $S$-arithemetic space. We note that Theorem \ref{thm:AESgrids thm}
generalizes Theorem 3.1 of \cite{AESgrids} by taking into account
more general quadratic forms and by also taking into account more
than one prime. Yet, we note that our proof of Theorem \ref{thm:AESgrids thm}
strongly relies on the ideas and methods which already appear in the
proof of Theorem 3.1 of \cite{AESgrids}.

In the following we consider algebraic groups defined over $\Q$,
and we follow the notations and conventions as in \cite{platonov_rapinchuk},
Chapter 2.

\textbf{ }To ease the notation, we introduce 
\[
\G_{1}\overset{\text{def}}{=}\text{SO}{}_{Q},\ \G_{2}\overset{\text{def}}{=}\text{ASL}{}_{d-1},\ \G\overset{\text{def}}{=}\G_{1}\times\G_{2},
\]
where we recall that $Q$ is as in our \nameref{subsec:Standing-Assumption}.
For a finite set of primes $S$, we denote by $\Q_{S}\df\prod_{p\in S}\Q_{p}$,
where $\Q_{p}$ is the field of p-adic numbers, by $\Z_{S}\df\prod_{p\in S}\Z_{p}$,
where $\Z_{p}$ is the ring of p-adic integers, and by $\ZS\df\Z\left[\frac{1}{p};p\in S\right]$.

We consider
\[
\G(\R\times\Q_{S})\df\G(\R)\times\G(\Q_{S}),
\]
we define $\G\left(\ZS\right)\leq\G(\R\times\Q_{S})$ by
\[
\G\left(\ZS\right)\df\left\{ (\gamma_{1},\gamma_{2},\gamma_{1},\gamma_{2})\mid\gamma_{i}\in\G_{i}(\ZS)\right\} ,
\]
we recall that $\G(\ZS)$ is a lattice in $\G(\R\times\Q_{S})$ (see
\cite{platonov_rapinchuk}, Chapter 5), and we define
\[
\Y_{S}\overset{\text{def}}{=}\G(\R\times\Q_{S})/\G\left(\ZS\right)
\]
(we use the above notation in this section only. Note not to be confused
with the notation $\Y$ for the space of oriented grids). Let $g\in\SL_{d}(\Z)$
such that $Q(\ovec(g))>0$. Using the transitivity of the $\G(\R)$-action
on $\V_{Q(\ovec(g))}(\R)$ (see Corollary \ref{cor:transitivity of ASL_times_SO_d}),
we choose $t_{g}=\left(\left(t_{g}\right)_{1},\left(t_{g}\right)_{2}\right)\in\G(\R)$
such that 
\begin{equation}
g=a_{Q(\ovec(g))}\cdot t_{g}\underbrace{=}_{\text{definition of the \ensuremath{\G} action on \ensuremath{\SL_{d}}}}\o\left(\left(t_{g}\right)_{1}^{-1}\right)a_{Q(\ovec(g))}\left(t_{g}\right)_{2},\label{eq:def of theta_g}
\end{equation}
where $a_{Q(\ovec(g))}\in\V_{Q(\ovec(g))}(\R)$ was defined in Definition
\ref{def:def of a_T}.

We define the twisted orbit
\begin{equation}
O_{g,S}\df(t_{g},e_{S})\mathbf{L}_{g}(\R\times\Q_{S})\G(\ZS),\label{eq:def of O_g,s}
\end{equation}
where $\mathbf{L}_{g}\leq\G$ is the stabilizer of $g$ (see Lemma
\ref{lem:stabilizer lemma}).

We observe that $\mathbf{L}_{g}(\R)$ is a compact group since by
assuming that $Q(\ovec(g))>0$, it follows that $Q(\ovec(g))=TQ(\mathbf{e}_{d})$
for $T>0$ implying that $\mathbf{H}_{\ovec(g)}(\R)$ (which is isomorphic
to $\mathbf{L}_{g}(\R)$) is conjugate to $\mathbf{H}_{\mathbf{e}_{d}}(\R)$
(the action of $\SO_{Q}(\R)$ is transitive on $\H_{Q(\sqrt{T}\mathbf{e}_{d})}(\R)$),
which is compact by our \nameref{subsec:Standing-Assumption}. Then
by \cite[Theorem 5.7]{platonov_rapinchuk} we obtain that $\mathbf{L}_{g}(\R\times\Q_{S})\G\left(\ZS\right)\subseteq\Y_{S}$
is a compact orbit, and we define
\begin{equation}
\mu_{g,S}\df\left(t_{g},e_{S}\right)_{*}\mu_{\mathbf{L}_{g}(\R\times\Q_{S})\G\left(\ZS\right)}\label{eq:=00005Cmu_g,S}
\end{equation}
 where $\mu_{\mathbf{L}_{g}(\R\times\Q_{S})\G\left(\ZS\right)}$ is
the $\mathbf{L}_{g}(\R\times\Q_{S})$-invariant probability measure
supported on $\mathbf{L}_{g}(\R\times\Q_{S})\G\left(\ZS\right)$.
\begin{thm}
\emph{\label{thm:AESgrids thm}} Assume that $S$ is a finite set
of odd primes such that $Q$ is isotropic over $\Q_{p}$ for all $p\in S$.
Let \emph{$\left\{ g_{n}\right\} _{n=1}^{\infty}\subseteq\SL_{d}(\Z)$}
such that $Q(\ovec(g_{n}))>0$ for all $n\in\N$, such that $Q(\ovec(g_{n}))\to\infty$,
and such that there exists $p_{0}\in S$ for which $\ovec(g_{n})$
is $(Q,p_{0})$ co-isotropic for all $n\in\N$ \emph{(}see Definition
\ref{def:Isotropicity definition}\emph{)}. Then,
\[
\mu_{g_{n},S}\overset{\text{weak * }}{\longrightarrow}\mu_{\Y_{S}},
\]
where $\mu_{\Y_{S}}$ is the $\G(\R\times\Q_{S})$-invariant probability
measure on $\Y_{S}$.
\end{thm}

\subsection{\label{subsec:Proof-of-AES-Theorem}Proof of Theorem \ref{thm:AESgrids thm}}

The key input for the proof of Theorem \ref{thm:AESgrids thm} is
\cite[Theorem 4.6]{GO}, which we state in a simplified form in Theorem
\ref{thm:go thm} below.

For the rest of the section, we will denote the simply connected covering
of a semi-simple algebraic group $\mathbb{L}$ defined over $\Q$
by $\tilde{\mathbb{L}}$ and the universal covering map by $\pi:\tilde{\mathbb{L}}\to\mathbb{L}$
(for more details see e.g. \cite[Section 2.1.13]{platonov_rapinchuk}).
\begin{thm}
\label{thm:go thm}Let $\mathbf{G}$ be a connected semi-simple algebraic
group defined over $\Q$, let $S$ be a finite set of primes and let
$\mathbb{L}_{n}$, $n\in\N,$ be a sequence of connected semi-simple
$\Q$-subgroups of $\mathbf{G}$. Consider a sequence $\left\{ t_{n}\right\} _{n=1}^{\infty}\subseteq\mathbf{G}(\R\times\Q_{S})$
and let $\nu_{n}\df\left(t_{n}\right)_{*}\mu_{\pi\left(\tilde{\mathbb{L}}_{n}\left(\R\times\Q_{S}\right)\right)\mathbf{G}(\ZS)}$,
where $\mu_{\pi\left(\tilde{\mathbb{L}}_{n}\left(\R\times\Q_{S}\right)\right)\mathbf{G}(\ZS)}$
is the unique $\pi\left(\tilde{\mathbb{L}}_{n}\left(\R\times\Q_{S}\right)\right)$-invariant
probability measure supported on $\pi\left(\tilde{\mathbb{L}}_{n}\left(\R\times\Q_{S}\right)\right)\mathbf{G}\left(\ZS\right)$.
\begin{elabeling}{00.00.0000}
\item [{$\$1$\label{=0000A71}}] Assume that there exists $p\in S$ such
that for all $n\in\N$ and all connected non-trivial normal $\Q_{p}$-subgroups
$\mathbf{N}\trianglelefteqslant\mathbb{L}_{n}$ it holds that $\mathbf{N}(\Q_{p})$
is non-compact \emph{(}in terms of \cite{GO}, $S$ is strongly isotropic\emph{)}.
\end{elabeling}
Let $\nu$ be a probability measure on $\mathbf{G}(\R\times\Q_{S})/\mathbf{G}(\ZS)$
which is a weak-star limit of $\left\{ \nu_{n}\right\} _{n=1}^{\infty}$.
Then:
\begin{enumerate}
\item \label{enu:go_1}There exists a connected $\Q$-algebraic subgroup
$\M\leq\G$ such that $\nu=(t_{0})_{*}\mu_{M\G(\ZS)}$ where $M$
is a closed finite index subgroup of $\M(\R\times\Q_{S})$, $t_{0}\in\mathbf{G}(\R\times\Q_{S})$
and $\mu_{M\G(\ZS)}$ is the left $M$-invariant probability measure
supported on $M\G(\ZS)$.
\item \label{enu:go_2}There exists a sequence $\left\{ \gamma_{n}\right\} _{n=1}^{\infty}\subseteq\G(\ZS),$
such that for all large enough $n$ it holds that
\[
\gamma_{n}^{-1}\mathbb{L}_{n}\gamma_{n}\subseteq\M.
\]
\item \label{enu:go_3}There exists a sequence $\left\{ l_{n}\right\} _{n=1}^{\infty}\subseteq\pi(\tilde{\mathbb{L}}_{n}(\R\times\Q_{S}))$
such that 
\[
\lim_{n\to\infty}t_{n}l_{n}\gamma_{n}=t_{0}.
\]
\end{enumerate}
In addition,
\begin{elabeling}{00.00.0000}
\item [{$\$2$\label{=0000A72}}] assume that for all $n\in\N$ the centralizer
of $\mathbf{L}_{n}$ in $\mathbf{G}$ is $\Q$-anisotropic.
\end{elabeling}
Then the sequence of measures $\left\{ \nu_{n}\right\} _{n=1}^{\infty}$
is relatively compact in the space of probability measures on $\mathbf{G}(\R\times\Q_{S})/\mathbf{G}(\ZS)$,
and the group $\M$ above is semi-simple.
\end{thm}

For the rest of this section, we fix a finite set of odd primes $S$
and a sequence\emph{ }$\left\{ g_{n}\right\} _{n=1}^{\infty}\subseteq\SL_{d}(\Z)$
which meets the assumptions of Theorem \ref{thm:AESgrids thm}.

Recall that our goal is to find the limit of the measures $\mu_{g_{n},S}$
(defined in \eqref{eq:=00005Cmu_g,S}), but note that Theorem \ref{thm:go thm}
applies for a sequence of measures of the form $\left(x_{n}\right)_{*}\mu_{\pi(\tilde{\mathbf{L}}_{g_{n}}(\R\times\Q_{S}))\G\left(\ZS\right)}.$
As we will see in Section \ref{subsec:The-simply-connectedcover of L_g},
the subgroup $\pi\left(\tilde{\mathbf{L}}_{g_{n}}(\R\times\Q_{S})\right)\leq\mathbf{L}_{g_{n}}(\R\times\Q_{S})$
has a fixed finite index for all $n\in\N$ . Using this fact, we will
partition the orbits $O_{g_{n},S}$ (defined in \eqref{eq:def of O_g,s})
into finitely many pieces $O_{g_{n},S,\mathbf{i}}$ (defined in \eqref{eq:def of O_g_n,S,i}
below), and we will be able to apply Theorem \ref{thm:go thm} to
the sequence of natural measures $\mu_{g_{n},S,\mathbf{i}}$ (see
\eqref{eq:def of mu_g_n,S,i}) supported on $O_{g_{n},S,\mathbf{i}}$.
By finding the limiting measure of the sequence of $\mu_{g_{n},S,\mathbf{i}}$
for each choice of $\mathbf{i}$, we will obtain limit of the measures
$\mu_{g_{n},S}$ (see Section \ref{subsec:A-reduction--mu_g_S,i}).

\subsubsection{\label{subsec:The-simply-connectedcover of L_g}The universal covering
of $\mathbf{L}_{g}$ and of $\protect\G_{1}$.}

In the following we will recall some facts concerning the Spin group
which is the universal covering of an orthogonal group of a quadratic
form, and then we will be able to describe the subgroup $\pi\left(\tilde{\mathbf{L}}_{g_{n}}(\R\times\Q_{S})\right)$
in a useful way.

Assume that $m\geq3$ and let $\ef$ be a non-degenerate rational
quadratic form in $m$ variables. We denote by $\SO_{\ef}$ its special
orthogonal group, and for a field $\mathbb{F}\supseteq\Q$ we consider
the spinor norm $\phi:\SO_{\ef}(\mathbb{F})\to\mathbb{F}^{\times}/(\mathbb{F}^{\times})^{2}$
(see e.g. \cite[Chapter 10]{CASSELS_quad} for more details). We recall
that the spin group $\Spin_{\ef}$ is the simply connected covering
of $\SO_{\ef}$ (see \cite[Chapter 10]{CASSELS_quad}, or \cite[Section 2.3.2]{platonov_rapinchuk}
for more details) and we note the following exact sequences \textbf{(}see\textbf{
}\cite[Lemma 1]{elenberg_ven_rep}). For an odd prime $p$ it holds
\begin{equation}
\Spin_{\ef}(\Q_{p})\overset{\pi}{\to}\SO_{\ef}(\Q_{p})\overset{\phi}{\to}\Q_{p}^{\times}/\left(\Q_{p}^{\times}\right)^{2}\to0,\label{eq:p-adic spin exact sequence}
\end{equation}
for a positive definite $\ef$ we have
\begin{equation}
\Spin_{\ef}(\R)\overset{\pi}{\to}\SO_{\ef}(\R)\overset{\phi}{\to}0,\label{eq:exact sequence positive definite}
\end{equation}
and for an indefinite $\ef$ it holds
\begin{equation}
\Spin_{\ef}(\R)\overset{\pi}{\to}\SO_{\ef}(\R)\overset{\phi}{\to}\{\pm1\}\to0.\label{eq:exact sequence indefinite}
\end{equation}
We also note that $\pi(\Spin_{\ef}(\R))$ equals to the connected
component of $\SO_{\ef}(\R)$.

Returning to our case, we let $\ovec(g_{n})^{\perp(Q)}$ be the hyperplane
orthogonal to $\ovec(g_{n})$ with respect to the form $Q$. We observe
that $\text{Spin}_{Q\mid_{\ovec(g_{n})^{\perp(Q)}}}(\mathbb{F})$
naturally identifies with $\tilde{\mathbf{H}}_{\ovec(g_{n})}(\mathbb{F})$
(see \cite[Section 2.4, footnote 6]{elenberg_ven_rep}). Since we
assume that $\mathbf{H}_{\ovec(g_{n})}(\R)$ is compact, it follows
that $Q\mid_{\ovec(g_{n})^{\perp(Q)}}$ is positive definite. Therefore
we may conclude by \eqref{eq:exact sequence positive definite} that
\begin{equation}
\pi\left(\tilde{\mathbf{L}}_{g_{n}}(\R\times\Q_{S})\right)=\mathbf{L}_{g_{n}}(\R)\times\pi\left(\tilde{\mathbf{L}}_{g_{n}}(\Q_{S})\right).\label{eq:pi(L^tilde) is the product of all real with L^tilde p-adic}
\end{equation}
For $\mathbf{i}\in\prod_{p\in S}\Q_{p}^{\times}/\left(\Q_{p}^{\times}\right)^{2}$
we pick $h_{g_{n}}^{(\mathbf{i})}\in\mathbf{H}_{\ovec(g_{n})}(\Q_{S})$
such that $\phi(h_{g_{n}}^{(\mathbf{i})})=\mathbf{i}$. By \eqref{eq:p-adic spin exact sequence}
and \eqref{eq:pi(L^tilde) is the product of all real with L^tilde p-adic}
we deduce that $\left(e_{1,\infty},e_{2,\infty},h_{g_{n}}^{(\mathbf{i})},g_{n}^{-1}\o\left(h_{g_{n}}^{(\mathbf{i})}\right)g_{n}\right),$
$\mathbf{i}\in\prod_{p\in S}\Q_{p}^{\times}/\left(\Q_{p}^{\times}\right)^{2}$
is a complete set of representatives of $\pi\left(\tilde{\mathbf{L}}_{g_{n}}(\R\times\Q_{S})\right)$
cosets in $\mathbf{L}_{g_{n}}(\R\times\Q_{S})$. We define 
\begin{equation}
O_{g_{n},S,\mathbf{i}}\df l_{g_{n}}^{(\mathbf{i})}\pi\left(\tilde{\mathbf{L}}_{g_{n}}(\R\times\Q_{S})\right)\G(\ZS),\label{eq:def of O_g_n,S,i}
\end{equation}
where 
\begin{equation}
l_{g_{n}}^{(\mathbf{i})}\df\left(t_{g_{n}},h_{g_{n}}^{(\mathbf{i})},g_{n}^{-1}\o\left(h_{g_{n}}^{(\mathbf{i})}\right)g_{n}\right),\label{eq:def of l_g^(i)}
\end{equation}
(to recall $t_{g}$, see \eqref{eq:def of theta_g}) and we let 
\begin{equation}
\mu_{g_{n},S,\mathbf{i}}\df\left(l_{g_{n}}^{(\mathbf{i})}\right)_{*}\mu_{\pi\left(\tilde{\mathbf{L}}_{g_{n}}(\R\times\Q_{S})\right)\G(\ZS)},\label{eq:def of mu_g_n,S,i}
\end{equation}
 where $\mu_{\pi\left(\tilde{\mathbf{L}}_{g_{n}}(\R\times\Q_{S})\right)\G(\ZS)}$
is the left $\pi\left(\tilde{\mathbf{L}}_{g_{n}}(\R\times\Q_{S})\right)$-invariant
probability measure on the orbit $\pi\left(\tilde{\mathbf{L}}_{g_{n}}(\R\times\Q_{S})\right)\G(\ZS)$
. 

\subsubsection{\label{subsec:A-reduction--mu_g_S,i}A reduction - The limit of $\mu_{g_{n},S,\mathbf{i}}$
implies Theorem \ref{thm:AESgrids thm}}

We recall the following lemma from \cite{AESgrids}.
\begin{lem}
\label{lem:Lemma about the restriction of meeasure to cosets of subgroup}Let
$N\trianglelefteq K\leq G$ be locally compact groups such that $N$
is of index $k\in\N$ in $K$. Assume that $\Gamma\leq G$ is a lattice,
let $Kx\Gamma$ be a finite volume orbit and denote its $K$-invariant
probability measure by $\mu_{Kx\Gamma}.$ Then 
\[
\mu_{Kx\Gamma}=\frac{1}{k}\sum_{i=1}^{k}\mu_{k_{i}Nx\Gamma},
\]
where $k_{1},...,k_{N}$ is a complete list of representatives for
$N$ cosets in $K$ and $\mu_{k_{i}Nx\Gamma}$ is the $N$-invariant
probability measure on $k_{i}Nx\Gamma=Nk_{i}x\Gamma$.
\end{lem}

An immediate corollary from Lemma \ref{lem:Lemma about the restriction of meeasure to cosets of subgroup}
is that
\begin{equation}
\mu_{g_{n},S}=\frac{1}{k_{S}}\sum_{\mathbf{i}}\mu_{g_{n},S,\mathbf{i}},\label{eq:pieces of orbit}
\end{equation}
where $k_{S}=\left|\prod_{p\in S}\Q_{p}^{\times}/\left(\Q_{p}^{\times}\right)^{2}\right|$.

Next, for each $\mathbf{i}\in\prod_{p\in S}\Q_{p}^{\times}/\left(\Q_{p}^{\times}\right)^{2}$
we choose $\rho_{S}^{(\mathbf{i})}\in\SO_{Q}(\Q_{S})$ such that $\phi(\rho_{S}^{(\mathbf{i})})=\mathbf{i}$,
and we denote 
\[
\Y_{S,\mathbf{i}}\df(e_{1,\infty},e_{2,\infty},\rho_{S}^{(\mathbf{i})},e_{2,S})\pi\left(\tilde{\G}(\R\times\Q_{S})\right)\G(\ZS).
\]

We claim that
\begin{equation}
\pi\left(\tilde{\G}(\R\times\Q_{S})\right)\G(\ZS)=\left(\G(\R)\times\pi(\tilde{\G}(\Q_{S}))\right)\G(\ZS).\label{eq:the orgbit of the ambient simply connected}
\end{equation}
Indeed, recall that $\G_{2}$ is simply connected, so that 
\begin{equation}
\pi\left(\tilde{\G}(\R\times\Q_{S})\right)\G(\ZS)=\left(\pi(\tilde{\G}_{1}(\R))\times\G_{2}(\R)\times\pi(\tilde{\G}(\Q_{S}))\right)\G(\ZS),\label{eq:pi(G^tilde) times the zs pts}
\end{equation}
hence to prove \eqref{eq:the orgbit of the ambient simply connected}
it is sufficient by \eqref{eq:pi(G^tilde) times the zs pts} to show
that 
\[
\pi(\tilde{\G}_{1}(\R))\G_{1}(\ZS)=\G_{1}(\R)\G_{1}(\ZS).
\]
If $Q$ is positive definite then we deduce by \eqref{eq:exact sequence positive definite}
that $\pi(\tilde{\G}_{1}(\R))=\G_{1}(\R)$. If on the other-hand $Q$
is indefinite, we note that there exists $\gamma\in\G_{1}(\Z)$ with
$\phi(\gamma)=-\R^{\times}/\left(\R^{\times}\right)^{2}$ (there are
integral vectors $\mathbf{v}_{+}$ and $\mathbf{v}_{-}$ such that
$Q(\mathbf{v}_{\pm})\in\pm\R^{\times}$. The orthogonal transformation
$\gamma$ obtained by the composition of the associated reflections
$\gamma=\tau_{\mathbf{v}_{+}}\circ\tau_{\mathbf{v}_{-}}$ has $\phi(\gamma)=-\R^{\times}/\left(\R^{\times}\right)^{2}$),
which shows that

\begin{align*}
\pi(\tilde{\G}_{1}(\R))\G_{1}(\ZS)= & \left(\pi(\tilde{\G}_{1}(\R))\bigcup\pi(\tilde{\G}_{1}(\R))\gamma\right)\G_{1}(\ZS)\\
\underbrace{=}_{\eqref{eq:exact sequence indefinite}} & \G_{1}(\R)\G_{1}(\ZS).
\end{align*}

We let 
\[
\mu_{\Y_{S,\mathbf{i}}}\df(e_{1,\infty},e_{2,\infty},\rho_{S}^{(\mathbf{i})},e_{2,S})_{*}\mu_{\G(\R)\times\pi(\tilde{\G}(\Q_{S}))\G(\ZS)},
\]
where $\mu_{\G(\R)\times\pi(\tilde{\G}(\Q_{S}))\G(\ZS)}$ is the $\G(\R)\times\pi(\tilde{\G}(\Q_{S}))$
invariant probability measure supported on $\pi\left(\tilde{\G}(\R\times\Q_{S})\right)\G(\ZS)$.

Running over $\mathbf{i}\in\prod_{p\in S}\Q_{p}^{\times}/\left(\Q_{p}^{\times}\right)^{2}$
we obtain by \eqref{eq:p-adic spin exact sequence} that $(e_{1,\infty},e_{2,\infty},\rho_{S}^{(\mathbf{i})},e_{2,S})$
is a complete list of representatives of $\G(\R)\times\pi(\tilde{\G}(\Q_{S}))$
cosets in $\G(\R\times\Q_{S})$, and we conclude by Lemma \ref{lem:Lemma about the restriction of meeasure to cosets of subgroup}
that 
\begin{equation}
\mu_{\Y_{S}}=\frac{1}{k_{S}}\sum_{\mathbf{i}}\mu_{\Y_{S,\mathbf{i}}}.\label{eq:pieces of ambient space}
\end{equation}

We note that for each $\mathbf{i}\in\prod_{p\in S}\Q_{p}^{\times}/\left(\Q_{p}^{\times}\right)^{2}$
we have $O_{g_{n},S,\mathbf{i}}\subseteq\Y_{S,\mathbf{i}}$, and our
goal in the following will be to prove that 
\begin{equation}
\mu_{g_{n},S,\mathbf{i}}\to\Y_{S,\mathbf{i}},\label{eq:limit of =00005Cmu_g_n,S,i}
\end{equation}
which by \eqref{eq:pieces of orbit} and \eqref{eq:pieces of ambient space}
will imply Theorem \ref{thm:AESgrids thm}.

For the rest of the proof we fix $\mathbf{i}\in\prod_{p\in S}\Q_{p}^{\times}/\left(\Q_{p}^{\times}\right)^{2}$.
We now proceed to prove \eqref{eq:limit of =00005Cmu_g_n,S,i}, which
will be done in two steps.

\subsubsection{\label{subsec:First-step--}First step - the limit of $\underline{\mu}_{g_{n},S,\mathbf{i}}$}

We cannot apply Theorem \ref{thm:go thm} as is because $\Y_{S}$
is not a quotient of a semi-simple group. Therefore in our first step
below we project to a smaller space in which we can apply Theorem
\ref{thm:go thm}.

Denote $\Gbar_{2}\df\SL_{d-1}$, and define $\Gbar\df\G_{1}\times\Gbar_{2}$.
Consider the natural map
\[
\pi_{\Gbar_{2}}:\G_{2}\to\Gbar_{2},
\]
given by
\[
\left(\begin{array}{cc}
m & *\\
\mathbf{0} & 1
\end{array}\right)\mapsto m,
\]
and let $\pi_{\underline{\G}}:\G\to\Gbar$ be defined by $\pi_{\underline{\G}}(\rho,\eta)\df(\rho,\pi_{\Gbar_{2}}(\eta)),\ \forall\left(\rho,\eta\right)\in\G$.

We define 
\[
\underbar{\ensuremath{\Lbold}}_{g}\df\pi_{\underline{\G}}(\mathbf{L}_{g})=\left\{ \left(h,\pi_{\Gbar_{2}}\left(g^{-1}\o\left(h\right)g\right)\right)\mid\ h\in\mathbf{H}_{\ovec(g)}\right\} .
\]

The following lemma has essentially the same content as \cite[Lemma 3.4 ]{AESgrids}
(the proof is also essentially the same).
\begin{lem}
\label{lem:LEmma 3.4 of AES :)} Let \emph{$g\in\SL_{d}(\Z)$} such
that $Q(\ovec(g))>0$. Then:
\begin{enumerate}
\item \label{enu:AES_lemma3.4_maximality}$\mathbf{H}_{\ovec(g)}(\R)$ \emph{(}resp.
$\pi_{\Gbar_{2}}(g^{-1}\o(\mathbf{H}_{\tau(g)}(\R))g)$\emph{)} is
maximal among connected algebraic subgroups of $\G_{1}(\R)$ \emph{(}resp.
$\Gbar_{2}(\R)$\emph{).}
\item \label{enu:AES_lemma34._cetrelizer anisotropic}Assumption \nameref{=0000A72}
of Theorem \ref{thm:go thm} holds for $\underline{\mathbf{L}}_{g}$.
\item \label{enu:AES_lemma34._L_i strongly iso}Let $p$ be an odd prime,
and assume that there exists $\mathbf{u}\in\Q_{p}\otimes\ovec(g)^{\perp(Q)}$
such that $Q(\mathbf{u})=0$. Then assumption \nameref{=0000A71}
of Theorem \ref{thm:go thm} is valid for $\mathbf{L}_{g}(\Q_{p})$
and for $\underline{\mathbf{L}}_{g}(\Q_{p})\df\{(h,\pi_{\Gbar_{2}}(g^{-1}\o(h)g))\mid h\in\mathbf{H}_{\ovec(g)}\}$.
\end{enumerate}
\end{lem}

\begin{proof}
To obtain \eqref{enu:AES_lemma3.4_maximality} we recall by \cite{Dynkin_maximal}
that the stabilizer of a non-isotropic vector in a special orthogonal
group of a non-degenerate quadratic form is a maximal connected Lie
subgroup of $\G_{1}$, hence it follows that $\mathbf{H}_{\ovec(g)}(\R)$
is maximal among connected algebraic subgroups of $\G_{1}(\R)$. Next,
by Lemma \ref{lem:the second factor of L_g is orthogonal group of a form in d-1 var},
$\pi_{\Gbar_{2}}(g^{-1}\o(\mathbf{H}_{\tau(g)}(\R))g))$ is the stabilizer
of a non-degenerate quadratic form in $d-1$ variables, and since
$\mathbf{H}_{\ovec(g)}(\R)$ is compact, we have $\pi_{\Gbar_{2}}(g^{-1}\o(\mathbf{H}_{\tau(g)}(\R))g))\cong\SO_{d-1}(\R)$,
which is a well known maximal Lie subgroup of $\Gbar_{2}(\R)$. Next,
to prove \eqref{enu:AES_lemma34._cetrelizer anisotropic}, it is sufficient
to prove that the centralizer of $\mathbf{H}_{\ovec(g)}(\R)$ (resp.
$\pi_{\Gbar_{2}}(g^{-1}\o(\mathbf{H}_{\tau(g)}(\R))g)$) in $\G_{1}(\R)$
(resp. $\Gbar_{2}(\R)$) is finite. In fact, if not, we would obtain
a proper connected algebraic subgroup containing $\mathbf{H}_{\ovec(g)}(\R)$
(resp. $\pi_{\Gbar_{2}}(g^{-1}\o(\mathbf{H}_{\tau(g)}(\R))g)$, which
is a contradiction to \eqref{enu:AES_lemma3.4_maximality}. Finally,
if we assume that there exists $\mathbf{u}\in\Q_{p}\otimes\ovec(g)^{\perp(Q)}$
for an odd prime $p$ such that $Q(\mathbf{u})=0$, then by the proof
of \cite[Lemma 3.4]{AESgrids} we get that assumption \nameref{=0000A71}
of Theorem \ref{thm:go thm} is valid for $\mathbf{H}_{\ovec(g)}(\Q_{p})$.
Since $\mathbf{H}_{\ovec(g)}(\Q_{p})\cong\mathbf{L}_{g}(\Q_{p})\cong\underline{\mathbf{L}}_{g}(\Q_{p})$,
assumption \nameref{=0000A71} of Theorem \ref{thm:go thm} is valid
for $\mathbf{L}_{g}(\Q_{p})$ and $\underline{\mathbf{L}}_{g}(\Q_{p})$.
\end{proof}
Consider $\vartheta_{\underline{\G}}:\G(\R\times\Q_{S})/\G(\ZS)\to\Gbar(\R\times\Q_{S})/\Gbar(\ZS)$
be the map induced by $\pi_{\underline{\G}}$, and note that $\vartheta_{\underline{\G}}$
has compact fibers. We define $\X_{S,\mathbf{i}}\df\vartheta_{\underline{\G}}(\Y_{S,\mathbf{i}})$,
and $\underline{O}_{g,S,\mathbf{i}}\df\vartheta_{\underline{\G}}\left(O_{g,S,\mathbf{i}}\right)$,
which are equivalently described by
\[
\mathcal{X}_{S,\mathbf{i}}\df\pi\left(\tilde{\Gbar}(\R\times\Q_{S})\right)(e_{1,\infty},e_{2,\infty},\rho_{S}^{(\mathbf{i})},e_{2,S})\Gbar\left(\ZS\right),
\]
\[
\underline{O}_{g,S,\mathbf{i}}\df\underline{l}_{g_{n}}^{(\mathbf{i})}\pi\left(\tilde{\underline{\mathbf{L}}}_{g_{n}}(\R\times\Q_{S})\right)\Gbar(\ZS),
\]
where $\underline{l}_{g_{n}}^{(\mathbf{i})}\df\pi_{\underline{\G}}(l_{g_{n}}^{(\mathbf{i})})$
(see \eqref{eq:def of l_g^(i)} for the definition of $l_{g_{n}}^{(\mathbf{i})}$).

Let $\underbar{\ensuremath{\mu}}_{g_{n},S,\mathbf{i}}\df\left(\vartheta_{\underline{\G}}\right)_{*}\mu_{g_{n},S,\mathbf{i}}$,
and we note that 
\[
\underbar{\ensuremath{\mu}}_{g_{n},S,\mathbf{i}}=(\underline{l}_{g_{n}}^{(\mathbf{i})})_{*}\mu_{\pi(\underline{\tilde{\mathbf{L}}}_{g_{n}}(\R\times\Q_{S}))\G\left(\ZS\right)},
\]
where $\mu_{\pi(\underline{\tilde{\mathbf{L}}}_{g_{n}}(\R\times\Q_{S}))\G\left(\ZS\right)}$
is the $\pi\left(\tilde{\underbar{\ensuremath{\mathbf{L}}}}_{g_{n}}(\R\times\Q_{S})\right)$-invariant
probability measure supported on $\pi\left(\tilde{\underbar{\ensuremath{\mathbf{L}}}}_{g_{n}}(\R\times\Q_{S})\right)\Gbar(\ZS)$.

Let $\nu$ be a weak star limit of a subsequence $\underline{\mu}_{g_{n},S,\mathbf{i}}$,
$n\in C_{1}\subseteq\N$. Then by Lemma \ref{lem:LEmma 3.4 of AES :)},\eqref{enu:AES_lemma34._cetrelizer anisotropic}
and Theorem \ref{thm:go thm}, $\underbar{\ensuremath{\nu}}$ is a
probability measure and there exists a semi-simple connected $\Q$-algebraic
subgroup $\M\leq\G$ such that 
\begin{equation}
\underbar{\ensuremath{\nu}}=\left(t_{0}\right)_{*}\mu_{M\G(\ZS)},\label{eq:nu_bar}
\end{equation}
where $M$ is a closed finite index subgroup of $\M(\R\times\Q_{S})$
and $t_{0}\in\G(\R\times\Q_{S}).$

For the rest of this section, our goal will be to prove that $\M=\underbar{\ensuremath{\G}},$
and as we now show, this will prove that 
\begin{equation}
\underbar{\ensuremath{\nu}}=\mu_{\mathcal{X}_{S,\mathbf{i}}}\df\left(\vartheta_{\underline{\G}}\right)_{*}\mu_{\Y_{S,\mathbf{i}}},\label{eq:nu=00003Dmu_X_S,i}
\end{equation}
which is the unique $\pi\left(\tilde{\Gbar}(\R\times\Q_{S})\right)$-invariant
probability measure on $\X_{S,\mathbf{i}}.$

So assume that $M\leq\Gbar$ is of finite index. We now show that
\begin{equation}
\pi(\tilde{\Gbar}(\R\times\Q_{S}))\subseteq M.\label{eq:M includes the image of the simply connected cover}
\end{equation}
Since $\pi(\tilde{\Gbar}(\R))\times\{e_{S}\}$ is the connected component
of $\Gbar(\R)\times\{e_{S}\}$, we get\textbf{
\begin{equation}
\pi(\tilde{\Gbar}(\R))\times\{e_{S}\}\subseteq M\cap\left(\Gbar(\R)\times\{e_{S}\}\right).\label{eq:pi(G^tilde) in the real component of M}
\end{equation}
}Let $\Gbar^{+}(\Q_{S})$ the group generated by unipotent elements
of $\Gbar(\Q_{S})$. By Corollary 6.7 of \cite{borel_tits_homo_abstrait},
any subgroup of finite index contains the group $\Gbar^{+}(\Q_{S})$.
Since $M\cap\left(\{e_{\infty}\}\times\Gbar(\Q_{S})\right)\leq\{e_{\infty}\}\times\Gbar(\Q_{S})$
is of finite index, we deduce 
\begin{equation}
\{e_{\infty}\}\times\Gbar^{+}(\Q_{S})\subseteq M\cap\left(\{e_{\infty}\}\times\Gbar(\Q_{S})\right).\label{eq:pi(G^tilde) in the s-adic component of M}
\end{equation}
Since we assume that $Q$ is isotropic for all\textbf{ $p\in S$},\textbf{
}by Lemma 1 of \cite{elenberg_ven_rep} we have that $\G_{1}^{+}(\Q_{S})=\pi(\tilde{\G}_{1}(\Q_{S}))$,
and it is well known that $\Gbar_{2}^{+}(\Q_{S})=\Gbar_{2}(\Q_{S})=\pi\left(\tilde{\Gbar}_{2}(\Q_{S})\right)$.
Thus we conclude that
\begin{equation}
\Gbar^{+}(\Q_{S})=\pi(\tilde{\Gbar}(\Q_{S})),\label{eq:G^plus equals}
\end{equation}
and by \eqref{eq:pi(G^tilde) in the real component of M}, \eqref{eq:pi(G^tilde) in the s-adic component of M},
and \eqref{eq:G^plus equals} we deduce that \eqref{eq:M includes the image of the simply connected cover}
holds. Since for all $n$, the measure $\underbar{\ensuremath{\mu}}_{g_{n},S,\mathbf{i}}$
is supported on $\underline{O}_{g,S,\mathbf{i}}\subseteq\mathcal{X}_{S,\mathbf{i}}$,
we deduce that $t_{0}M\Gbar(\ZS)\subseteq\mathcal{X}_{S,\mathbf{i}}$,
and by \eqref{eq:M includes the image of the simply connected cover}
we conclude that $t_{0}M\Gbar(\ZS)\supseteq\mathcal{X}_{S,\mathbf{i}}$,
which shows the implication $\M=\Gbar\implies$ \eqref{eq:nu=00003Dmu_X_S,i}.

Now assume for contradiction that $\M\lneq\Gbar$. Let 
\[
\pi_{1}:\Gbar\to\G_{1},\ \pi_{2}:\Gbar\to\Gbar_{2},
\]
be the natural maps. Since $\M$ is semi-simple and since $\G_{1}$
and $\Gbar_{2}$ have no isomorphic simple Lie factors (due to ambient
dimensions, accidental isomorphisms play no role), it follows that
$\pi_{1}\left(\M\right)\lneq\G_{1}$ or $\pi_{2}\left(\M\right)\lneq\Gbar_{2}$.

By Theorem \ref{thm:go thm}, we let $\left\{ \gamma_{g_{n}}\right\} _{n=1}^{\infty}\subseteq\Gbar(\ZS)$
and a further subsequence $C_{2}\subseteq C_{1}$ such that $\left|C_{1}\smallsetminus C_{2}\right|<\infty$,
which satisfies
\begin{equation}
\gamma_{g_{n}}^{-1}\underbar{\ensuremath{\mathbf{L}}}_{g_{n}}\gamma_{g_{n}}\subseteq\M,\ \forall n\in C_{2},\label{eq:conj of L_v into M}
\end{equation}
and we let $\left\{ l_{n}\right\} _{n=1}^{\infty}\subseteq\pi(\underline{\tilde{\mathbf{L}}}_{g_{n}}(\R\times\Q_{S}))$
such that 
\begin{equation}
(\underline{l}_{g_{n}}^{(\mathbf{i})})l_{g_{n}}\gamma_{g_{n}}\to t_{0}.\label{eq:conv to t_0}
\end{equation}

\subsubsection*{In case $\pi_{1}\left(\protect\M\right)\lneq\protect\G_{1}$}

Let $\delta_{g_{n}}\df\pi_{1}(\gamma_{g_{n}})\in\G_{1}(\ZS)$. Since
$\pi_{1}\left(\M\right)$ is a strict, connected, semi-simple $\Q$
subgroup of $\G_{1}$, we obtain that $\pi_{1}(\M(\R))\lneq\G_{1}(\R)$,
and by maximality of the subgroups $\mathbf{H}_{\ovec(g_{n})}(\R)$,
(see Lemma \ref{lem:LEmma 3.4 of AES :)},\eqref{enu:AES_lemma3.4_maximality})
we obtain that for all $i,j\in C_{2}$ 
\[
\delta_{g_{i}}^{-1}\mathbf{H}_{\ovec(g_{i})}(\R)\delta_{g_{i}}=\delta_{g_{j}}^{-1}\mathbf{H}_{\ovec(g_{j})}(\R)\delta_{g_{j}}=\pi_{1}(\M(\R)),
\]
which implies that 
\begin{equation}
\mathbf{H}_{\delta_{g_{i}}^{-1}\ovec\left(g_{i}\right)}(\R)=\mathbf{H}_{\delta_{g_{j}}^{-1}\ovec(g_{j})}(\R).\label{eq:stab groups equal}
\end{equation}
We fix $i\in C_{2}$, and by \eqref{eq:stab groups equal} we may
deduce that for each $j\in C_{2}$ there exists $0\neq\alpha_{j}\in\ZS$
such that 
\begin{equation}
\delta_{g_{i}}^{-1}\ovec(g_{i})=\alpha_{j}\delta_{g_{j}}^{-1}\ovec(g_{j}).\label{eq:delta_v equality}
\end{equation}
We will now show that $\left\{ \alpha_{j}\right\} _{j\in C_{2}}$
is bounded and bounded away from $0$, which will be a contradiction
since $Q(\ovec(g_{j}))\to\infty$, since $i$ is fixed, and since
by \eqref{eq:delta_v equality} we have
\[
Q\left(\ovec(g_{i})\right)=Q\left(\delta_{g_{i}}^{-1}\ovec(g_{i})\right)=\alpha_{j}^{2}Q\left(\ovec(g_{j})\right).
\]

By recalling that $\left\{ \ovec(g_{j})\right\} _{j\in C_{2}}$ is
a sequence of primitive integral vectors, we deduce that $\left\{ \delta_{g_{j}}^{-1}\ovec(g_{j})\right\} _{j\in C_{2}}$
are primitive vectors in $\ZS^{d}$ considered as a $\ZS$ module.
This implies that $\alpha_{j}\in\ZS^{\times}$ where 
\[
\ZS^{\times}=\left\{ \prod_{p\in S}p^{n_{p}}\mid n_{p}\in\Z\right\} .
\]
By \eqref{eq:conv to t_0} we obtain a sequence $\left\{ h_{g_{j}}\right\} _{j\in C_{2}}$
with $h_{g_{j}}\in\pi\left(\tilde{\mathbf{H}}_{\ovec\left(g_{j}\right)}(\Q_{S})\right)$
such that 
\[
h_{g_{j}}^{(\mathbf{i})}h_{g_{j}}\delta_{g_{j}}\to\pi_{1,S}(t_{0}),
\]
where $\pi_{1,S}:\Gbar(\R\times\Q_{S})\to\G_{1}(\Q_{S})$ is the natural
map and $t_{0}$ is given in \eqref{eq:nu_bar}. By multiplying both
sides of \eqref{eq:delta_v equality} with $h_{g_{j}}^{(\mathbf{i})}h_{g_{j}}\delta_{g_{j}}$,
we obtain that 
\begin{equation}
\lim_{C_{2}\ni j\to\infty}\alpha_{j}\ovec(g_{j})=\pi_{1,S}(t_{0})\delta_{g_{i}}^{-1}\ovec(g_{i}).\label{eq:limit alf_j*v_j}
\end{equation}
Since $\ovec(g_{j})$ is a primitive integral vector, $\norm{\ovec(g_{j})}_{p}$
(the maximum of the p-adic valuations of the entries) is constant
in $j$ for all $p\in S$. Thus, by \eqref{eq:limit alf_j*v_j}, the
p-adic valuation of $\alpha_{j}$ is bounded, and since $\alpha_{j}\in\ZS^{\times}$,
we conclude that $\left\{ \alpha_{j}\right\} _{j\in C_{2}}$ is bounded
and bounded away from $0$.

\subsubsection*{In case $\pi_{2}\left(\protect\M\right)\lneq\protect\Gbar_{2}$}

We will obtain a contradiction in a similar way as we had in the case
that $\pi_{1}\left(\M\right)\lneq\Gbar_{2}$. We denote $\eta_{g_{n}}\df\pi_{2}(\gamma_{g_{n}})\in\Gbar_{2}(\ZS)$.
Since $\pi_{2}\left(\M\right)$ is a strict, connected semi-simple
$\Q$ subgroup of $\Gbar_{2}$, we obtain by maximality (see Lemma
\ref{lem:LEmma 3.4 of AES :)},\eqref{enu:AES_lemma3.4_maximality})
and by recalling \eqref{eq:conj of L_v into M}, that for all $i,j\in C_{2}$
\begin{equation}
\eta_{g_{i}}^{-1}\pi_{\Gbar_{2}}\left(g_{i}^{-1}\o\left(\mathbf{H}_{\ovec(g_{i})}\left(\R\right)\right)g_{_{i}}\right)\eta_{g_{i}}=\eta_{g_{j}}^{-1}\pi_{\Gbar_{2}}\left(g_{j}^{-1}\o\left(\mathbf{H}_{\ovec(g_{j})}(\R)\right)g_{j}\right)\eta_{g_{j}}.\label{eq:equality_eta_v_stab_H_phi}
\end{equation}
By Lemma \ref{lem:the second factor of L_g is orthogonal group of a form in d-1 var},
we find that \eqref{eq:equality_eta_v_stab_H_phi} can be rewritten
by
\[
\SO_{\ef_{g_{j}}^{\eta_{g_{j}}}}(\R)=\SO_{\ef_{g_{i}}^{\eta_{g_{i}}}}(\R),
\]
where the quadratic form $\ef_{g}^{\gamma}$ is a given by \eqref{eq:quadratic form phi_g^=00005Cgamma}.
By recalling Lemma 3.3 of \cite{AESgrids}, we find that there exists
$\alpha_{j}\in\Q^{\times}$ such that 
\begin{equation}
\alpha_{j}\ef_{g_{j}}^{\eta_{g_{j}}}=\ef_{g_{i}}^{\eta_{g_{i}}},\label{eq:quadratic forms are scallar dependent}
\end{equation}
where we fix $i$ and let $j\in C_{2}$ vary. Our plan now is to show
that $\{\alpha_{j}\}_{j\in C_{2}}$ is bounded and bounded away from
$0$. This will be a contradiction since we assume that $Q(\ovec(g_{j}))\to\infty$
and since by Lemma \ref{lem:discreminant of M_ef_g^=00005Cgamma}
we have that $\text{disc}(\ef_{g_{j}}^{\eta_{g_{j}}})=\frac{1}{\text{disc}(Q)}Q(\ovec(g_{j}))$,
where $\text{disc}(\ef)$ denotes the determinant of the companion
matrix of a quadratic form $\ef$.

We recall that (see \eqref{eq:def of companion matrix =00005Cef_g^=00005Cgamma})
\[
\ef_{g}^{\eta}(\mathbf{u})=\mathbf{u}^{t}\left(\eta^{t}\hat{g}^{t}M^{-1}\hat{g}\eta\right)\mathbf{u},
\]
where $\hat{g}$ is the matrix formed by the first $d-1$ columns
of $g$ and where $M$ is the companion matrix of $Q$. Therefore,
by \eqref{eq:quadratic forms are scallar dependent} we deduce 
\[
\alpha_{j}(\eta_{j}^{t}\hat{g_{j}}^{t}M^{-1}\hat{g_{j}}\eta_{j})=\eta_{i}^{t}\hat{g_{i}}^{t}M^{-1}\hat{g_{i}}\eta_{i},
\]
which in turn implies that 
\begin{equation}
\alpha_{j}(\eta_{j}^{t}\hat{g_{j}}^{t}\text{adj}(M)\hat{g_{j}}\eta_{j})=\eta_{i}^{t}\hat{g_{i}}^{t}\text{adj}(M)\hat{g_{i}}\eta_{i},\label{eq:scallar with conjugate adj(M)}
\end{equation}
where $\text{adj}(M)$ is the matrix adjugate of $M$, which has integral
entries as $M$ is integral. We denote
\[
\text{\ensuremath{\bar{M}}}_{\ef_{g}^{\eta}}\df\eta^{t}\hat{g}^{t}\text{adj}(M)\hat{g}\eta,
\]
and for $l\in C_{2}$ we let $q_{l}\in\N$ be defined by 
\[
q_{l}\df g.c.d(\hat{g_{l}}^{t}\text{adj}(M)\hat{g_{l}}).
\]
We rewrite \eqref{eq:scallar with conjugate adj(M)} to
\begin{equation}
\frac{\alpha_{j}q_{j}}{q_{i}}\left(\frac{1}{q_{j}}\bar{M}_{\ef_{g_{j}}^{\eta_{j}}}\right)=\frac{1}{q_{i}}\bar{M}_{\ef_{g_{i}}^{\eta_{i}}},\label{eq:equality of M_phi}
\end{equation}
and by noting that $\frac{1}{q_{l}}\hat{g_{l}}^{t}\text{adj}(M)\hat{g_{l}}$
has co-primes entries, we may deduce that $\frac{\alpha_{j}q_{j}}{q_{i}}\in\left(\ZS\right)^{\times}$.

By \eqref{eq:conv to t_0} there exists a sequence $\left\{ k_{g_{j}}\right\} _{j\in C_{2}}$
with $k_{g_{j}}\in\SO_{\ef_{g_{j}}}(\Q_{S})$ such that 
\begin{equation}
\underbar{\ensuremath{k}}_{g_{j}}^{(\mathbf{i})}k_{g_{j}}\eta_{g_{j}}\to\pi_{2,S}\left(t_{0}\right),\label{eq:limit of l_j's}
\end{equation}
where $\underbar{\ensuremath{k}}_{g_{j}}^{(\mathbf{i})}\df\pi_{\Gbar_{2}}\left(g_{j}^{-1}\o\left(h_{g_{j}}^{(\mathbf{i})}\right)g_{j}\right)\in\SO_{\ef_{g_{j}}}(\Q_{S})$,
$\pi_{2,S}:\Gbar(\R\times\Q_{S})\to\Gbar_{2}(\Q_{S})$ is the natural
map and $t_{0}$ is given in \eqref{eq:nu_bar}. We conclude by denoting
$\bar{M}_{\ef_{g_{j}}}=\bar{M}_{\ef_{g_{j}}^{e}}$, and by noting
that $\bar{M}_{\ef_{g_{j}}}$ is a multiple of the companion matrix
of the quadratic form $\ef_{g_{j}}$, that

\begin{align}
\left(\left(\underbar{\ensuremath{k}}_{g_{j}}^{(\mathbf{i})}k_{g_{j}}\eta_{g_{j}}\right)^{t}\right)^{-1}\bar{M}_{\ef_{g_{j}}^{\eta_{g_{j}}}}\left(\underbar{\ensuremath{k}}_{g_{j}}^{(\mathbf{i})}k_{g_{j}}\eta_{g_{j}}\right)^{-1}= & \left(\left(\underbar{\ensuremath{k}}_{g_{j}}^{(\mathbf{i})}k_{g_{j}}\right)^{-1}\right)^{t}\bar{M}_{\ef_{g_{j}}}\left(\underbar{\ensuremath{k}}_{g_{j}}^{(\mathbf{i})}k_{g_{j}}\right)^{-1}\label{eq:conju of M_ef_g_j^=00005Ceta}\\
= & \bar{M}_{\ef_{g_{j}}}.\nonumber 
\end{align}
To simplify notation, we denote the fixed matrix $\frac{1}{q_{i}}\bar{M}_{\ef_{g_{i}}^{\eta_{g_{i}}}}$
by $B$ and we deduce by \eqref{eq:equality of M_phi} and \eqref{eq:conju of M_ef_g_j^=00005Ceta}
that 
\begin{equation}
\frac{\alpha_{j}q_{j}}{q_{i}}\left(\frac{1}{q_{j}}\bar{M}_{g_{j}}\right)=\left(\left(\underbar{\ensuremath{k}}_{g_{j}}^{(\mathbf{i})}k_{g_{j}}\eta_{g_{j}}\right)^{t}\right)^{-1}B\left(\left(\underbar{\ensuremath{k}}_{g_{j}}^{(\mathbf{i})}k_{g_{j}}\eta_{g_{j}}\right)\right)^{-1}.\label{eq:conjugation of companion matrices by product of l_g's}
\end{equation}
We conclude by \eqref{eq:limit of l_j's} and \eqref{eq:conjugation of companion matrices by product of l_g's}
that the p-adic norm of $\frac{\alpha_{j}q_{j}}{q_{i}}\left(\frac{1}{q_{j}}\bar{M}_{g_{j}}\right)$
is bounded for all $p\in S$, and since $\frac{1}{q_{j}}\bar{M}_{g_{j}}$
is a primitive integral matrix, the p-adic norm of $\frac{\alpha_{j}q_{j}}{q_{i}}\left(\frac{1}{q_{j}}\bar{M}_{g_{j}}\right)$
equals to the p-adic valuation $\left|\frac{\alpha_{j}q_{j}}{q_{i}}\right|_{p}$
for all $p\in S$. Since $\left\{ \frac{\alpha_{j}q_{j}}{q_{i}}\right\} _{j\in C_{2}}\subseteq\left(\ZS\right)^{\times}$,
we conclude that $\left\{ \frac{\alpha_{j}q_{j}}{q_{i}}\right\} _{j\in C_{2}}$
is bounded in absolute value from above and away from $0$.

Finally, using Lemma \ref{lem:gcd of companion matrix} we deduce
that $q_{j}$ is uniformly bounded in $j\in C_{2}$ from above and
below, which implies in turn that $\alpha_{j}$ is bounded in $j\in C_{2}$
from above and away from $0$.

\subsubsection{Second step - Upgrading to $\protect\G$}

In a summary of the first step, it holds that
\begin{equation}
\left(\vartheta_{\underline{\G}}\right)_{*}\mu_{g_{n},S,\mathbf{i}}=\underbar{\ensuremath{\mu}}_{g_{n},S,\mathbf{i}},\label{eq:pi_S of mu_v_n,S}
\end{equation}
and it holds that $\underbar{\ensuremath{\mu}}_{g_{n},S,\mathbf{i}}\to\mu_{\mathcal{X}{}_{S,\mathbf{i}}}$,
where 
\begin{equation}
\left(\vartheta_{\underline{\G}}\right)_{*}\mu_{\Y_{S,\mathbf{i}}}=\mu_{\mathcal{X}{}_{S,\mathbf{i}}}.\label{eq:pi_s of Y_s}
\end{equation}
Let $\nu$ be a weak-star limit of a subsequence $\left\{ \mu_{g_{n},S,\mathbf{i}}\right\} _{n\in C_{1}}$,
for $C_{1}\subseteq\N$. Using \eqref{eq:pi_S of mu_v_n,S} and \eqref{eq:pi_s of Y_s},
we deduce that $\nu$ is a probability measure.

In order to prove that $\nu=\mu_{\Y_{S,\mathbf{i}}}$, we will apply
Theorem \ref{thm:go thm} in the ambient space
\[
\G'(\R\times\Q_{S})/\G'(\ZS),
\]
where $\G'\df\G_{1}\times\SL_{d}$.

By Theorem \ref{thm:go thm} and Lemma \ref{lem:LEmma 3.4 of AES :)},\eqref{enu:AES_lemma34._L_i strongly iso}
there exists a connected $\Q$-algebraic subgroup $\M\leq\G'$ such
that 
\begin{equation}
\nu=\left(t_{0}\right){}_{*}\mu_{M\G'(\ZS)},\label{eq:nu}
\end{equation}
where $M\leq\M(\R\times\Q_{S})$ is a closed finite index subgroup
and $t_{0}\in\G'(\R\times\Q_{S}).$

As explained in \cite{AESgrids} (see below equation $(4.5)$ in \cite{AESgrids}),
it follows that $\M\le\G$, that $t_{0}\in\G(\R\times\Q_{S})$, that
there exists a sequence $\left\{ \gamma_{g_{n}}\right\} _{n\in C_{2}}\subseteq\G(\ZS)$,
where $C_{2}\subseteq C_{1}$, $\left|C_{1}\smallsetminus C_{2}\right|<\infty$
such that 
\begin{equation}
\gamma_{g_{n}}^{-1}\mathbf{L}_{g_{n}}\gamma_{g_{n}}\subseteq\M,\label{eq:conjugation of L_g into M}
\end{equation}
and that either $\M=\G$ or $\M=\G_{1}\times\SL_{d-1}^{\mathbf{t}}$
where 
\[
\SL_{d-1}^{\mathbf{t}}\df c_{\mathbf{t}}\iota\left(\SL_{d-1}\right)c_{\mathbf{t}}^{-1},\ c_{\mathbf{t}}\df\left(\begin{array}{cc}
I_{d-1} & \mathbf{t}\\
\mathbf{0} & 1
\end{array}\right),\ \mathbf{t}\in\Q^{d-1},
\]
and $\iota:\SL_{d-1}\to\ASL_{d-1}$ is the natural embedding which
maps $m\mapsto\left(\begin{array}{cc}
m & \mathbf{0}\\
\mathbf{0} & 1
\end{array}\right)$. As in Section \ref{subsec:First-step--} (by the same argument),
the proof will be done once we show that $\M=\G$.

Assume by contradiction that $\M=\G_{1}\times\SL_{d-1}^{\mathbf{t}}$.
We let $\pi_{2}:\G\to\G_{2}$ be the coordinate map, and we denote
$\eta_{g}\df\pi_{2}(\gamma_{g})$. By definition of $\SL_{d-1}^{\mathbf{t}}$
and by \eqref{eq:conjugation of L_g into M} we obtain $\forall n\in C_{2}$
that 
\begin{equation}
c_{\mathbf{t}}^{-1}\eta_{g_{n}}^{-1}g_{n}^{-1}\o\left(\mathbf{H}_{\ovec(g_{n})}\right)g_{n}\eta_{g_{n}}c_{\mathbf{t}}\subseteq\iota\left(\SL_{d-1}\right).\label{eq:inclusion of H_v by conjugation of c_t}
\end{equation}
We fix $N\in\N$ such that $N\mathbf{t}\in\Z_{\text{prim }}^{d-1}$.
By using that $\iota\left(\SL_{d-1}\right)$ fixes $\mathbf{e}_{d}$
and by using \eqref{eq:inclusion of H_v by conjugation of c_t} we
conclude that
\begin{equation}
\tilde{\mathbf{v}}_{n}\df\left(g_{n}\eta_{g_{n}}c_{\mathbf{t}}\right)\left(Ne_{d}\right),\label{eq:def_of_v^tilde_g}
\end{equation}
is fixed by the left linear action of $\o\left(\mathbf{H}_{\ovec(g_{n})}\right)$.
By Lemma \ref{lem:F(SO_Q) and stabilizer of a vector}, the group
$\o(\mathbf{H}_{\ovec(g_{n})})$ is the stabilizer subgroup of the
non-isotropic vector $M\ovec(g_{n})$ under the left linear action
of $\SO_{Q^{*}}$. The space of fixed vectors for such groups is one-dimensional,
hence there exists $\alpha_{g_{n}}\in\Q$ such that
\begin{equation}
\alpha_{g_{n}}\left(M\ovec(g_{n})\right)=\tilde{\mathbf{v}}_{n}.\label{eq:equality v_n and v^tilde_n}
\end{equation}
Again as above, we will show that $\left\{ \alpha_{g_{n}}\right\} _{n\in C_{2}}$
is bounded and bounded away from $0$.

Before continuing, we will now explain why the boundedness of $\left\{ \alpha_{g_{n}}\right\} _{n\in C_{2}}$
yields a contradiction. By definition of $\tilde{\mathbf{v}}_{n}$
in \eqref{eq:def_of_v^tilde_g}, we may express $\tilde{\mathbf{v}}_{n}$
by 
\[
\tilde{\mathbf{v}}_{n}=\sum_{i=1}^{d-1}a_{i}\left(g_{n}\mathbf{e}_{i}\right)+N\left(g_{n}\mathbf{e}_{d}\right),
\]
where $a_{1},...,a_{d-1}\in\Q$. We now observe that

\begin{align*}
\alpha_{g_{n}}Q(\ovec(g_{n}))= & \alpha_{g_{n}}\ovec(g_{n})^{t}M\ovec(g_{n})\underbrace{=}_{\eqref{eq:equality v_n and v^tilde_n}}\ovec(g_{n})^{t}\tilde{\mathbf{v}}_{n}\\
= & \sum_{i=1}^{d-1}a_{i}\underbrace{\left\langle \ovec(g_{n}),g_{n}\mathbf{e}_{i}\right\rangle }_{=0}+N\underbrace{\left\langle \ovec(g_{n}),g_{n}\mathbf{e}_{d}\right\rangle }_{=1}\\
= & N,
\end{align*}
and since $Q(\tau(g_{n}))\to\infty$ and $N$ is fixed, this will
be a contradiction.

We now proceed to show the boundedness of $\left\{ \alpha_{g_{n}}\right\} _{n\in C_{2}}$.
We denote for $n\in\N$ by $q_{g_{n}}$ the g.c.d of $M\ovec(g_{n})$,
and we rewrite \eqref{eq:equality v_n and v^tilde_n} by 
\[
(\alpha_{g_{n}}q_{g_{n}})\left(\frac{1}{q_{g_{n}}}M\ovec(g_{n})\right)=g_{n}\eta_{g_{n}}(Nc_{\mathbf{t}}\mathbf{e}_{d})
\]
Using that $\frac{1}{q_{g_{n}}}M\ovec(g_{n})$ and $Nc_{\mathbf{t}}\mathbf{e}_{d}$
are primitive integral vectors, we deduce by the preceding equality
that $\alpha_{g_{n}}q_{g_{n}}\in\ZS^{\times}$. By Theorem \ref{thm:go thm},\eqref{enu:go_3}
there exists a sequence $\left\{ h_{g_{n}}\right\} _{n\in C_{2}}$
with $h_{g_{n}}\in\pi\left(\tilde{\mathbf{H}}_{\ovec(g)}(\Q_{S})\right)$
such that 
\[
h_{g_{n}}^{(\mathbf{i})}h_{g_{n}}\delta_{g_{n}}\to\pi_{1,S}\left(t_{0}\right),
\]
where $\delta_{g_{n}}\df\pi_{1}\left(\gamma_{g_{n}}\right)$, $\pi_{1,S}:\G(\R\times\Q_{S})\to\G_{1}(\Q_{S})$
is the natural map, $t_{0}$ is given in \eqref{eq:nu}, and 
\begin{equation}
g_{n}^{-1}\o\left(h_{g_{n}}^{(\mathbf{i})}h_{g_{n}}\right)g_{n}\eta_{g_{n}}\to\pi_{2,S}\left(t_{0}\right),\label{eq:limit of conjugate h_v^(i)h_v_j  in grids case}
\end{equation}
where $\eta_{g_{n}}\df\pi_{2}\left(\gamma_{g_{n}}\right)$, $\pi_{2,S}:\G(\R\times\Q_{S})\to\G_{2}(\Q_{S})$
is the natural map. We obtain by \eqref{eq:equality v_n and v^tilde_n}
and by recalling that $\o(h_{g_{n}}^{(\mathbf{i})}h_{g_{n}})$ stabilizes
$M\ovec(g_{n})$ (see Lemma \ref{lem:F(SO_Q) and stabilizer of a vector}),
that
\begin{equation}
\begin{aligned}\alpha_{g_{n}}M\ovec(g_{n})= & \o(h_{g_{n}}^{(\mathbf{i})}h_{g_{n}})\tilde{\mathbf{v}}_{n}\\
\underbrace{=}_{\text{recalling \eqref{eq:def_of_v^tilde_g}}} & g_{n}\left(g_{n}^{-1}\o(h_{g_{n}}^{(\mathbf{i})}h_{g_{n}})g_{n}\eta_{g_{n}}\right)c_{\mathbf{t}}\left(Ne_{d}\right).
\end{aligned}
\label{eq:g*limiting to g_2,S*c_T}
\end{equation}
Since $g_{n}\in\SL_{d}(\Z)$, we get for any $p\in S$ that
\[
\norm{g_{n}\left(g_{n}^{-1}\o(h_{g_{n}}^{(\mathbf{i})}h_{g_{n}})_{p}g_{n}\eta_{g_{n}}\right)c_{\mathbf{t}}\left(Ne_{d}\right)}_{p}=\norm{\left(g_{n}^{-1}\o(h_{g_{n}}^{(\mathbf{i})}h_{g_{n}})_{p}g_{n}\eta_{g_{n}}\right)c_{\mathbf{t}}\left(Ne_{d}\right)}_{p},
\]
where $\norm{\cdot}_{p}$ is the maximum of p-adic valuations of the
entries, and $\o(h_{g_{n}}^{(\mathbf{i})}h_{g_{n}})_{p}$ is the $p$'th
component of $\o(h_{g_{n}}^{(\mathbf{i})}h_{g_{n}})\in\G(\Q_{S})$.
By \eqref{eq:limit of conjugate h_v^(i)h_v_j  in grids case} and
\eqref{eq:g*limiting to g_2,S*c_T} we deduce for all $p\in S$ that
the $p$-adic valuation of $\norm{\alpha_{g_{n}}M\ovec(g_{n})}_{p}$
is bounded. Since $\frac{1}{q_{g_{n}}}M\ovec(g_{n})$ is a primitive
integral vector, we get that 
\[
\norm{\alpha_{g_{n}}q_{g_{n}}\left(\frac{1}{q_{g_{n}}}M\ovec(g_{n})\right)}_{p}=\left|\alpha_{g_{n}}q_{g_{n}}\right|_{p},
\]
which implies in turn that $\left|\alpha_{g_{n}}q_{g_{n}}\right|_{p}$
is bounded in $n\in\N$, for all $p\in S$. By recalling that $\alpha_{g_{n}}q_{g_{n}}\in(\ZS)^{\times},$
we conclude that $\left\{ \alpha_{g_{n}}q_{g_{n}}\right\} _{n\in C_{2}}$
is bounded and bounded away from $0$. Finally, since $M\Z^{d}\subseteq\Z^{d}$
and since $M\ovec(g_{n})\in M\Z^{d}$ is a primitive vector in the
lattice $M\Z^{d}$, we get by \cite{Cas_geo_num}, Chapter 1, Theorem
1,B. that $q_{g_{n}}\leq\det(M)$, which completes the proof.

\section{\label{sec:Equivalence-classes-of integral points} Equivalence classes
of integral points and their relation to the S-arithmetic orbits}

In this section we define for each $T>0$ an equivalence relation
on $\V_{T}(\Z)$ for which there are finitely many equivalence classes
$E_{g_{1}}\bigsqcup...\bigsqcup E_{g_{N}}=\V_{T}(\Z)$, see Section
\ref{subsec:The-equivelence-relation}. The motivation for this equivalence
relation is a connection established in Section \ref{subsec:duality principle}
between each equivalence class $E_{g}$ and the orbit $O_{g,S}$ (the
main result is Corollary \ref{cor:main correspondence of E_g and O_g,S}).

\subsubsection*{Outline for the rest of the paper}

The current section may be viewed as a prelude to Section \ref{sec:Statistics-of-the-equivelence-classes}
in which we use the aforementioned connection (Corollary \ref{cor:main correspondence of E_g and O_g,S})
and Theorem \ref{thm:AESgrids thm} to prove Theorem \ref{thm:refinement of the main theorem-1},
which gives the limiting distribution of the normalized counting measures
on the subsets $\left\{ (\pi_{\V_{Q(\ovec(g))}}(x),\vartheta_{q}(x))\mid x\in E_{g}\right\} \subseteq\V_{Q(\mathbf{e}_{d})}(\R)\times\V_{a}(\Z/(q))$,
as $Q(\ovec(g))\to\infty$.

In Section \ref{sec:Proof-of-theorems} we achieve our main goal of
proving Theorems \ref{thm:main thm for Z} - \ref{thm:main_thm_with_congruences-forZ}
concerning the limit of the normalized counting measures supported
on $\left\{ (\pi_{\V_{T}}(x),\vartheta_{q}(x))\mid x\in\V_{T}(\Z)\right\} $,
$T\in\N$, by rewriting the counting measures on $\left\{ (\pi_{\V_{T}}(x),\vartheta_{q}(x))\mid x\in\V_{T}(\Z)\right\} $
as an average of the counting measures on $\left\{ (\pi_{\V_{T}}(x),\vartheta_{q}(x))\mid x\in E_{g}\right\} $
and by employing Theorem 8.1.

\subsection{\label{subsec:The-equivelence-relation}The equivalence relation}

A natural way to ``generate'' integral points on $\V_{T}(\Z)$ from
a given $g\in\V_{T}(\Z)$ is to view $g$ as a point in $\V_{T}(\Q_{S})$
and to consider the intersection of orbits
\begin{equation}
E_{g}\df g\cdot\G(\ZS)\bigcap g\cdot\G(\Z_{S})\label{eq:definition of E_g}
\end{equation}

(to recall the definition of the right action of $\G$ on $\V_{T}$
see \eqref{eq:definion of action of =00005Cmatbb=00007BG=00007D on Z_t}).
We define our equivalence relation on $\V_{T}(\Z)$ by $g\sim g'\iff E_{g}=E_{g'}$.
Clearly, the equivalence class of each $g\in\V_{T}(\Z)$ is given
by $E_{g}$.
\begin{lem}
\label{lem:finitely many ASL_d-1 orbits.}For each $T>0$, it holds
that each equivalence class $E_{g}$ is composed of finitely many
\emph{$\G(\Z)$} orbits, and it holds that there are finitely many
equivalence classes.
\end{lem}

\begin{proof}
Note that each equivalence class is $\G(\Z)$ invariant, hence each
equivalence class is composed of $\G(\Z)$ orbits. There are finitely
many $\G(\Z)$ orbits in $\V_{T}(\Z)$ by Corollary \ref{cor:transitivity of ASL_times_SO_d},
\eqref{enu:finitely many orbits in V_x(Z)}, which proves our claim.
\end{proof}

\subsection{\label{subsec:Decomposition-of-the}A decomposition of the orbits
$O_{g,S}$}

For the rest of this section we fix a finite set of primes $S$, and
we take $g\in\SL_{d}(\Z)$ such that $Q(\ovec(g))>0$.

The goal of this section is to deduce the decomposition \eqref{eq:decomposition of O_v,p cap U},
which is a technical fact that we will need in Section \ref{subsec:duality principle}
to relate the orbit $O_{g,S}$ with $E_{g}$.

We recall the definition of $O_{g,S}$ and we rewrite it as follows
\begin{align}
O_{g,S}= & (t_{g},e_{S})\mathbf{L}_{g}(\R\times\Q_{S})\G(\ZS)\label{eq:rewriting O_g,S}\\
= & \left(t_{g}\mathbf{L}_{g}(\R)t_{g}^{-1}\times\mathbf{L}_{g}(\Q_{S})\right)(t_{g},e_{S})\G(\ZS),\nonumber 
\end{align}
where $t_{g}$ is defined in \eqref{eq:def of theta_g}. By Lemma
\ref{cor:Special alligned base points} we deduce that $t_{g}\mathbf{L}_{g}(\R)t_{g}^{-1}=H$
(where $H=\mathbf{L}_{I_{d}}(\R)$) and by \eqref{eq:rewriting O_g,S}
we deduce that
\begin{equation}
O_{g,S}=H\times\mathbf{L}_{g}(\Q_{S})(t_{g},e_{S})\G(\ZS).\label{eq:O_g,S=00003DH=00005CtimesL_g(Q_S)}
\end{equation}

We have that $\Lbold_{g}$ is a $\Q$-group, hence we obtain
\begin{equation}
\Lbold_{g}(\Q_{S})=\bigsqcup_{h\in M}\Lbold_{g}(\Z_{S})h\Lbold_{g}(\ZS),\label{eq:decomposition of SO_V_perp}
\end{equation}
where $M=M(g)$ is a finite set of representatives of the double coset
space (see \cite[Chapter 5]{platonov_rapinchuk}). Using \eqref{eq:O_g,S=00003DH=00005CtimesL_g(Q_S)}
and \eqref{eq:decomposition of SO_V_perp} we obtain the decomposition
\begin{equation}
O_{g,S}=\bigsqcup_{h\in M}O_{g,S,h},\label{eq:decomposition _v,p}
\end{equation}
where
\begin{equation}
O_{g,S,h}\overset{\text{def}}{=}(H\times\Lbold_{g}(\Z_{S}))\left(t_{g},h\right)\G\left(\ZS\right).\label{eq:O_g,S,h}
\end{equation}

\subsubsection{Intersection with the principle genus}

We will be actually interested in the intersection $O_{g,S}\cap\U_{S}$,
where $\U_{S}\subseteq\G(\R\times\Q_{S})/\G(\ZS)$ is the clopen orbit
of the clopen subgroup $\G(\R\times\Z_{S})$ passing through the identity
coset $\G(\ZS)$, namely
\begin{equation}
\U_{S}\df\G(\R\times\Z_{S})\G(\ZS).\label{eq:def of U_S}
\end{equation}
Since $\G(\Z)\leq\G(\R\times\Z_{S})$, where $\G(\Z)\subseteq\G(\R\times\Z_{S})$
is diagonally embedded, is the stabilizer subgroup stabilizing $\G(\ZS)$
by the natural left action, we conclude that $\U_{S}$ is naturally
identified with $\G(\R\times\Z_{S})/\G(\Z)$, where each element $\left(g_{\infty},g_{S}\right)\G(\ZS)\in\U_{S}$
viewed as a point in $\G(\R\times\Q_{S})/\G(\ZS)$ identifies with
$\left(g_{\infty}\gamma^{-1},g_{S}\gamma^{-1}\right)\G(\Z)\in\G(\R\times\Z_{S})/\G(\Z)$,
where $\gamma\in\G(\ZS)$ is an arbitrary element which gives that
$g_{S}\gamma^{-1}\in\G(\Z_{S}).$

We observe that $O_{g,S,h}\cap\U_{S}\neq\emptyset$ if and only if
$O_{g,S,h}\subseteq\U_{S}$, which shows that 
\begin{equation}
O_{g,S}\cap\U_{S}=\bigsqcup_{h\in M_{0}}O_{g,S,h},\label{eq:decomposition of O_v,p cap U}
\end{equation}
where $M_{0}=M_{0}(g)\subseteq M(g)$ is a finite subset.

For all $h\in M_{0}$, since $O_{g,S,h}\subseteq\U_{S}$, we obtain
that $h\in\mathbf{L}_{g}(\Q_{S})\cap\G(\R\times\Z_{S})\G(\ZS)$. Namely
there are $c\in\G(\Z_{S})$ and $\gamma\in\G(\ZS)$ such that 
\begin{equation}
h=c\gamma^{-1}.\label{eq:decomposition of h}
\end{equation}
Then, for $h\in M_{0}$, we get that the orbit $O_{g,S,h}$ (defined
in \eqref{eq:O_g,q,h}) is identified by
\begin{equation}
O_{g,S,h}=(H\times\Lbold_{g}(\Z_{S}))(t_{g}\gamma,c)\G(\Z).\label{eq:identification of O_g,S,h in G(R=00005Ctimes=00005CZ_S)/G(Z)}
\end{equation}

\subsection{\label{subsec:duality principle}A duality principle relating $E_{g}$
with $O_{g,S}\cap\protect\U_{S}$}

The main idea that stands behind the relation of $E_{g}$ with $O_{g,S}$
can be roughly described as a ``duality'' principle by which we
transfer a left $(H\times\Lbold_{g}(\Z_{S}))$-orbit \eqref{eq:identification of O_g,S,h in G(R=00005Ctimes=00005CZ_S)/G(Z)}
in $\G(\R\times\Z_{S})/\G(\Z)$ to a right $\G(\Z)$-orbit in $(H\times\Lbold_{g}(\Z_{S}))\backslash\G(\R\times\Z_{S})$
via the following diagram of natural maps
\begin{equation}
\xymatrix{ & \G(\R\times\Z_{S})\ar[dr]\ar[dl]\\
(H\times\Lbold_{g}(\Z_{S}))\backslash\G(\R\times\Z_{S}) &  & \G(\R\times\Z_{S})/\G(\Z)
}
\label{eq:duality diagram}
\end{equation}
Namely, by \eqref{eq:duality diagram}, we transfer an orbit \eqref{eq:identification of O_g,S,h in G(R=00005Ctimes=00005CZ_S)/G(Z)}
to a $\text{\ensuremath{\G}(\ensuremath{\Z})}$-orbit $\mathcal{Q}_{g,S,h}\subseteq(H\times\Lbold_{g}(\Z_{S}))\backslash\G(\R\times\Z_{S})$
passing through the base point $(H\times\Lbold_{g}(\Z_{S}))(t_{g}\gamma,c)$.
By using the right action of $\G(\R\times\Z_{S})$ on $\V_{Q(\mathbf{e}_{d})}(\R)\times\V_{Q(\ovec(g))}(\Z_{S})$,
and by recalling that $H\times\Lbold_{g}(\Z_{S})$ is the stabilizer
of $(I_{d},g)$, we may identify $\mathcal{Q}_{g,S,h}$ with (where
we abuse notations)
\begin{equation}
\mathcal{Q}_{g,S,h}=(I_{d}\cdot(t_{g}\gamma),g\cdot c)\cdot\G(\Z)\subseteq\V_{Q(\mathbf{e}_{d})}(\R)\times\V_{Q(\ovec(g))}(\Z_{S}).\label{eq:O_g,S,h in the variety}
\end{equation}

To relax notations, we denote the homeomorphism $\pi_{\V_{T}}:\V_{T}(\R)\to\V_{Q(\mathbf{e}_{d})}(\R)$
(defined in \eqref{eq:def of pi_Z}) by $\pi_{\V}$. The lemma below
gives the key correspondence between $O_{g,S}\cap\U_{S}$ and $E_{g}$.
\begin{lem}
\label{lem:the union of O_g,S,h in the variety}It holds that $\bigsqcup_{h\in M_{0}}\mathcal{Q}_{g,S,h}=\{(\pi_{\V}(x),x)\mid x\in E_{g}\}$,
and that $|M_{0}|=\left|E_{g}/\G(\Z)\right|$.
\end{lem}

\begin{proof}
Let us first show that for each $h\in M_{0}$ it holds that $\mathcal{Q}_{g,S,h}\subseteq\{(\pi_{\V}(x),x)\mid x\in E_{g}\}.$
By writing $h=c\gamma^{-1}$ for $c\in\G(\Z_{S})$ and $\gamma\in\G(\ZS)$
and by noting that $g=g\cdot h=g\cdot(c\gamma^{-1})$, we obtain that
$g\cdot\gamma=g\cdot c$. Then, $g'$ defined by
\[
g'\df g\cdot\gamma=g\cdot c.
\]
is in $E_{g}$. By definition of $t_{g}$ (see \eqref{eq:def of theta_g})
we have
\[
g'=g\cdot\gamma=a_{Q(\ovec(g))}\cdot(t_{g}\gamma),
\]
and by using the equivariance of the map $\pi_{\V}$, we get
\[
\pi_{\V}(g')=\pi_{\V}(g\cdot\gamma)=\pi_{\V}(a_{Q(\ovec(g))}\cdot(t_{g}\gamma))\underbrace{=}_{\text{recalling }\eqref{eq:def of pi_Z}}I_{d}\cdot(t_{g}\gamma).
\]
We may now conclude that 
\[
(\pi_{\V}(g'),g')\cdot\G(\Z)=\mathcal{Q}_{g,S,h},
\]
and by using equivariance of $\pi_{\V}$, we deduce that
\[
(\pi_{\V}(g'),g')\cdot\G(\Z)=\left\{ (\pi_{\V}(g'\cdot\gamma),g'\cdot\gamma)\mid\gamma\in\G(\Z)\right\} \subseteq\{(\pi_{\V}(x),x)\mid x\in E_{g}\}.
\]

We will now prove the inclusion in the opposite direction. We let
$g'\in E_{g}$ and we note that, according to the definition of $E_{g}$
(see \eqref{eq:definition of E_g}), there are $c\in\G(\Z_{S})$ and
$\gamma\in\G(\ZS)$ such that
\[
g'=g\cdot\gamma=g\cdot c.
\]
We can deduce from the preceding equality that $h\df c\gamma^{-1}$
is an element of $\mathbf{L}_{g}(\Q_{S})\cap\G(\R\times\Z_{S})\G(\ZS)$,
and we conclude by the preceding paragraph that $(\pi_{\V}(g'),g')\in\mathcal{Q}_{g,S,h}$.

Finally, since $O_{g,S,h}$, $h\in M_{0}$, are disjoint $(H\times\Lbold_{g}(\Z_{S}))$-orbits
in $\G(\R\times\Z_{S})/\G(\Z)$, it follows that $\mathcal{Q}_{g,S,h}$,
$h\in M_{0}$, are disjoint $\G(\Z)$-orbits in $(H\times\Lbold_{g}(\Z_{S}))\backslash\G(\R\times\Z_{S})$.
As $\{(\pi_{\V}(x),x)\mid x\in E_{g}\}/\G(\Z)$ is in bijection with
$E_{g}/\G(\Z)$, it follows that $M_{0}=\left|E_{g}/\G(\Z)\right|$.
\end{proof}
We are actually interested in the set $\{(\pi_{\V}(x),\vartheta_{q}(x))\mid x\in E_{g}\}$,
and in order to relate it to the orbits $O_{g,S,h}$ we will consider
the projection modulo $q$ in the following subsection.

\subsubsection{Taking the residue modulo $q$}

We note that the natural ring homomorphism 
\[
\vartheta_{p^{k}}:\Z_{p}\to\Z_{p}/p^{k}\Z_{p}\cong\Z/(p^{k}),
\]
induces a homomorphism $\red_{p^{k}}:\G(\Z_{p})\to\G(\Z/(p^{k}))$.
Let $q\in\N$ and assume that $S$ includes the primes $S_{q}$ appearing
in the prime decomposition of $q$. The Chinese remainder theorem
yields the identification
\[
\prod_{p_{i}\in S',}\G(\Z/p_{i}^{k_{i}}\Z)\cong\G(\Z/(q)),
\]
and so we obtain the map $\red_{q}:\G(\Z_{S})\to\G(\Z/(q))$ in the
obvious way. We also note that $\vartheta_{q}(\mathbf{L}_{g}(\Z_{S}))\subseteq\mathbf{L}_{\vartheta_{q}(g)}(\Z/(q)).$

We consider the map $\left(id_{\infty}\times\red_{q}\right):\G(\R\times\Z_{S})\to\G(\R\times\Z/(q))$
given by 
\[
\left(id_{\infty}\times\red_{q}\right)(g_{\infty},g_{S})\df(g_{\infty},\vartheta_{q}(g_{S})),
\]
 and we upgrade Diagram \eqref{eq:duality diagram} to the following
diagram
\[
\xymatrix{ & \G(\R\times\Z_{S})\ar[dr]\ar[dl]\\
(H\times\Lbold_{g}(\Z_{S}))\backslash\G(\R\times\Z_{S})\ar[d]_{\left(id_{\infty}\times\red_{q}\right)} &  & \G(\R\times\Z_{S})/\G(\Z)\ar[d]_{\left(id_{\infty}\times\red_{q}\right)}\\
(H\times\Lbold_{\vartheta_{q}(g)}(\Z/(q)))\backslash\G(\R\times\Z/(q)) &  & \G(\R\times\Z/(q))/\G_{(q)}(\Z)
}
\]
where
\[
\G_{\left(q\right)}(\Z)\df\left\{ \left(u,\red_{q}\left(u\right)\right)\mid u\in\G(\Z)\right\} \leq\G\left(\R\times\Z/(q)\right).
\]
We let 
\begin{equation}
\begin{aligned}O_{g,q,h}\df & \left(id_{\infty}\times\red_{q}\right)\circ O_{g,S,h}\\
= & (H\times\vartheta_{q}(\Lbold_{g}(\Z_{S})))(t_{g}\gamma,\vartheta_{q}(c))\G_{(q)}(\Z)
\end{aligned}
\label{eq:O_g,q,h}
\end{equation}
and we let
\begin{equation}
\begin{aligned}\mathcal{Q}_{g,q,h} & \df\left(id_{\infty}\times\red_{q}\right)(\mathcal{Q}_{g,S,h})\end{aligned}
\label{eq:O_g,q,h in the variety}
\end{equation}
where $\mathcal{Q}_{g,q,h}$ is the right $\G_{(q)}(\Z)$-orbit passing
through $(H\times\Lbold_{\vartheta_{q}(g)}(\Z/(q)))(t_{g}\gamma,\vartheta_{q}(c))$.
\begin{lem}
\label{lem:the projection of orbit to real place}Let $h,h'\in M_{0}$
be two different elements and let $\gamma,\gamma'\in\G(\ZS)$ which
appear in a decomposition \eqref{eq:decomposition of h} of $h,h'$
correspondingly. Then $Ht_{g}\gamma\G(\Z)\cap Ht_{g}\gamma'\G(\Z)=\emptyset$.
\end{lem}

\begin{proof}
Assume for contradiction that $Ht_{g}\gamma\G(\Z)\cap Ht_{g}\gamma'\G(\Z)\neq\emptyset$.
Then there exists $\kappa\in H$ and $u\in\G(\Z)$ such that 
\begin{equation}
t_{g}^{-1}\kappa t_{g}\gamma u=\gamma'.\label{eq:=00005Ctheta_g ^-1=00005Ckappa=00005Ctheta_g is in G(=00005CZS)}
\end{equation}
This gives that 
\begin{equation}
\left(t_{g}^{-1}\kappa t_{g}\right)h^{-1}\left(cuc'^{-1}\right)=h'^{-1},\label{eq:equivalence of h and h'}
\end{equation}
where $c,c'\in\G(\Z_{S})$ appear in the decomposition \eqref{eq:decomposition of h}
of $h,h'$ correspondingly. By the definition of $t_{g}$ and by \eqref{eq:=00005Ctheta_g ^-1=00005Ckappa=00005Ctheta_g is in G(=00005CZS)}
we conclude that $t_{g}^{-1}\kappa t_{g}\in\Lbold_{g}(\R)\cap\G(\ZS)=\mathbf{L}_{g}(\ZS)$,
and by \eqref{eq:equivalence of h and h'} we get $\left(cuc'^{-1}\right)\in\Lbold_{g}(\Q_{S})\cap\G(\Z_{S})=\mathbf{L}_{g}(\Z_{S})$.
Hence \eqref{eq:equivalence of h and h'} shows that $h$ and $h'$
are equivalent, which is a contradiction since $h,h'$ are representatives
for two different cosets in the space $\mathbf{L}_{g}(\Z_{S})\backslash\mathbf{L}_{g}(\Q_{S})/\mathbf{L}_{g}(\ZS)$.
\end{proof}
From Lemma \ref{lem:the union of O_g,S,h in the variety}, we obtain
the following corollary, which is the main conclusion of our discussion
in this section.
\begin{cor}
\label{cor:main correspondence of E_g and O_g,S} It holds that $\bigsqcup_{h\in M_{0}}\mathcal{Q}_{g,q,h}=\{(\pi_{\V}(x),\vartheta_{q}(x))\mid x\in E_{g}\}$,
and that $M_{0}=\left|E_{g}/\G(\Z)\right|$.
\end{cor}

\begin{proof}
By Lemma \ref{lem:the projection of orbit to real place} it follows
that $\bigsqcup_{h\in M_{0}}\mathcal{Q}_{g,q,h}$ is indeed a disjoint
union, and by using Lemma \ref{lem:the union of O_g,S,h in the variety}
we obtain that
\[
\begin{aligned}\bigsqcup_{h\in M_{0}}\mathcal{Q}_{g,q,h}= & \left(id_{\infty}\times\red_{q}\right)\left(\bigsqcup_{h\in M_{0}}\mathcal{Q}_{g,S,h}\right)\\
= & \left(id_{\infty}\times\red_{q}\right)\left(\{(\pi_{\V}(x),x)\mid x\in E_{g}\}\right)\\
= & \{(\pi_{\V}(x),\vartheta_{q}(x))\mid x\in E_{g}\}.
\end{aligned}
\]
\end{proof}

\section{\label{sec:Statistics-of-the-equivelence-classes}Statistics of the
equivalence classes $E_{g}$}

We are now ready to study the statistics of $E_{g}$ as $Q(\ovec(g))\to\infty$
by using the limiting distribution of the orbits $O_{g,S}$ (Theorem
\ref{thm:AESgrids thm}), and by exploiting the connection between
the equivalence classes $E_{g}$ and the orbits $O_{g,S}$ (Corollary
\ref{cor:main correspondence of E_g and O_g,S}).

We now list the assumptions that will hold throughout this section,
which will allow us to employ Theorem \ref{thm:AESgrids thm}.
\begin{itemize}
\item $Q$ is a form as in our \nameref{subsec:Standing-Assumption} and
$q\in2\N+1$ is such that $Q$ is non-singular modulo $q$.
\item $\left\{ g_{n}\right\} _{n=1}^{\infty}\subseteq\SL_{d}(\Z)$ satisfy
that $Q(\ovec(g_{n}))\to\infty$ and for all $n\in\N$
\begin{itemize}
\item $Q(\ovec(g_{n}))>0$
\item there is a prime $p_{0}$ for which $\ovec(g_{n})$ is $(Q,p_{0})$
co-isotropic (see Definition \ref{def:Isotropicity definition}),
\item The reduction mod $q$ is fixed in $n$, namely $\red_{q}(g_{n})=\bar{g}$,
for all $n\in\N$.
\end{itemize}
\item $S_{q}$ denotes the set of primes decomposing $q$ and $S\df S_{q}\cup\{p_{0}\}$.
\end{itemize}
By Lemma \ref{lem:properties of a form non-singular modulo q}\eqref{enu:-is-isotropic for all p=00005Cin S_q},
we deduce that the assumptions of Theorem \ref{thm:AESgrids thm}
indeed hold for $S$ and the sequence $\left\{ g_{n}\right\} _{n=1}^{\infty}$.

We denote $a\df Q(\ovec(\bar{g}))\in\Z/(q)$ and consider the measures
on $\V_{Q(\mathbf{e}_{d})}(\R)\times\V_{a}(\Z/(q))$ given by 
\[
\nu_{g_{n}}^{q}\df\frac{1}{\left|E_{g_{n}}/\G(\Z)\right|}\sum_{x\in E_{g_{n}}}\delta_{\left(\pi_{\V}(x),\vartheta_{q}(x)\right)},\ n\in\N.
\]

Our main goal in this section is to prove the following theorem.
\begin{thm}
\label{thm:refinement of the main theorem-1} Consider $O_{\bar{g}}\subseteq\V_{a}(\Z/(q))$
defined by 
\[
O_{\bar{g}}\df\bar{g}\cdot\G(\Z/(q)),
\]
 and let $\mu_{O_{\bar{g}}}$ be the normalized counting measure on
$O_{\bar{g}}$. Then for all $f\in C_{c}(\V_{Q(\mathbf{e}_{d})}(\R)\times\V_{a}(\Z/(q)))$
it holds that 
\[
\lim_{n\to\infty}\nu_{g_{n}}^{q}(f)=\mu_{\V}\otimes\mu_{O_{\bar{g}}}(f).
\]
\end{thm}

\subsection{Outline of proof for Theorem \ref{thm:refinement of the main theorem-1}}

We now outline the method we will use in the proof of Theorem \ref{thm:refinement of the main theorem-1},
building on Theorem \ref{thm:AESgrids thm} and the link between the
equivalence classes $E_{g_{n}}$ and the orbits $O_{g_{n},S}$.

We denote
\begin{equation}
G\df\G(\R\times\Z/(q)),\ K\df(H\times\Lbold_{\bar{g}}(\Z/(q))),\ \Gamma=\G_{(q)}(\Z),\label{eq:notations of K G =00005CGamma}
\end{equation}
and we consider the following diagram of natural maps 
\[
\xymatrix{K\backslash G\ar[dr]^{\pi_{\Gamma}} &  & G/\Gamma\ar[dl]_{\pi_{K}}\\
 & K\backslash G/\Gamma
}
\]
where $\pi_{K}$ and $\pi_{\Gamma}$ denote the natural quotient map.

We recall that $\bigsqcup_{h\in M_{0}}O_{g,q,h}$ is a disjoint union
of finitely many $(H\times\vartheta_{q}(\Lbold_{g_{n}}(\Z_{S})))$-orbits
and we recall that $(H\times\vartheta_{q}(\Lbold_{g_{n}}(\Z_{S})))\subseteq K$.
Hence $\mathcal{R}_{g_{n},q}\subseteq K\backslash G/\Gamma$ defined
by

\begin{equation}
\mathcal{R}_{g_{n},q}\df\pi_{K}(\bigsqcup_{h\in M_{0}}O_{g,q,h}),\label{eq:def_of_R_g,q}
\end{equation}
is a finite set, and by Lemma \ref{lem:the projection of orbit to real place}
we obtain that $\left|\mathcal{R}_{g_{n},q}\right|=|M_{0}|=\left|E_{g_{n}}/\G(\Z)\right|$.
In Section \ref{subsec:Equidistribution-in K=00005CG/=00005CGaam},
we will prove, by relying on Theorem \ref{thm:AESgrids thm}, that
the uniform probability counting measures $\lambda_{g_{n},q}$ on
$\mathcal{R}_{g_{n},q}$ equidistribute towards the natural probability
measure $\mu_{K\backslash G/\Gamma}$ on $K\backslash G/\Gamma$.

To deduce the limit of our counting measures that are supported on
$\{(\pi_{\V}(x),\vartheta_{q}(x))\mid x\in E_{g_{n}}\}$ by the equidistribution
of $\{\lambda_{g_{n},q}\}_{n=1}^{\infty}$ we observe that $\mathcal{R}_{g_{n},q}$
can be also described by
\begin{equation}
\mathcal{R}_{g_{n},q}\df\pi_{\Gamma}(\bigsqcup_{h\in M_{0}}\mathcal{Q}_{g,q,h})=\pi_{\Gamma}(\{(\pi_{\V}(x),\vartheta_{q}(x))\mid x\in E_{g}\}),\label{eq:R_g,q as the set of =00005CGammaorbits}
\end{equation}
and we use ``unfolding'' technique (similarly to Section \ref{subsec:Unwinding})
to lift the measure $\lambda_{g_{n},q}$ for $n\in\N$ to the counting
measure on $K\backslash G$ supported on $\{(\pi_{\V}(x),\vartheta_{q}(x))\mid x\in E_{g}\}$.

\subsubsection{Unfolding}

We now discuss the ``unfolding'' process mentioned above which lifts
an equidistribution result in $K\backslash G/\Gamma$ to an equidistribution
result in $K\backslash G$.

Let $m_{G}$, $m_{G/\Gamma}$, be $G$-invariant measures on $G,\ G/\Gamma$
respectively, such that $m_{G/\Gamma}$ is a probability measure and
all the measures are Weil normalized (a notion introduced in Section
\ref{subsec:Measures-as measures on fibre bundles}), namely such
that for all $\ef\in C_{c}(G)$
\begin{equation}
\int_{G}\ef(g)dm_{G}(g)=\int_{G/\Gamma}\left(\sum_{\gamma\in\Gamma}\ef(g\gamma)\right)dm_{G/\Gamma}(g\Gamma).\label{eq:unwinding of a measure on G-1}
\end{equation}

We define a measure on $K\backslash G$ by $\mu_{K\backslash G}\df\left(\pi_{K}\right)_{*}m_{G},$
and a measure on $K\backslash G/\Gamma$ by $\mu_{K\backslash G/\Gamma}\df\left(\pi_{K}\right)_{*}m_{G/\Gamma}$
(which is well defined, since we assume that $K$ is compact).

Assume that $S_{n}\subseteq K\backslash G/\Gamma$ is a finite set,
and consider the measures $\bar{\nu}_{n}$ supported on $K\backslash G$
defined by 
\[
\bar{\nu}_{n}\df\frac{1}{|S_{n}|}\sum_{x\in\left(\pi_{\Gamma}\right)^{-1}(S_{n})}\delta_{x}.
\]
Let
\begin{equation}
\begin{aligned}\mathcal{F}\df & \left\{ Kg\Gamma\mid\left|\text{Stab}_{\Gamma}(Kg)\right|>1\right\} \\
= & \left\{ Kg\Gamma\mid\left|g^{-1}Kg\cap\Gamma\right|>1\right\} .
\end{aligned}
\label{eq:definintio of set of points that are non-triv stab by Gam-1}
\end{equation}

\begin{lem}
\label{lem:unwinding with negliglbe non-triv stab-1} Assume that
$S_{n}\subseteq K\backslash G/\Gamma$, $n\in\N$, are finite sets
such that the probability counting measures supported on $S_{n}$
converge weakly to $\mu_{K\backslash G/\Gamma}$, and that
\begin{equation}
\frac{|\mathcal{F}\cap S_{n}|}{\left|S_{n}\right|}\to0.\label{eq:negligability of non-trivialy stabilized points-1}
\end{equation}
Then for every $f\in C_{c}(K\backslash G)$, it holds that $\bar{\nu}_{n}(f)\to\mu_{K\backslash G}(f)$.
\end{lem}

The proof of Lemma \ref{lem:unwinding with negliglbe non-triv stab-1}
involves elementary tools, hence we decided to include the complete
details in the appendix.

Our goal in the following section is to verify the assumptions of
Lemma \ref{lem:unwinding with negliglbe non-triv stab-1} for $S_{n}=\mathcal{R}_{g_{n},q}$,
which will prove Theorem \ref{thm:refinement of the main theorem-1}.

\subsection{\label{subsec:Equidistribution-in K=00005CG/=00005CGaam}Equidistribution
in $K\backslash G/\Gamma$}

Let $\tilde{\eta}_{g_{n},S}$ be the measure supported on $O_{g_{n},S}\cap\U_{S}$,
given by
\begin{equation}
\tilde{\eta}_{g_{n},S}\df\mu_{g_{n},S}\mid_{\U_{S}},\label{eq:def_of_eta}
\end{equation}
where $\U_{S}=\G(\R\times\Z_{S})\G(\ZS)\cong\G(\R\times\Z_{S})/\G(\Z$)
and $\mu_{g_{n},S}$ defined in \eqref{eq:=00005Cmu_g,S} is the natural
probability measure supported on $O_{g_{n},S}$. We consider the following
probability measure $\eta_{g_{n},q}$ on $K\backslash G/\Gamma$ supported
on $\mathcal{R}_{g_{n},q}$ (by Corollary \ref{cor:main correspondence of E_g and O_g,S})
defined by
\[
\eta_{g_{n},q}\df(\pi_{K}\circ(id_{\infty}\times\vartheta_{q}))_{*}\tilde{\eta}_{g_{n},S}.
\]

We obtain the following corollary which follows by Theorem \ref{thm:AESgrids thm}.
\begin{cor}
\label{cor:s-arithemetic measures converge to haar-1}It holds that
\begin{equation}
\eta_{g_{n},q}\to\mu_{K\backslash G/\Gamma},\label{eq:equidisitrbution of omega modulo}
\end{equation}
where $\mu_{K\backslash G/\Gamma}$ is the push-forward by the natural
quotient map $\pi_{K}$ of the unique $G$-invariant probability measure
on $G/\Gamma$.
\end{cor}

\begin{proof}
Since $\U_{S}\subseteq\G(\R\times\Q_{S})/\G(\ZS)$ is a clopen set,
we get by Theorem \ref{thm:AESgrids thm} that
\begin{equation}
\tilde{\eta}_{g_{n},S}\overset{\text{weak *}}{\longrightarrow}\mu_{\U_{S}},\label{eq:weak convergence to haar of eta-1}
\end{equation}
where $\mu_{\U_{S}}$ is the unique $\G(\R\times\Z_{S})$ invariant
probability on 
\[
\U_{S}\cong\G(\R\times\Z_{S})/\G(\Z).
\]
By Lemma \ref{lem:properties of a form non-singular modulo q},\eqref{enu:The-reduction-map in onto modulo q},
by the Chinese remainder theorem, and by noting that $\vartheta_{q}\left(\ASL_{d-1}(\Z_{S})\right)=\ASL_{d-1}(\Z/(q))$,
we conclude that 
\[
\vartheta_{q}\left(\G(\Z_{S})\right)=\G(\Z/(q)).
\]
Hence $(id_{\infty}\times\vartheta_{q}):\G(\R\times\Z_{S})/\G(\Z)\to G/\Gamma$
is onto. It now follows that
\[
\begin{aligned}\eta_{g_{n},q}=(\pi_{K}\circ(id_{\infty}\times\vartheta_{q}))_{*}\tilde{\eta}_{g_{n},S}\to & (\pi_{K})_{*}(id_{\infty}\times\vartheta_{q}){}_{*}\mu_{\U_{S}}\\
= & (\pi_{K})_{*}\mu_{G/\Gamma}\\
= & \mu_{K\backslash G/\Gamma}.
\end{aligned}
\]
\end{proof}

\subsubsection{Weights of the measures $\eta_{g_{n},q}$}

In the following we study the weights of the atoms of the measures
$\eta_{g_{n},q}$ which are supported on the finite sets $\mathcal{R}_{g_{n},q}$.

We express $\eta_{g_{n},q}$ by

\begin{equation}
\eta_{g_{n},q}=\sum_{h\in M_{0}}\alpha_{h}^{(n)}\delta_{\pi_{K}\left(O_{g_{n},q,h}\right)},\label{eq:rh_Dk_rh_infty_eta-1}
\end{equation}
and by recalling \eqref{eq:def_of_eta} and the decomposition \eqref{eq:decomposition of O_v,p cap U}
of $O_{g_{n},S}\cap\U_{S}$, we conclude that
\[
\alpha_{h}^{(n)}=\tilde{\eta}_{g_{n},S}\left((\pi_{K}\circ(id_{\infty}\times\vartheta_{q}))^{-1}\left(O_{g_{n},q,h}\right)\right)\underbrace{=}_{\text{Lemma }\ref{lem:the projection of orbit to real place}}\tilde{\eta}_{g_{n},S}(O_{g_{n},S,h}).
\]
It follows that 
\begin{equation}
\alpha_{h}^{(n)}=\tilde{\eta}_{g_{n},S}(O_{g_{n},S,h})=\frac{\alpha^{(n)}}{\left|\text{stab}_{H\times\mathbf{L}_{g_{n}}(\Z_{S})}\left((t_{g_{n}}\gamma,c)\G(\Z)\right)\right|},\label{eq:measure of eta_g,S(O_g,S,h)}
\end{equation}
where $c\in\G(\Z_{S}),$ $\gamma\in\G(\ZS)$ decompose $h$ as in
\eqref{eq:decomposition of h}, where $\text{stab}_{H\times\mathbf{L}_{g_{n}}(\Z_{S})}\left(x\right)$
for $x\in\G(\R\times\Z_{S})/\G(\Z)$ is the stabilizer of $x$ under
the natural left action of $H\times\mathbf{L}_{g_{n}}(\Z_{S})$, and
where $\alpha^{(n)}\in\R_{>0}$ is a normalizing factor which turns
$\eta_{g_{n},q}$ to a probability measure.
\begin{lem}
\label{lem:eqaulity of stabilizer of H=00005CtimesL_g and interal ortohoganl mat stab a vec}
Let \emph{$g\in\left\{ g_{n}\right\} _{n=1}^{\infty}$}, and let $h\in M_{0}$
be such that $h=c\gamma{}^{-1}$, for $\gamma\in\G\left(\ZS\right)$
and $c\in\G(\Z_{S})$. Then 
\[
\left|\text{stab}_{H\times\mathbf{L}_{g_{n}}(\Z_{S})}\left((t_{g}\gamma,c)\G(\Z)\right)\right|\leq\left|\mathbf{H}_{\ovec\left(I_{d}\cdot(t_{g}\gamma)\right)}(\R)\cap\G_{1}(\Z)\right|.
\]
\end{lem}

\begin{proof}
We have that 
\[
\text{stab}_{H\times\mathbf{L}_{g}(\Z_{S})}\left((t_{g}\gamma,c)\G(\Z)\right)=\left(H\times\mathbf{L}_{g}(\Z_{S})\right)\cap x_{g,h}\G(\Z)x_{g,h}^{-1}.
\]
where $x_{g,h}=(t_{g}\gamma,c)$. We recall that $H\times\mathbf{L}_{g}(\Z_{S})$
is a graph of a function $f:\mathbf{H}_{\mathbf{e}_{d}}(\R)\times\mathbf{H}_{\ovec(g)}(\Z_{S})\to\G(\R)\times\G(\ZS)$,
(see Lemma \ref{lem:stabilizer lemma}), which gives
\[
\left|\left(H\times\mathbf{L}_{g}(\Z_{S})\right)\cap x_{g,h}\G(\Z)x_{g,h}^{-1}\right|\leq\left|\left(\mathbf{H}_{\mathbf{e}_{d}}(\R)\times\mathbf{H}_{\ovec(g)}(\Z_{S})\right)\cap\pi_{1}(x_{g,h})\G_{1}(\Z)\pi_{1}(x_{g,h})^{-1}\right|,
\]
where $\pi_{1}:\G\to\G_{1}$ is the natural projection, and $\pi_{1}(x_{g,h})=\left(\pi_{1}(t_{g}\gamma),\pi_{1}(c)\right)$.
We observe that

\[
\begin{aligned}\left|\left(\mathbf{H}_{\mathbf{e}_{d}}(\R)\times\mathbf{H}_{\ovec(g)}(\Z_{S})\right)\bigcap\pi_{1}(x_{g,h})\G_{1}(\Z)\pi_{1}(x_{g,h})^{-1}\right|\\
=\left|\pi_{1}(x_{g,h})^{-1}\left(\mathbf{H}_{\mathbf{e}_{d}}(\R)\times\mathbf{H}_{\ovec(g)}(\Z_{S})\right)\right. & \left.\pi_{1}(x_{g,h})\bigcap\G_{1}(\Z)\right|\\
\underbrace{\leq}_{\G_{1}(\Z)\text{ is diagonally embedded}} & \left|\pi_{1}(t_{g}\gamma)^{-1}\mathbf{H}_{\mathbf{e}_{d}}(\R)\pi_{1}(t_{g}\gamma)\bigcap\G_{1}(\Z)\right|.
\end{aligned}
\]
We conclude that
\[
\left|\text{stab}_{H\times\mathbf{L}_{g}(\Z_{S})}\left((t_{g}\gamma,c)\G(\Z)\right)\right|\leq\left|\pi_{1}(t_{g}\gamma)^{-1}\mathbf{H}_{\mathbf{e}_{d}}(\R)\pi_{1}(t_{g}\gamma)\bigcap\G_{1}(\Z)\right|,
\]
and we note that we may finish the proof by verifying that 
\begin{equation}
\pi_{1}(t_{g}\gamma)^{-1}\mathbf{H}_{\mathbf{e}_{d}}(\R)\pi_{1}(t_{g}\gamma)=\mathbf{H}_{\ovec\left(I_{d}\cdot(t_{g}\gamma)\right)}(\R).\label{eq:conj of H_e_d by tau_g_gamma}
\end{equation}
To prove the latter equality we recall that the right $\SO_{Q}(\R)$
actions on $\SL_{d}(\R)$ and on $\R^{d}\smallsetminus\{\mathbf{0}\}$
are equivariant with respect to $\ovec:\SL_{d}(\R)\to\R^{d}\smallsetminus\{\mathbf{0}\}$
(to recall, see \eqref{eq:equivariance of the so_Q action w.r.t. tau}),
which shows that
\[
\mathbf{e}_{d}\cdot\pi_{1}(t_{g}\gamma)=\tau(I_{d}\cdot(t_{g}\gamma)),
\]
and which in turn implies \eqref{eq:conj of H_e_d by tau_g_gamma}.
\end{proof}
For $g_{\infty}\in\G(\R)$, $\eta\in H$ and $u\in\G(\Z)$ we note
that 
\[
\left|\mathbf{H}_{\ovec(I_{d}\cdot g_{\infty})}(\R)\cap\G_{1}(\Z)\right|=\left|\mathbf{H}_{\ovec(I_{d}\cdot(\eta g_{\infty}u))}(\R)\cap\G_{1}(\Z)\right|,
\]
and we define $\mathcal{E}\subseteq K\backslash G/\Gamma$ by
\begin{equation}
\mathcal{E}\df\left\{ \left(H\times\mathbf{L}_{\bar{g}}(\Z/(q))\right)(g_{\infty},g_{(q)})\G_{(q)}(\Z)\mid\left|\mathbf{H}_{\ovec(I_{d}\cdot g_{\infty})}(\R)\cap\G_{1}(\Z)\right|>1\right\} .\label{eq:the set E}
\end{equation}

\begin{lem}
\label{lem:equal weights}We denote $\alpha_{\text{max}}^{(n)}=\max_{h\in M_{0}}\{\alpha_{h}^{(n)}\}$,
where $\alpha_{h}^{(n)}$ are the weights of the atoms of $\eta_{g_{n},q}$
\emph{(}see \emph{\eqref{eq:rh_Dk_rh_infty_eta-1})}. Then there exists
$m>0$ such that $\frac{\alpha_{\text{max}}^{(n)}}{m}\leq\alpha_{h}^{(n)}\leq\alpha_{\text{max}}^{(n)}$,
$\forall n\in\N$. Moreover, for all $h\in M_{0}$ such that $\pi_{K}\left(O_{g_{n},q,h}\right)\notin\mathcal{E},$
it holds that $\alpha_{h}^{(n)}=\alpha_{\text{max}}^{(n)}$.
\end{lem}

\begin{proof}
It follows by Lemma \ref{lem:eqaulity of stabilizer of H=00005CtimesL_g and interal ortohoganl mat stab a vec}
and by \eqref{eq:measure of eta_g,S(O_g,S,h)} that
\begin{equation}
\frac{\alpha^{(n)}}{\left|\mathbf{H}_{\ovec\left(I_{d}\cdot(t_{g_{n}}\gamma)\right)}(\R)\cap\G_{1}(\Z)\right|}\leq\alpha_{h}^{(n)}\leq\alpha^{(n)}.\label{eq:sandwich alpha_h^(n)}
\end{equation}
We recall that 
\[
\pi_{K}\left(O_{g_{n},q,h}\right)=K\left(t_{g_{n}}\gamma,\red_{q}\left(c\right)\right)\Gamma,
\]
and we conclude by \eqref{eq:the set E} and \eqref{eq:sandwich alpha_h^(n)}
that
\[
\alpha_{h}^{(n)}=\alpha_{\text{max}}^{(n)}=\alpha^{(n)}\Longleftarrow\pi_{K}\left(O_{g_{n},q,h}\right)\notin\mathcal{E}.
\]
Finally, we show that $\left|\mathbf{H}_{\ovec\left(I_{d}\cdot(t_{g_{n}}\gamma)\right)}(\R)\cap\G_{1}(\Z)\right|$
is uniformly bounded from above. Indeed, since $\mathbf{H}_{\ovec(I_{d}\cdot(t_{g_{n}}\gamma))}(\R)$
is compact (being a conjugate of $\mathbf{H}_{\mathbf{e}_{d}}(\R)$,
which is compact by our \nameref{subsec:Standing-Assumption}), we
obtain that the subgroup $\mathbf{H}_{\ovec\left(I_{d}\cdot(t_{g_{n}}\gamma)\right)}(\R)\cap\G(\Z)\leq\GL_{d}(\Z)$
is finite. For a fixed $d\in\N$, the size of finite subgroups of
$\GL_{d}(\Z)$ is uniformly bounded (see for example \cite{Bound_on_order_of_fin_groups_friedland}),
which implies that there exists $m>0$ such that $\frac{\alpha_{\text{max}}^{(n)}}{m}\leq\alpha_{h}^{(n)}$.
\end{proof}
\begin{lem}
\label{lem:exeptional set of weights is rare}It holds that $\frac{\left|\mathcal{R}_{g_{n},q}\cap\mathcal{E}\right|}{|\mathcal{R}_{g_{n},q}|}\to0$
as \textbf{$n\to\infty$.}
\end{lem}

\begin{proof}
We claim that in order to prove $\lim_{n\to\infty}\frac{\left|\mathcal{R}_{g_{n},q}\cap\mathcal{E}\right|}{|\mathcal{R}_{g_{n},q}|}=0$,
it is sufficient to show that $\mathcal{E}\subseteq K\backslash G/\Gamma$
is closed and that
\begin{equation}
\mu_{K\backslash G/\Gamma}(\mathcal{E})=0.\label{eq:measure of eps}
\end{equation}
Indeed, by assuming the preceding limit, the proof will be complete
since

\begin{align*}
0= & \mu_{K\backslash G/\Gamma}(\mathcal{E})\underbrace{\geq}_{\text{Corollary \ref{cor:s-arithemetic measures converge to haar-1}}}\limsup_{n\to\infty}\eta_{g_{n},q}(\mathcal{E})\\
= & \limsup_{n\to\infty}\frac{\eta_{g_{n},q}(\mathcal{E}\cap R_{g_{n},q})}{\eta_{g_{n},q}\left(\mathcal{R}_{g_{n},q}\right)}\underbrace{\geq}_{\text{Lemma \ref{lem:equal weights}}}\limsup_{n\to\infty}\frac{\frac{\alpha_{\text{max}}^{(n)}}{m}\left|\mathcal{R}_{g_{n},q}\cap\mathcal{E}\right|}{\alpha_{\text{max}}^{(n)}|\mathcal{R}_{g_{n},q}|}.
\end{align*}

We will now proceed to prove \eqref{eq:measure of eps}. Consider
the natural projection
\[
\text{p}:\left(H\times\mathbf{L}_{\bar{g}}(\Z/(q))\right)\backslash\G(\R\times\Z/(q))/\G_{(q)}(\Z)\to\mathbf{H}_{\mathbf{e}_{d}}(\R)\backslash\G_{1}(\R)/\G_{1}(\Z),
\]
and note that 
\[
\text{p}(\mathcal{E})=\left\{ \mathbf{H}_{\mathbf{e}_{d}}(\R)\rho\G_{1}(\Z)\mid\rho^{-1}\mathbf{H}_{\mathbf{e}_{d}}(\R)\rho\cap\G_{1}(\Z)\neq\{e\}\right\} .
\]
We now recall some basic facts concerning orbifolds (we follow \cite{Riem_geom_orbif}).
Since $\mathbf{H}_{\mathbf{e}_{d}}(\R)$ is compact, it follows that
$\mathbf{H}_{\mathbf{e}_{d}}(\R)\backslash\G_{1}(\R)/\G_{1}(\Z)$
is an orbifold, and the set $\text{p}(\mathcal{E})$ is known as its
singular set (see \cite[Definition 25]{Riem_geom_orbif}). The singular
set is closed and has empty interior, see \cite[Proposition 26]{Riem_geom_orbif},
hence in particular $\mathcal{E}$ is closed (as a preimage of a closed
set). Now since $\mathbf{H}_{\mathbf{e}_{d}}(\R)$ is compact, it
is known that there exists a $\G_{1}(\R)$ right invariant Riemannian
metric on $\mathbf{H}_{\mathbf{e}_{d}}(\R)\backslash\G_{1}(\R)$.
Hence by \cite[Proposition 34]{Riem_geom_orbif}, the singular set
is locally the image of a union of finitely many sub-manifolds of
$\mathbf{H}_{\mathbf{e}_{d}}(\R)\backslash\G_{1}(\R)$ under the natural
quotient map. Therefore 
\[
\mu_{\mathbf{H}_{\mathbf{e}_{d}}(\R)\backslash\G_{1}(\R)/\G_{1}(\Z)}(\text{p}(\mathcal{E}))=0,
\]
which implies \eqref{eq:measure of eps}.
\end{proof}
\begin{lem}
\label{lem:the set E  contains set F are same} It holds that $\mathcal{F}\subseteq\mathcal{E}$,
where $\mathcal{F}\subseteq K\backslash G/\Gamma$ is given by \eqref{eq:definintio of set of points that are non-triv stab by Gam-1}.
\end{lem}

\begin{proof}
We recall that $\mathcal{F}$ is given by 
\[
\mathcal{F}=\left\{ K(g_{\infty},g_{(q)})\Gamma\mid\left|(g_{\infty},g_{(q)})^{-1}K(g_{\infty},g_{(q)})\cap\Gamma\right|>1\right\} .
\]
We let $K(g_{\infty},g_{(q)})\Gamma\in\mathcal{F},$ and upon recalling
the notations of $K$, $G$ and $\Gamma$ in \eqref{eq:notations of K G =00005CGamma},
we deduce that there exists $u\in\G(\Z)\smallsetminus\{e\}$ and $h_{\infty}\in H$
such that 
\[
g_{\infty}^{-1}h_{\infty}g_{\infty}=u.
\]
By recalling the definition of $H$ (see \eqref{eq:def_of_H}) we
obtain that 
\begin{equation}
\pi_{1}(g_{\infty}^{-1}h_{\infty}g_{\infty})=\pi_{1}(u)\in\G_{1}(\Z)\smallsetminus\{e\},\label{eq:pi_1(conj of h_infity) not trivial in G_1(Z)}
\end{equation}
where $\pi_{1}:\G\to\G_{1}$ is the natural projection. We have that
\[
\pi_{1}(g_{\infty}^{-1}h_{\infty}g_{\infty})=\pi_{1}(g_{\infty})^{-1}\pi_{1}(h_{\infty})\pi_{1}(g_{\infty}),
\]
and that $\pi_{1}(h_{\infty})\in\pi_{1}(H)=\mathbf{H}_{\mathbf{e}_{d}}(\R),$
which implies by \eqref{eq:pi_1(conj of h_infity) not trivial in G_1(Z)}
that 
\begin{equation}
\left|\pi_{1}(g_{\infty})^{-1}\mathbf{H}_{\mathbf{e}_{d}}(\R)\pi_{1}(g_{\infty})\cap\G_{1}(\Z)\right|>1.\label{eq:H_e_d conju intersect G_1(Z) non-trivialy}
\end{equation}
By \eqref{eq:H_e_d conju intersect G_1(Z) non-trivialy}, by observing
that
\[
\pi_{1}(g_{\infty})^{-1}\mathbf{H}_{\mathbf{e}_{d}}(\R)\pi_{1}(g_{\infty})=\mathbf{H}_{\mathbf{e}_{d}\cdot\pi_{1}(g_{\infty})}(\R)\underbrace{=}_{\eqref{eq:equivariance of the so_Q action w.r.t. tau}}\mathbf{H}_{\ovec(I_{d}\cdot g_{\infty})}(\R),
\]
and by recalling \eqref{eq:the set E} which defines $\mathcal{E}$,
we obtain that $K(g_{\infty},g_{(q)})\Gamma\in\mathcal{E}$.
\end{proof}
We now state the key corollary of this section, which verifies the
assumptions of Lemma \ref{lem:unwinding with negliglbe non-triv stab-1}
and finishes our proof of Theorem \ref{thm:refinement of the main theorem-1}.
\begin{cor}
\label{cor:weights are equal}It holds that 
\[
\lim_{n\to\infty}\frac{\left|\mathcal{F}\cap\mathcal{R}_{g_{n},q}\right|}{|\mathcal{R}_{g_{n},q}|}=0,
\]
and it holds that the sequence probability counting measures $\lambda_{g_{n},q}$
supported on $\mathcal{R}_{g_{n},q}$ for $n\in\N$ converges to $\mu_{K\backslash G/\Gamma}$.
\end{cor}

\begin{proof}
By Lemma \ref{lem:exeptional set of weights is rare} and Lemma \ref{lem:the set E  contains set F are same},
we deduce that $\lim_{n\to\infty}\frac{\left|\mathcal{F}\cap\mathcal{R}_{g_{n},q}\right|}{|\mathcal{R}_{g_{n},q}|}=0$.
By Corollary \ref{lem:equal weights}, Lemma \ref{lem:exeptional set of weights is rare},
we obtain 
\begin{equation}
\eta_{g_{n},q}-\lambda_{g_{n},q}\to0,\label{eq:difference between proj_nu and omega}
\end{equation}
and by \eqref{eq:equidisitrbution of omega modulo}, we deduce that
$\lambda_{g_{n},q}\to\mu_{K\backslash G/\Gamma}.$
\end{proof}

\section{\label{sec:Proof-of-theorems}Proof of theorems \ref{thm:moduli main thm }
and \ref{thm:moduli_main_thm_with_congruences} for $\protect\V$}

We let $Q$ be as in our \nameref{subsec:Standing-Assumption}. We
consider a sequence $\left\{ T_{n}\right\} _{n=1}^{\infty}\subseteq\N$
such that $T_{n}\to\infty$, and assume that there is an odd prime
$p_{0}$ for which it holds that $T_{n}$ has the $(Q,p_{0})$ co-isotropic
property for all $n\in\N$ (see Definition \ref{def:Isotropicity definition}).

For each $n\in\N$, let $g_{1,n},..,g_{m(n),n}\in\V_{T_{n}}(\Z)$
be a complete set of representatives for the equivalence relation
defined in Section \ref{sec:Equivalence-classes-of integral points},
namely 
\[
E_{g_{1},n}\bigsqcup...\bigsqcup E_{g_{m(n)},n}=\V_{T_{n}}(\Z).
\]
We claim that each of vector of the list $\ovec(g_{1,n}),..,\ovec(g_{m(n),n})$
is also $(Q,p_{0})$ co-isotropic (see Definition \ref{def:Isotropicity definition}).
Indeed, by Witt's theorem, the action of $\SO_{Q}(\Q)$ is transitive
on $\H_{T_{n}}(\Q)$, and if $\mathbf{v}\in\H_{T_{n}}(\Q)$ is $(Q,p)$
co-isotropic, then it follows that $\rho\mathbf{v}$ is $(Q,p)$ co-isotropic,
for $\rho\in\SO_{Q}(\Q)$.

We now fix an arbitrary sequence $\left\{ g_{j_{n},n}\right\} _{n=1}^{\infty}$
for $1\leq j_{n}\leq m(n)$, we fix $q\in2\N+1$ such that $Q$ is
non-singular modulo $q$ and we let $S=S_{q}\cup\{p_{0}\}$ where
$S_{q}$ is the set of primes appearing in the prime decomposition
of $q$.

\subsection{Proof of Theorem \ref{thm:main thm for Z}}

We partition the sequence $\left\{ g_{j_{n},n}\right\} _{n=1}^{\infty}$
into finitely many subsequences $\left\{ g_{j_{n},n}\right\} _{n\in C}$,
$C\subseteq\N$ such that for all $n\in C$ the reduction mod $q$
is fixed, say $\bar{g}\df\vartheta_{q}(g_{j_{n},n}),\ \forall n\in C$.
Then, we may apply Theorem \ref{thm:refinement of the main theorem-1}
to any of those unbounded subsequences.

We let $f\in C_{c}(\V_{Q(\mathbf{e}_{d})}(\R))$ and we consider $\tilde{f}\in C_{c}\left(\V_{Q(\mathbf{e}_{d})}(\R)\times\V_{a}(\Z/(q))\right)$
defined by $\tilde{f}(x,y)\df f(x)$, where $a\df Q(\bar{g})\in\Z/(q)$.
Then, in the notations of Theorem \ref{thm:refinement of the main theorem-1},
we have
\[
\lim_{C\ni n\to\infty}\nu_{g_{j_{n},n}}^{q}(\tilde{f})=\mu_{\V}(f),
\]
which implies in turn that for the full sequence (namely, without
the assumption that $\vartheta_{q}(g_{j_{n},n})$ is fixed in $n$)
it holds that
\begin{equation}
\lim_{n\to\infty}\nu_{g_{j_{n},n}}^{q}(\tilde{f})=\mu_{\V}(f).\label{eq:limit of arb seq of equiv rep (withou congruences)}
\end{equation}
We recall that 
\[
\nu_{g_{j,n}}^{q}=\frac{1}{\left|E_{g_{j,n}}/\G(\Z)\right|}\sum_{x\in E_{g_{j,n}}}\delta_{\left(\pi_{\V_{T_{n}}}(x),\vartheta_{q}(x)\right)},
\]
and that (see \eqref{eq:congruence counting measures on nu_Z_Q(E_d))})
\[
\nu_{T_{n}}^{\V,q}=\frac{1}{\left|\V_{T_{n}}(\Z)/\G(\Z)\right|}\sum_{x\in\V_{T_{n}}(\Z)}\delta_{\left(\pi_{\V_{T_{n}}}(x),\vartheta_{q}(x)\right)}.
\]
It follows that
\begin{equation}
\sum_{j=1}^{m(n)}\left(\sum_{x\in E_{g_{j,n}}}\delta_{\left(\pi_{\V_{T_{n}}}(x),\vartheta_{q}(x)\right)}\right)=\sum_{x\in\V_{T_{n}}(\Z)}\delta_{\left(\pi_{\V_{T_{n}}}(x),\vartheta_{q}(x)\right).},\label{eq:sum of counting func for rep (without cong)}
\end{equation}
and that
\begin{equation}
\sum_{j=1}^{m_{n}}\left|E_{g_{j_{n},n}}/\G(\Z)\right|=\left|\V_{T_{n}}(\Z)/\G(\Z)\right|=\left|\H_{T_{n},\text{prim}}(\Z)/\G_{1}(\Z)\right|.\label{eq:sum of equiv class g(z) orb}
\end{equation}
We now note the following elementary lemma (which we give without
a proof).
\begin{lem}
\label{lem:elementary lemma}Let $\left\{ a_{i,n}\right\} _{i=1,n=1}^{m_{n},\infty}\left\{ b_{i,n}\right\} _{i=1,n=1}^{m_{n},\infty}$
be positive real sequences. Assume $a_{i_{n},n}/b_{i_{n},n}\to L$,
for any sequence $\left\{ i_{n}\right\} _{n=1}^{\infty}$ such that
$i_{n}\in\{1,..,m_{n}\}$. Then $\frac{\sum_{i=1}^{m_{n}}a_{i,n}}{\sum_{i=1}^{m_{n}}b_{i,n}}\to L.$
\end{lem}

We may now deduce by Lemma \ref{lem:elementary lemma} and \eqref{eq:limit of arb seq of equiv rep (withou congruences)},
\eqref{eq:sum of counting func for rep (without cong)}, \eqref{eq:sum of equiv class g(z) orb}
that 
\[
\nu_{T_{n}}^{\V}(f)\underbrace{=}_{\text{recalling }\eqref{eq:counting mease on z}}\frac{1}{\left|\H_{T_{n},\text{prim}}(\Z)/\G_{1}(\Z)\right|}\sum_{x\in\V_{T_{n}}(\Z)}f(\pi_{\z_{T_{n}}}(x))=\nu_{T_{n}}^{\V,q}(\tilde{f})\to\mu_{\V}(f),
\]
which proves Theorem \ref{thm:main thm for Z}.

\subsection{Proof of Theorem \ref{thm:main_thm_with_congruences-forZ}}

We assume further that there is a fixed $a\in\left(\Z/(q)\right)^{\times}$
such that $\red_{q}\left(T_{n}\right)=a,\ \forall n\in\N$.

By Corollary \ref{cor:transitivity of ASL_times_SO_d}, \emph{\eqref{enu:mod_q_transitively}}
\[
\text{\ensuremath{\V}}_{a}(\Z/(q))=\vartheta_{q}(g_{j,n})\cdot\G(\Z/(q)),\ \forall n\in\N,\ \forall j\leq m(n)
\]
 Then, by using Theorem \ref{thm:refinement of the main theorem-1},
and following the same arguments as above, we obtain Theorem \ref{thm:main_thm_with_congruences-forZ}.

\appendix

\section{\label{sec:Unfolding}Unfolding}

In the following we let $G$ be locally compact second countable group,
$\Gamma\leq G$ be a lattice, $\tilde{\Gamma}\leq\Gamma$, and $K\leq G$
be a compact subgroup. We will discuss in this section a mechanism
which lifts an equidistribution result in $K\backslash G/\Gamma$
to an equidistribution result in $K\backslash G/\tilde{\Gamma}$ (see
Corollary \ref{cor:unwinding with negliglbe non-triv stab}).

Let $m_{G}$, $m_{G/\Gamma}$, $m_{G/\tilde{\Gamma}}$ be $G$-invariant
measures on $G,\ G/\Gamma,\ G/\tilde{\Gamma}$ respectively, such
that $m_{G/\Gamma}$ is a probability measure and such that all the
measures are Weil normalized (a notion introduced in Section \ref{subsec:Measures-as measures on fibre bundles}),
namely such that for all $\ef\in C_{c}(G)$
\begin{equation}
\int_{G}\ef(g)dm_{G}(g)=\int_{G/\Gamma}\left(\sum_{\gamma\in\Gamma}\ef(g\gamma)\right)dm_{G/\Gamma}(g\Gamma)=\int_{G/\tilde{\Gamma}}\left(\sum_{\tilde{\gamma}\in\tilde{\Gamma}}\ef(g\tilde{\gamma})\right)dm_{G/\tilde{\Gamma}}(g\tilde{\Gamma})\label{eq:unwinding of a measure on G}
\end{equation}
(such a normalization exists by Theorem 2.51 in \cite{Folland_harmonic}).
Let $f\in C_{c}(G/\tilde{\Gamma})$ and consider 
\begin{equation}
\bar{f}(x\Gamma)\df\sum_{\gamma\tilde{\Gamma}\in\Gamma/\tilde{\Gamma}}f(x\gamma\tilde{\Gamma}).\label{eq:def of f bar in appendix}
\end{equation}
We claim that $\bar{f}\in C_{c}(G/\Gamma)$. Indeed, by \cite[Proposition 2.50]{Folland_harmonic}
there exists $\ef\in C_{c}(G)$ such that 
\[
f(x\tilde{\Gamma})=\sum_{\tilde{\gamma}\in\tilde{\Gamma}}\ef(x\tilde{\gamma}),
\]
which shows that
\begin{equation}
\begin{aligned}\bar{f}(x\Gamma)= & \sum_{\gamma\tilde{\Gamma}\in\Gamma/\tilde{\Gamma}}f(x\gamma\tilde{\Gamma})\\
= & \sum_{\gamma\tilde{\Gamma}\in\Gamma/\tilde{\Gamma}}\sum_{\tilde{\gamma}\in\tilde{\Gamma}}\ef(x\gamma\tilde{\gamma})=\sum_{\gamma\in\Gamma}\ef(x\gamma),
\end{aligned}
\label{eq:f^bar equals sum ef x gamma}
\end{equation}
and we note that $\sum_{\gamma\in\Gamma}\ef(x\gamma)\in C_{c}(G/\Gamma)$.
\begin{lem}
\label{lem:unwinding G/Gamma TIlde}It holds that 
\begin{equation}
\int_{G/\tilde{\Gamma}}f(x\tilde{\Gamma})dm_{G/\tilde{\Gamma}}(g\tilde{\Gamma})=\int_{G/\Gamma}\bar{f}(x\Gamma)dm_{G/\Gamma}(g\Gamma),\label{eq:unwinding of measure on G/Gamma^tilde}
\end{equation}
for all $f\in C_{c}(G/\tilde{\Gamma})$
\end{lem}

\begin{proof}
Let $\ef\in C_{c}(G)$, and assume that $f(x\tilde{\Gamma})=\sum_{\tilde{\gamma}\in\tilde{\Gamma}}\ef(x\tilde{\gamma})$.
Then 
\[
\begin{aligned}\int_{G/\tilde{\Gamma}}f(x\tilde{\Gamma})dm_{G/\tilde{\Gamma}}(g\tilde{\Gamma})= & \int_{G/\tilde{\Gamma}}\left(\sum_{\tilde{\gamma}\in\tilde{\Gamma}}\ef(g\tilde{\gamma})\right)dm_{G/\tilde{\Gamma}}(g\tilde{\Gamma})\\
\underbrace{=}_{\eqref{eq:unwinding of a measure on G}} & \int_{G}\ef(g)dm_{G}(g)\\
\underbrace{=}_{\eqref{eq:unwinding of a measure on G}} & \int_{G/\Gamma}\left(\sum_{\gamma\in\Gamma}\ef(g\gamma)\right)dm_{G/\Gamma}(g\Gamma)\\
= & \int_{G/\Gamma}\left(\sum_{\gamma\tilde{\Gamma}\in\Gamma/\tilde{\Gamma}}\left(\sum_{\tilde{\gamma}\in\tilde{\Gamma}}\ef(g\gamma\tilde{\gamma})\right)\right)dm_{G/\Gamma}(g\Gamma)\\
= & \int_{G/\Gamma}\left(\sum_{\gamma\tilde{\Gamma}\in\Gamma/\tilde{\Gamma}}f(x\gamma\tilde{\Gamma})\right)dm_{G/\Gamma}(g\Gamma)\underbrace{=}_{\text{\eqref{eq:def of f bar in appendix}}}\int_{G/\Gamma}\bar{f}(x\Gamma)dm_{G/\Gamma}(g\Gamma).
\end{aligned}
\]
\end{proof}
We denote by $\pi_{K}$ the natural quotient map $\pi_{K}:G\to K\backslash G$.
We define a measure on $K\backslash G/\tilde{\Gamma}$ by $\mu_{K\backslash G/\tilde{\Gamma}}\df\left(\pi_{K}\right)_{*}m_{G/\tilde{\Gamma}},$
and on $K\backslash G/\Gamma$ by $\mu_{K\backslash G/\Gamma}\df\left(\pi_{K}\right)_{*}m_{G/\Gamma}$
(which is well defined, since we assume that $K$ is compact).
\begin{lem}
\label{lem:Unwinding on K=00005CG/Gamma^tilde } Assume that $S_{n}\subseteq K\backslash G/\Gamma$,
$n\in\N$, are finite sets such that the uniform probability measures
supported on $S_{n}$ converge weakly to $\mu_{K\backslash G/\Gamma}$.
Assume that $\left\{ Kg_{i,n}\tilde{\Gamma}\right\} \subseteq K\backslash G/\tilde{\Gamma}$
are representatives for $S_{n}$ (namely a choice of one point in
the preimage of $Kg_{i,n}\Gamma$ under the natural projection for
each $1\leq i\leq\left|S_{n}\right|$). Then for all $f\in C_{c}(K\backslash G/\tilde{\Gamma})$
it holds that
\[
\lim_{n\to\infty}\frac{1}{\left|S_{n}\right|}\sum_{i=1}^{\left|S_{n}\right|}\sum_{\gamma\tilde{\Gamma}\in\Gamma/\tilde{\Gamma}}f(Kg_{i,n}\gamma\tilde{\Gamma})=\mu_{K\backslash G/\tilde{\Gamma}}(f).
\]
\end{lem}

\begin{proof}
Let $f\in C_{c}(K\backslash G/\tilde{\Gamma})$, consider 
\[
\bar{f}(Kx\Gamma)\df\sum_{\gamma\tilde{\Gamma}\in\Gamma/\tilde{\Gamma}}f(Kx\gamma\tilde{\Gamma}),
\]
and note that $\bar{f}(Kx\Gamma)\in C_{c}(K\backslash G/\Gamma)$
(indeed, since $f\circ\pi_{K}\in C_{c}(G/\tilde{\Gamma}$), it follows
that $\bar{f}(Kx\Gamma)=\overline{f\circ\pi_{K}}(x\Gamma)\in C_{c}(G/\Gamma)$
by the discussion above Lemma \ref{lem:unwinding G/Gamma TIlde}).
 By the assumption of the lemma, we have that
\begin{equation}
\frac{1}{\left|S_{n}\right|}\sum_{i=1}^{\left|S_{n}\right|}\bar{f}(Kg_{i,n}\Gamma)\to\int_{K\backslash G/\Gamma}\overline{f}(Kg\Gamma)d\mu_{K\backslash G/\Gamma}.\label{eq:lmit of avarages on S_n}
\end{equation}
The proof is complete by observing that the left hand side of \eqref{eq:lmit of avarages on S_n}
may be rewritten by
\[
\frac{1}{\left|S_{n}\right|}\sum_{i=1}^{\left|S_{n}\right|}\bar{f}(Kg_{i,n}\Gamma)=\frac{1}{\left|S_{n}\right|}\sum_{i=1}^{\left|S_{n}\right|}\sum_{\gamma\tilde{\Gamma}\in\Gamma/\tilde{\Gamma}}f(Kg_{i,n}\gamma\tilde{\Gamma}),
\]
and the right hand side of \eqref{eq:lmit of avarages on S_n} may
be rewritten by
\[
\int_{K\backslash G/\Gamma}\overline{f}(Kg\Gamma)dm_{K\backslash G/\Gamma}=\int_{G/\Gamma}\sum_{\gamma\tilde{\Gamma}\in\Gamma/\tilde{\Gamma}}f\circ\pi_{K}(x\gamma\tilde{\Gamma})dm_{G/\Gamma}=
\]
\[
\underbrace{=}_{\eqref{eq:unwinding of measure on G/Gamma^tilde}}\int_{G/\tilde{\Gamma}}f\circ\pi_{K}(x\tilde{\Gamma})dm_{G/\tilde{\Gamma}}=\mu_{K\backslash G/\tilde{\Gamma}}(f).
\]
\end{proof}
Phrased differently, Lemma \ref{lem:Unwinding on K=00005CG/Gamma^tilde }
states that for the locally finite atomic measures 
\[
\nu_{n}\df\frac{1}{\left|S_{n}\right|}\sum_{i=1}^{\left|S_{n}\right|}\sum_{\gamma\tilde{\Gamma}\in\Gamma/\tilde{\Gamma}}\delta_{Kg_{i,n}\gamma\tilde{\Gamma}},
\]
it holds that $\nu_{n}(f)\to\mu_{K\backslash G/\tilde{\Gamma}}(f)$
for all $f\in C_{c}(K\backslash G/\tilde{\Gamma})$. We observe that
$\nu_{n}$ are not uniform measures, namely, some atoms can have different
weights. We let $\pi^{\tilde{\Gamma}}:G/\tilde{\Gamma}\to G/\Gamma$
be the natural map, and we note that the support of $\nu_{n}$ can
be expressed by 
\[
\text{supp}(\nu_{n})\df\left(\pi^{\tilde{\Gamma}}\right)^{-1}(S_{n}).
\]
We define $\bar{\nu}_{n}$ to be the uniform measures supported on
$\text{supp}(\nu_{n})$, namely 
\[
\bar{\nu}_{n}\df\frac{1}{|S_{n}|}\sum_{x\in\left(\pi^{\tilde{\Gamma}}\right)^{-1}(S_{n})}\delta_{x}.
\]
Similarly to Lemma\emph{ }\ref{lem:Unwinding on K=00005CG/Gamma^tilde },
we would like to show $\bar{\nu}_{n}(f)\to\mu_{K\backslash G/\tilde{\Gamma}}(f)$,\emph{
}for all $f\in C_{c}(K\backslash G/\tilde{\Gamma})$\emph{. }This
requires an additional assumption that the points which are counted
more then once are negligible. We define $\mathcal{F}\subseteq K\backslash G/\Gamma$
by 
\begin{equation}
\begin{aligned}\mathcal{F}\df & \left\{ Kg\Gamma\mid\left|\text{Stab}_{\Gamma}(Kg)\right|>1\right\} \\
= & \left\{ Kg\Gamma\mid\left|g^{-1}Kg\cap\Gamma\right|>1\right\} .
\end{aligned}
\label{eq:definintio of set of points that are non-triv stab by Gam_appendix}
\end{equation}

\begin{cor}
\label{cor:unwinding with negliglbe non-triv stab} Assume that $S_{n}\subseteq K\backslash G/\Gamma$,
$n\in\N$, are finite sets such that the uniform probability measures
supported on $S_{n}$ converge weakly to $\mu_{K\backslash G/\Gamma}$,
and assume that
\begin{equation}
\frac{|\mathcal{F}\cap S_{n}|}{\left|S_{n}\right|}\to0.\label{eq:negligability of non-trivialy stabilized points}
\end{equation}
Then it holds that $\bar{\nu}_{n}(f)\to\mu_{K\backslash G/\tilde{\Gamma}}(f)$,
for all $f\in C_{c}(K\backslash G/\tilde{\Gamma})$.
\end{cor}

We require the following basic lemma for the proof of Corollary \eqref{cor:unwinding with negliglbe non-triv stab}.
\begin{lem}
\label{lem:uniform bound intersection with cpct set}Let $U\subseteq K\backslash G/\tilde{\Gamma}$
be a set with compact closure. Then there exists $m_{U}>0$ such that
for all $g\in G$ it holds that
\begin{equation}
\left|\left\{ \gamma\tilde{\Gamma}\in\Gamma/\tilde{\Gamma}\mid Kg\gamma\tilde{\Gamma}\in U\right\} \right|\leq m_{U}.\label{eq:uniform bound for number of intersection w/ cpct}
\end{equation}
\end{lem}

\begin{proof}
Let $U\subseteq K\backslash G/\tilde{\Gamma}$ be a set with compact
closure. We let $\tilde{U}\subseteq G$ be a compact set such that
$\overline{U}=K\tilde{U}\tilde{\Gamma}$ (where $\overline{U}$ denotes
the closure of $U$), and we observe that 
\[
\left|\left\{ \gamma\tilde{\Gamma}\in\Gamma/\tilde{\Gamma}\mid Kg\gamma\tilde{\Gamma}\in U\right\} \right|=\left|\left\{ \gamma\tilde{\Gamma}\in\Gamma/\tilde{\Gamma}\mid Kg\gamma\tilde{\Gamma}\in K\tilde{U}\tilde{\Gamma}\right\} \right|\leq\left|\Gamma\cap g^{-1}K\tilde{U}\right|,
\]
for all $g\in G$. We recall that a lattice subgroup is uniformly
discrete, namely, there exists an open neighborhood of identity $\mathcal{N}$
such that $\left|u\mathcal{N}\cap\Gamma\right|\leq1,\ \forall u\in G$.
 Since $K\tilde{U}$ is compact, there exist $u_{1},...,u_{m_{U}}\in G$
such that $u_{1}\mathcal{N}\cup..\cup u_{m_{U}}\mathcal{N}\supseteq K\tilde{U}$.
This implies that $g^{-1}u_{1}\mathcal{N}\cup..\cup g^{-1}u_{m_{U}}\mathcal{N}\supseteq g^{-1}K\tilde{U}$.
Since there is at most one point of $\Gamma$ in each set $g^{-1}u_{i}\mathcal{N}$,
it follows that $\left|\Gamma\cap g^{-1}K\tilde{U}\right|\leq m_{U}$,
which implies \eqref{eq:uniform bound for number of intersection w/ cpct}.
\end{proof}
\begin{proof}[Proof of Corollary \ref{cor:unwinding with negliglbe non-triv stab}]
 We denote by $\left\{ Kg_{i,n}\tilde{\Gamma}\right\} _{i=1}^{|S_{n}|}\subseteq K\backslash G/\tilde{\Gamma}$
a set of representatives for $\left(\pi^{\tilde{\Gamma}}\right)^{-1}(S_{n})$
(a choice of a unique point in each fiber) and we fix a \emph{positive
}function $f\in C_{c}(K\backslash G/\tilde{\Gamma})$.

By noting that the weights of the atoms of $\nu_{n}$ are larger then
the weights of the atoms of $\bar{\nu}_{n},$ we find that
\[
\bar{\nu}_{n}(f)\leq\nu_{n}(f).
\]
We consider the uniform counting measure $\nu_{n}^{-}$ supported
on $\left(\pi^{\tilde{\Gamma}}\right)^{-1}(S_{n}\smallsetminus\mathcal{F})$
where each atom has mass $\frac{1}{|S_{n}|}$. We note that for all
$Kg_{i,n}\tilde{\Gamma}\in\left(\pi^{\tilde{\Gamma}}\right)^{-1}(S_{n}\smallsetminus\mathcal{F})$
and for any two distinct $\gamma_{1}\tilde{\Gamma},\gamma_{2}\tilde{\Gamma}\in\Gamma/\tilde{\Gamma}$
it holds that
\[
Kg_{i,n}\gamma_{1}\tilde{\Gamma}\neq Kg_{i,n}\gamma_{2}\tilde{\Gamma}.
\]
Namely, the weights of the atoms of $\nu_{n}^{-}$ and of $\bar{\nu}_{n}$
are the same on $\left(\pi^{\tilde{\Gamma}}\right)^{-1}(S_{n}\smallsetminus\mathcal{F})$,
which implies that
\[
\nu_{n}^{-}(f)\leq\bar{\nu}_{n}(f).
\]
We observe that 
\[
\nu_{n}(f)-\nu_{n}^{-}(f)=\frac{1}{\left|S_{n}\right|}\sum_{Kg_{i,n}\tilde{\Gamma}\in\left(\pi^{\tilde{\Gamma}}\right)^{-1}(\mathcal{F}\cap S_{n})}\sum_{\gamma\tilde{\Gamma}\in\Gamma/\tilde{\Gamma}}f(Kg_{i,n}\gamma\tilde{\Gamma}).
\]
We denote by $U$ the support of $f$, and we obtain by the triangle
inequality and by Lemma \ref{lem:uniform bound intersection with cpct set}
that 
\[
\nu_{n}(f)-\nu_{n}^{-}(f)\leq\norm f_{\infty}\frac{m_{U}}{\left|S_{n}\right|}\left|S_{N}\cap\mathcal{F}\right|\to0.
\]
Finally, since Lemma \ref{lem:Unwinding on K=00005CG/Gamma^tilde }
gives
\[
\lim_{n\to\infty}\nu_{n}(f)=\mu_{K\backslash G/\tilde{\Gamma}}(f),
\]
then we also get 
\[
\lim_{n\to\infty}\bar{\nu}_{n}(f)=\mu_{K\backslash G/\tilde{\Gamma}}(f).
\]
\end{proof}
\bibliographystyle{alpha}
\bibliography{cite_lib}

\end{document}